\pdfoutput=1

\documentclass[10pt]{article}

\usepackage{fullpage}

\usepackage{verbatim}

\usepackage{amsmath,amssymb,amsfonts}
\usepackage{amsthm}
\usepackage[titletoc,toc,title]{appendix}

\usepackage{array} 
\usepackage{listings}
\usepackage{mathtools}
\usepackage{bm}
\usepackage{tikz}
\usepackage[normalem]{ulem}
\usepackage{hhline}

\usepackage{graphicx}
\usepackage{subfig}
\usepackage{color}

\usepackage{pgfplots}
\usepackage{pgfplotstable}
\definecolor{markercolor}{RGB}{124.9, 255, 160.65}
\pgfplotsset{width=10cm,compat=1.3}
\pgfplotsset{
tick label style={font=\small},
label style={font=\small},
legend style={font=\small}
}

\usetikzlibrary{calc}

\newcommand{\logLogSlopeTriangle}[5]
{

    \pgfplotsextra
    {
        \pgfkeysgetvalue{/pgfplots/xmin}{\xmin}
        \pgfkeysgetvalue{/pgfplots/xmax}{\xmax}
        \pgfkeysgetvalue{/pgfplots/ymin}{\ymin}
        \pgfkeysgetvalue{/pgfplots/ymax}{\ymax}

        \pgfmathsetmacro{\xArel}{#1}
        \pgfmathsetmacro{\yArel}{#3}
        \pgfmathsetmacro{\xBrel}{#1-#2}
        \pgfmathsetmacro{\yBrel}{\yArel}
        \pgfmathsetmacro{\xCrel}{\xArel}

        \pgfmathsetmacro{\lnxB}{\xmin*(1-(#1-#2))+\xmax*(#1-#2)} 
        \pgfmathsetmacro{\lnxA}{\xmin*(1-#1)+\xmax*#1} 
        \pgfmathsetmacro{\lnyA}{\ymin*(1-#3)+\ymax*#3} 
        \pgfmathsetmacro{\lnyC}{\lnyA+#4*(\lnxA-\lnxB)}
        \pgfmathsetmacro{\yCrel}{\lnyC-\ymin)/(\ymax-\ymin)} 

        \coordinate (A) at (rel axis cs:\xArel,\yArel);
        \coordinate (B) at (rel axis cs:\xBrel,\yBrel);
        \coordinate (C) at (rel axis cs:\xCrel,\yCrel);

        \draw[#5]   (A)-- node[pos=0.5,anchor=north] {1}
                    (B)-- 
                    (C)-- node[pos=0.5,anchor=west] {#4}
                    cycle;
    }
}

\newcommand{\logLogSlopeTriangleFlip}[5]
{

    \pgfplotsextra
    {
        \pgfkeysgetvalue{/pgfplots/xmin}{\xmin}
        \pgfkeysgetvalue{/pgfplots/xmax}{\xmax}
        \pgfkeysgetvalue{/pgfplots/ymin}{\ymin}
        \pgfkeysgetvalue{/pgfplots/ymax}{\ymax}

        \pgfmathsetmacro{\xBrel}{#1-#2}
        \pgfmathsetmacro{\yBrel}{#3}
        \pgfmathsetmacro{\xCrel}{#1}

        \pgfmathsetmacro{\lnxB}{\xmin*(1-(#1-#2))+\xmax*(#1-#2)} 
        \pgfmathsetmacro{\lnxA}{\xmin*(1-#1)+\xmax*#1} 
        \pgfmathsetmacro{\lnyA}{\ymin*(1-#3)+\ymax*#3} 
        \pgfmathsetmacro{\lnyC}{\lnyA+#4*(\lnxA-\lnxB)}
        \pgfmathsetmacro{\yCrel}{\lnyC-\ymin)/(\ymax-\ymin)} 

	\pgfmathsetmacro{\xArel}{\xBrel}
        \pgfmathsetmacro{\yArel}{\yCrel}

        \coordinate (A) at (rel axis cs:\xArel,\yArel);
        \coordinate (B) at (rel axis cs:\xBrel,\yBrel);
        \coordinate (C) at (rel axis cs:\xCrel,\yCrel);

        \draw[#5]   (A)-- node[pos=0.5,anchor=east] {#4}
                    (B)-- 
                    (C)-- node[pos=0.5,anchor=south] {1}
                    cycle;
    }
}

\usepackage{stmaryrd}

\newcommand{\tensor}[1]{\underline{\bm{#1}}}

\newcommand{\td}[2]{\frac{{\rm d}#1}{{\rm d}{\rm #2}}}
\newcommand{\pd}[2]{\frac{\partial#1}{\partial#2}}

\newcommand{\nor}[1]{\left\| #1 \right\|}
\newcommand{\LRp}[1]{\left( #1 \right)}

\newcommand{\LRa}[1]{\left\langle #1 \right\rangle}
\newcommand{\LRb}[1]{\left| #1 \right|}
\newcommand{\LRc}[1]{\left\{ #1 \right\}}

\newcommand{\LRl}[1]{\left. #1 \right|}

\newcommand{\Grad} {\ensuremath{\nabla}}
\newcommand{\Div} {\ensuremath{\nabla\cdot}}
\newcommand{\jump}[1] {\ensuremath{\llbracket#1\rrbracket}}
\newcommand{\avg}[1] {\ensuremath{\LRc{\!\{#1\}\!}}}

\newcommand{\Oh}{{\Omega_h}}
\renewcommand{\L}{L^2\LRp{\Omega}}
\newcommand{\Lk}{L^2\LRp{D^k}}
\newcommand{\Ldk}{L^2\LRp{\partial D^k}}
\newcommand{\Dhat}{\widehat{D}}

\newtheorem{theorem}{Theorem}[section]
\newtheorem{lemma}[theorem]{Lemma}

\newcommand{\eval}[2][\right]{\relax
  \ifx#1\right\relax \left.\fi#2#1\rvert}

\newcommand{\reviewerOne}[1]{{\color{black}#1}}
\newcommand{\reviewerTwo}[1]{{\color{black}#1}}
\newcommand{\reviewerThree}[1]{{\color{black}#1}}

\newcommand{\LinfDk}{L^{\infty}\LRp{D^k}}

\newcolumntype{C}[1]{>{\centering\let\newline\\\arraybackslash\hspace{0pt}}m{#1}}

\newcommand*\diff[1]{\mathop{}\!{\mathrm{d}#1}}

\makeatletter
\renewcommand\d[1]{\mspace{6mu}\mathrm{d}#1\@ifnextchar\d{\mspace{-3mu}}{}}
\makeatother

\date{}
\author{Jesse Chan}
\title{Weight-adjusted discontinuous Galerkin methods: matrix-valued weights and elastic wave propagation in heterogeneous media}

\begin{document}


\maketitle

\begin{abstract}
Weight-adjusted inner products \cite{chan2016weight1,chan2016weight2} are easily invertible approximations to weighted $L^2$ inner products.  These approximations can be paired with a discontinuous Galerkin (DG) discretization to produce a time-domain method for wave propagation which is low storage, energy stable, and high order accurate for arbitrary heterogeneous media and curvilinear meshes.  In this work, we extend weight-adjusted DG (WADG) methods to the case of matrix-valued weights, with the linear elastic wave equation as an application.  We present a DG formulation of the symmetric form of the linear elastic wave equation, with upwind-like dissipation incorporated through simple penalty fluxes.  A semi-discrete convergence analysis is given, and numerical results confirm the stability and high order accuracy of WADG for several problems in elastic wave propagation.  
\end{abstract}


\section{Introduction}

Efficient and accurate methods for elastic wave propagation form a foundation for a broad range of applications, from seismic and medical imaging to rupture and earthquake simulation.  Finite differences are the most common choice of method \cite{virieux1986p}; however, finite element methods have garnered interest due to their low numerical dispersion and ability to accommodate geometrically flexible unstructured meshes.  

Typical methods for time-domain wave propagation utilize explicit time stepping, since the hyperbolic partial differential equations (PDEs) which govern wave propagation admit a reasonable stable time-step restriction.  However, \reviewerOne{unless special techniques (such as diagonal mass lumping) are applied,} finite element methods require the inversion of a global mass matrix when paired with explicit time integrators.  Spectral element methods (SEM) sidestep this issue on hexahedral meshes by \reviewerOne{choosing nodal basis functions which are discretely orthogonal with respect to an underintegrated $L^2$ inner product, which produces a} diagonal mass matrix \cite{komatitsch1998spectral}.  The inversion of a globally coupled matrix can also be avoided through the use of discontinuous Galerkin (DG) methods, which result in a locally invertible block diagonal mass matrices.  Due to difficulties in extending mass-lumping techniques from hexahedra to tetrahedra, high order DG methods are often employed for seismic simulations which require the use of simplicial meshes \cite{kaser2006arbitrary, dumbser2006arbitrary, de2007arbitrary,delcourte2009high,delcourte2015analysis,ye2016discontinuous}.  High order DG methods also lend themselves well to efficient implementations using Graphics Processing Units (GPUs) \cite{klockner2009nodal, modave2015nodal, modave2016gpu, chan2015gpu}.  

Most high order DG methods on simplicial meshes assume that models of media and material coefficients are constant over each element, \reviewerOne{which allows them to deal with discontinuous wave speeds across element interfaces.  However, if the media is such that material gradients are non-zero in the interior of an element, piecewise constant approximations can yield inaccurate simulations of wave propagation \cite{castro2010seismic, mercerat2015nodal, bencomo2015discontinuous}.  }
 This limitation can be overcome by incorporating sub-element heterogeneities into weighted mass matrices, resulting in a DG method which is both high order accurate and energy stable \cite{mercerat2015nodal, bencomo2015discontinuous}.  On tetrahedral meshes, this approach necessitates the pre-computation and storage of factorizations or inverses for each local mass matrix, which greatly increases both storage costs and data transferred at high orders of approximation.  These costs are especially problematic for accelerator architectures such as GPUs, which possess limited memory.  

Storage costs associated with weighted mass matrices can be avoided by approximating weighted $L^2$ inner products using weight-adjusted inner products, which result in easily invertible approximations to weighted mass matrices \cite{chan2016weight1,chan2016weight2}.  For sufficiently regular weights, high order accuracy is also retained.  When paired with an energy stable DG formulation, these approximations result in weight-adjusted DG methods (WADG), which preserve energy stability and high order accuracy while retaining a low asymptotic storage cost.  Additionally, unlike mass-lumping techniques, WADG methods do not rely on the use of inexact quadrature rules, and reduce to the exact inversion of mass matrices for constant weights.  

Weight-adjusted DG methods have been applied to acoustic wave propagation in heterogeneous media and on curvilinear meshes \cite{chan2016weight1,chan2016weight2}.  Both of these previous applications have involved scalar weighting functions.  In this work, we extend weight-adjusted inner products to matrix-valued weights.  \reviewerOne{This provides a way to approximate the inverse of a block system of mass matrices which are coupled together by a spatially varying matrix-valued weighting function.  This approximate inverse involves the application of scalar mass matrix inverses and the matrix-free application of a system of weighted block mass matrices using quadrature.  We show that this approach reduces storage and computational costs compared to the storage of inverses or factorizations of the full block mass matrix system, and apply this approximation to derive a stable and high order accurate method for elastic wave propagation in arbitrary heterogeneous media.  This method is based on an} energy stable DG formulation of the symmetric form of the elastic wave equations, with upwind-like numerical dissipation introduced through simple penalty fluxes \cite{warburton2013low}.  In contrast to the fluxes proposed in \cite{ye2016discontinuous}, the penalty fluxes used here can be made to be independent of material coefficients.  

This work proceeds as follows: Sections~\ref{sec:symelas} and \ref{sec:dgform} present an energy stable DG formulation with simple penalty fluxes for the symmetric hyperbolic form of the elastic wave equation, and discuss issues related to storage and inversion of local mass matrices for material coefficients with sub-element variations.  Section~\ref{sec:mwadg} extends weight-adjusted approximations to weighted $L^2$ inner products and mass matrices to the case of matrix-valued weights, and provides interpolation estimates which account for the regularity of the matrix weight.  These results are incorporated into a weight-adjusted DG method for the linear elastic wave equations in Section~\ref{sec:wadgelas}.  Finally, numerical results in Section~\ref{sec:numerical} demonstrate the accuracy of this method for several problems in linear elasticity.  

\section{Symmetric form of the elastic wave equation}
\label{sec:symelas}
We begin with the linear elastic wave equation in a domain $\Omega \in \mathbb{R}^d$.  These equations can be written as a first order velocity-stress system for velocity $\bm{v}$ and symmetric stress tensor $\tilde{\bm{S}}$ 
\begin{align*}
\rho \pd{\bm{v}}{t} &= \Div{\tilde{\bm{S}}} + \bm{f}\\
\pd{\tilde{\bm{S}}}{t} &= \frac{1}{2} \tensor{\bm{C}}\LRp{\Grad\bm{v} + \Grad\bm{v}^T},
\end{align*}
where $\bm{f}$ is the body force per unit volume, $\rho$ is density, and $\tensor{\bm{C}}$ is the symmetric constitutive stiffness tensor relating stress and strain.  We rewrite these equations as a symmetric hyperbolic system of PDEs \cite{hughes1978classical} using Voigt notation
\begin{align}
\rho \pd{\bm{v}}{t} &= \sum_{i=1}^d \bm{A}_i^T \pd{\bm{\sigma}}{\bm{x}_i}\nonumber + \bm{f}\\
\bm{C}^{-1} \pd{\bm{\sigma}}{t} &= \sum_{i=1}^d \bm{A}_i \pd{\bm{v}}{\bm{x}_i}\reviewerOne{,}
\label{eq:symelas}
\end{align}
where $\bm{C}$ is the symmetric matrix form of the constitutive tensor $\tensor{\bm{C}}$ and $\bm{\sigma}$ is a vector of length $N_d = \frac{d(d+1)}{2}$, the number of unique entries of the stress tensor $\tilde{\bm{S}}$ in $d$ dimensions.  We note that the matrices $\bm{A}_i$ are spatially constant, while $\rho,\bm{C}$, and $\bm{C}^{-1}$ can vary spatially.  Furthermore, we will assume that $\rho$ and $\bm{C}$ are positive-definite and bounded pointwise such that 
\begin{align*}
0 &< \rho_{\min} \leq \rho(\bm{x}) \leq \rho_{\max} < \infty, \\
0 &< c_{\min} \leq \bm{u}^T\bm{C}(\bm{x})\bm{u} \leq c_{\max} < \infty\\
0 &< \tilde{c}_{\min} \leq \bm{u}^T\bm{C}^{-1}(\bm{x})\bm{u} \leq \tilde{c}_{\max} < \infty
\end{align*}
for all $\bm{x}\in \mathbb{R}^d$ and all $\bm{u} \in \mathbb{R}^{N_d}$.  

In two dimensions, $\bm{v} = (\bm{v}_1, \bm{v}_2)^T$ and $\bm{\sigma}= (\sigma_{xx},\sigma_{yy},\sigma_{xy})^T$ 
\[
\tilde{\bm{S}} = \LRp{\begin{array}{cc}
\sigma_{xx} & \sigma_{xy}\\
\sigma_{xy} & \sigma_{yy}
\end{array}},
\]
while the matrices $\bm{A}_i$ are 
\[
\bm{A}_1 = \left(\begin{array}{cc}
1 & 0\\
0 & 0\\
0 & 1
\end{array}\right), \qquad 
\bm{A}_2 = \left(\begin{array}{cc}
0 & 0\\
0 & 1\\
1 & 0
\end{array}\right).
\]
In three dimensions, the velocity is $\bm{v} = (\bm{v}_1,\bm{v}_2,\bm{v}_3)^T$, while $\bm{\sigma}= (\sigma_{xx},\sigma_{yy},\sigma_{zz},\sigma_{yz},\sigma_{xz},\sigma_{xy})^T$ denotes the unique entries of the stress tensor $\tilde{\bm{S}}$
\[
\tilde{\bm{S}} = \LRp{\begin{array}{ccc}
\sigma_{xx} & \sigma_{xy} & \sigma_{xz}\\
\sigma_{xy} & \sigma_{yy}& \sigma_{yz}\\
\sigma_{xz} & \sigma_{yz}& \sigma_{zz}
\end{array}}.
\]
The matrices $\bm{A}_i$ are then
\[
\bm{A}_1 = \left(\begin{array}{ccc}
1 & 0 & 0\\
0 & 0& 0\\
0 & 0& 0\\
0 & 0& 0\\
0 & 0 & 1\\
0 & 1 & 0
\end{array}\right), \qquad 
\bm{A}_2 = \left(\begin{array}{ccc}
0 & 0 & 0\\
0 & 1 & 0\\
0 & 0 & 0\\
0 & 0 & 1\\
0 & 0 & 0\\
1 & 0 & 0
\end{array}\right), \qquad
\bm{A}_3 = \left(\begin{array}{ccc}
0 & 0 & 0\\
0 & 0 & 0\\
0 & 0 & 1\\
0 & 1 & 0\\
1 & 0 & 0\\
0 & 0 & 0
\end{array}\right)
\]

In general anisotropic media, $\bm{C}$ is symmetric and positive-definite.  For two-dimensional isotropic media, $\bm{C}$ \reviewerOne{and its inverse are} given as
\[
\bm{C}= \LRp{
\begin{array}{ccc}
2\mu+\lambda  &  \lambda  &     0\\
\lambda &       2\mu+\lambda &      0\\
0 & 0 & \mu
\end{array}}, \qquad 
\reviewerOne{\bm{C}^{-1} = \frac{1}{4\mu^2+\mu\lambda}\LRp{
\begin{array}{ccc}
2\mu+\lambda  &  -\lambda  &     0\\
-\lambda &       2\mu+\lambda &      0\\
0 & 0 & \frac{4\mu^2+\mu\lambda}{\mu}
\end{array}},} 
\]
where $\lambda,\mu$ are Lame parameters.  For three-dimensional isotropic media, $\bm{C}$ \reviewerOne{and its inverse are} given instead by
\[
\bm{C} = \LRp{\begin{array}{cccc}
2\mu+\lambda  &  \lambda  &     \lambda   & \\
\lambda &   2\mu+\lambda & \lambda   &  \\
\lambda & \lambda &   2\mu+\lambda &    \\
 &   &    & {\mu}\bm{I}^{3\times3} \\
\end{array}}, \qquad 
\reviewerOne{\bm{C}^{-1} = \frac{1}{2\mu+3\lambda}\LRp{\begin{array}{cccc}
\mu+\lambda  &  -\lambda/2  &     -\lambda/2   & \\
-\lambda/2 &   \mu+\lambda & -\lambda/2   &  \\
-\lambda/2 & -\lambda/2 &   \mu+\lambda &    \\
 &   &    & \frac{2\mu+3\lambda}{\mu}\bm{I}^{3\times3} \\
\end{array}}}
\]
We will consider both spatially varying isotropic and anisotropic media in this work.  

\section{An energy stable discontinuous Galerkin formulation for elastic wave propagation}
\label{sec:dgform}

\reviewerThree{Energy stable discontinuous Galerkin methods have been constructed based on non-symmetric formulations of the elastodynamics equations \cite{wilcox2010high}.  However, it is also straightforward to derive an energy stable discontinuous Galerkin formulation based on the symmetric first order formulation of the elastic wave equations (\ref{eq:symelas}).}  We assume that the domain $\Omega$ is Lipschitz and exactly triangulated by a mesh $\Omega_h$, which consists of elements $D^k$.  We further assume that each element $D^k$ is the image of a reference element $\widehat{D}$ under the local elemental mapping 
\[
\bm{x}^k = \bm{\Phi}^k \widehat{\bm{x}},
\]
where $\bm{x}^k = \LRc{x^k,y^k}$ for $d=2$ and $\bm{x}^k = \LRc{x^k,y^k,z^k}$ for $d=3$ denote physical coordinates on $D^k$ and $\widehat{\bm{x}} = \LRc{\widehat{x},\widehat{y}}$  for $d = 2$ and $\widehat{\bm{x}} = \LRc{\widehat{x},\widehat{y},\widehat{z}}$ for $d = 3$ denote coordinates on the reference element.  We denote the determinant of the Jacobian of $\bm{\Phi}^k$ as $J$, and refer to it as the Jacobian for the remainder of this work.  

We will approximate solution components over each element $D^k$ from an approximation space $V_h\LRp{D^k}$, which we define as the composition of the mapping $\bm{\Phi}^k$ and a reference approximation space $V_h\LRp{\widehat{D}}$
\[
V_h\LRp{D^k} = \bm{\Phi}^k \circ V_h\LRp{\widehat{D}}.
\]
The global approximation space $V_h\LRp{\Oh}$ is then defined as the direct sum of elemental approximation spaces
\[
V_h\LRp{\Oh} = \bigoplus_{D^k}V_h\LRp{D^k}.  
\]
For the remainder of this work, we will take $V_h\LRp{\widehat{D}} = P^N\LRp{\widehat{D}}$, where $P^N\LRp{\widehat{D}}$ is the polynomial space of total degree $N$ on the reference simplex.  In two dimensions, $P^N$ on a triangle is
\[
P^N\LRp{\widehat{D}} = \LRc{ \widehat{x}^i \widehat{y}^j, \quad 0 \leq i + j \leq N},
\] 
and in three dimensions, $P^N$ on a tetrahedron is
\[
P^N\LRp{\widehat{D}} = \LRc{ \widehat{x}^i \widehat{y}^j \widehat{z}^k, \quad 0 \leq i + j +k \leq N}.  
\]
We denote the $L^2$ inner product and norm over $D^k$ by $\LRp{\cdot,\cdot}_{\Lk}$, such that
\[
\reviewerOne{\LRp{\bm{g},\bm{h}}_{\Lk} = \int_{D^k} \bm{g}\cdot\bm{h} \diff{\bm{x}} = \int_{\widehat{D}} \bm{g}\cdot\bm{h} J\diff{\widehat{\bm{x}}}, \qquad \nor{\bm{g}}_{\Lk}^2 = \LRp{\bm{g},\bm{g}}_{\Lk},}
\]
where \reviewerOne{$\bm{g},\bm{h}$} are real vector-valued functions.  Global $L^2$ inner products and norms are using local $L^2$ inner products and norms
\[
\reviewerOne{\LRp{\bm{g},\bm{h}}_{\L} = \sum_{D^k\in \Oh} \LRp{\bm{g},\bm{h}}_{\Lk}, \qquad \nor{\bm{g}}_{\L}^2 = \sum_{D^k\in \Oh} \nor{\bm{h}}^2_{\Lk}.}
\]
We define also the $L^2$ inner product and norm over the boundary $\partial D^k$ of an element
\[
\LRa{\bm{u},\bm{v}}_{\Ldk} = \int_{\partial D^k}\bm{u}\cdot\bm{v}\diff{\bm{x}} = \sum_{f\in \partial D^k}\int_{\widehat{f}}\bm{u}\cdot\bm{v}J^f\diff{\widehat{\bm{x}}}, \qquad \nor{\bm{u}}_{\Ldk}^2 = \LRa{\bm{u},\bm{u}}_{\Ldk},
\]
where $J^f$ is the Jacobian of the mapping from a reference face $\widehat{f}$ to a physical face $f$ of an element.  

Let $f$ be a face of an element $D^{k}$ with neighboring element $D^{k,+}$ and unit outward normal $\bm{n}$.  Let $u$ be a function which is discontinuous across element interfaces.  \reviewerOne{We define the interior value  $u^-$ and exterior value $u^+$  on a face $f$ of $D^k$ such that}
\[
{u}^- = \LRl{u}_{f \cap \partial D^k}, \qquad {u}^+ = \LRl{u}_{f \cap  \partial D^{k,+}}.  
\]
The jump and average of a scalar function $u\in V_h\LRp{\Omega_h}$ over $f$ are then defined as
\[
\jump{u} = u^+ - u^-, \qquad \avg{u} = \frac{u^+ + u^-}{2}.
\]
Jumps and averages of \reviewerTwo{vector-valued functions} $\bm{u}\in \mathbb{R}^m$ and $\tilde{\bm{S}}\in \mathbb{R}^{m\times n}$ are then defined component-wise 
\[
\LRp{\jump{\bm{u}}}_i = \jump{\bm{u}_i}, \qquad 1\leq i \leq m, \qquad \LRp{\jump{\tilde{\bm{S}}}}_{ij} = \jump{\tilde{\bm{S}}_{ij}}, \qquad 1\leq i \leq m, \quad 1\leq j \leq n.
\]

We can now specify a DG formulation for the linear elastic wave equation (\ref{eq:symelas}).  Symmetric hyperbolic systems readily admit a DG formulation based on penalty fluxes \cite{chan2016short}.  For the linear elastic wave equation in symmetric first order form, this formulation is given as
\begin{align}
&\sum_{D^k\in \Oh} \LRp{\rho \pd{\bm{v}}{t},\bm{w}}_{\Lk} = \sum_{D^k\in \Oh} \LRp{\LRp{ \sum_{i=1}^d \bm{A}_i^T\pd{\bm{\sigma}}{\bm{x}_i} + \bm{f},\bm{w}}_{\Lk}  + \LRa{\frac{1}{2}\bm{A}_n^T\jump{\bm{\sigma}} + \frac{\tau_{\bm{v}}}{2}\bm{A}_n^T\bm{A}_n\jump{\bm{v}},\bm{w}}_{\Ldk}} \nonumber\\
&\sum_{D^k\in \Oh} \LRp{\bm{C}^{-1} \pd{\bm{\sigma}}{t},\bm{q}}_{\Lk} = \sum_{D^k\in \Oh}\LRp{\LRp{\sum_{i=1}^d \bm{A}_i \pd{\bm{v}}{\bm{x}_i},\bm{q}}_{\Lk} + \LRa{\frac{1}{2}\bm{A}_n\jump{\bm{v}} + \frac{\tau_{\sigma}}{2}\bm{A}_n\bm{A}_n^T\jump{\bm{\sigma}},\bm{q}}_{\Ldk}},
\label{eq:dgform}
\end{align}
for all $\bm{w},\bm{q}\in V_h\LRp{\Omega_h}$.  Here, $\bm{A}_n$ is the normal matrix defined on a face $f$ as $\bm{A}_n = \sum_{i=1}^d \bm{n}_i \bm{A}_i$.  In two dimensions, $\bm{A}_n$ is 
\[
\bm{A}_n =  \LRp{\begin{array}{cc}
{n}_x & 0\\
 0 & {n}_y\\
{n}_y & {n}_x
\end{array}}.
\]
while in three dimensions, $\bm{A}_n$ is
\[
\bm{A}_n =  \LRp{\begin{array}{ccc}
{n}_x & 0 & 0\\
 0 & {n}_y & 0\\
0&  0 & {n}_z \\
0 &{n}_z & {n}_y\\
{n}_z & 0 & {n}_x\\
{n}_y & {n}_x & 0\\
\end{array}}.  
\]
\reviewerOne{
The terms $\tau_{\bm{v}},\tau_{\bm{\sigma}}$ are penalty parameters which are introduced on element interfaces.  We assume that $\tau_{\bm{v}},\tau_{\bm{\sigma}} \geq 0$ and that they are piecewise constant over each shared face between two elements.  These penalty parameters can be taken to be zero, which corresponds to a DG method using a non-dissipative central flux \cite{fezoui2005convergence, mercerat2015nodal}.  In Section~\ref{sec:energystability}, we show that when $\tau_{\bm{v}},\tau_{\bm{\sigma}}$ are positive, they introduce a dissipation of energy in a manner which is similar to the upwind flux \cite{hesthaven2007nodal, wilcox2010high}.  We note that the stability of the DG formulation is independent of the magnitude of these parameters; however, as discussed in Section~\ref{sec:scaling}, naively choosing the values of these parameters can result in a stiffer semi-discrete system of ODEs and a smaller maximum stable timestep.  

In many applications, $\bm{f}$ is a point source or Dirac delta, which is not $L^2$ integrable.  Thus, $\LRp{\bm{f},\bm{w}}_{\Lk}$ may not be well-defined.  In such cases when $\bm{f}(\bm{x}) = \bm{\beta}(\bm{x}) \delta(\bm{x}-\bm{x}_0)$ (for some vector $\bm{\beta}(\bm{x}) \in \mathbb{R}^d$), we commit a variational crime and evaluate its contribution as 
\[
\sum_k \LRp{\bm{\beta(\bm{x})}\delta(\bm{x}-\bm{x}_0),\bm{w}}_{\Lk} = \int_{\Omega} \bm{w} \cdot \bm{\beta} \delta(\bm{x}-\bm{x}_0)
 =  \bm{w}(\bm{x}_0)\cdot \bm{\beta(\bm{x}_0)}.  
\]
}

Finally, we note that, unlike the penalty DG formulation given in \cite{ye2016discontinuous}, material coefficients $\rho,\bm{C}$ appear only on the left hand side of (\ref{eq:dgform}).  Thus, efficient techniques for constant coefficient formulations \cite{chan2015bbdg} can be used to evaluate the right hand side of the formulation, even in the presence of sub-element variations in $\rho,\bm{C}$.  

\subsection{Boundary conditions}

In this work, we assume boundary conditions on velocity and traction of the form
\[
\bm{v} = \bm{v}_{\rm bc}, \qquad \tilde{\bm{S}}\bm{n} = \bm{t}_{\rm bc}
\]
where $\bm{v}_{\rm bc}$ and $\bm{t}_{\rm bc}$ are given values.  Traction boundary conditions where $\bm{t}_{\rm bc} = 0$ are referred to as free-surface boundary conditions.  We follow \cite{leveque2002finite,wilcox2010high} and impose boundary conditions on the DG formulation through exterior values and jumps of the solution.  Boundary conditions on the normal component of the stress can be imposed by noting that the numerical flux contains the term $\jump{\bm{A}_n^T\bm{\sigma}} = \jump{\tilde{\bm{S}}\bm{n}}$.  

For a face which lies on a boundary, velocity boundary conditions are imposed by setting 
\begin{align*}
\jump{\bm{v}} &= 2\LRp{\bm{v}_{\rm bc} - \bm{v}^- }, \qquad \jump{\bm{A}_n^T\bm{\sigma}}=\jump{\tilde{\bm{S}}\bm{n}} = 0, 
\end{align*}
while traction boundary conditions are enforced through
\begin{align*}
\jump{\bm{A}_n^T\bm{\sigma}} = \jump{\tilde{\bm{S}}\bm{n}} = 2\LRp{\bm{t}_{\rm bc} - \tilde{\bm{S}}^-\bm{n} }= 2\LRp{\bm{t}_{\rm bc} - \bm{A}_n^T\bm{\sigma}^- }, \qquad \jump{\bm{v}} = 0.
\end{align*}
For problems which involve the truncation of infinite or large domains, absorbing boundary conditions are required.  For such cases, we impose simple \reviewerOne{extrapolation} absorbing boundary conditions \cite{leveque2002finite} through jumps 
\begin{align*}
\jump{\bm{A}_n^T\bm{\sigma}} = \jump{\tilde{\bm{S}}\bm{n}} = - \tilde{\bm{S}}^-\bm{n} = -\bm{A}_n^T\bm{\sigma}^- , \qquad \jump{\bm{v}} = - \bm{v}^-. 
\end{align*}
We note that more accurate absorbing conditions can be imposed using, for example, perfectly matched layers \cite{berenger1994perfectly} or high order absorbing boundary conditions \cite{hagstrom2004new,modave2016high}.  

In all cases, boundary conditions are imposed by computing numerical fluxes using these modified jumps.  This imposition guarantees energy stability for free surface, non-reflective, and homogeneous velocity boundary conditions.  


\subsection{Energy stability}
\label{sec:energystability}

\reviewerOne{One can show that the DG formulation is energy stable for zero body load, zero velocity and traction boundary conditions, and non-reflective boundary conditions.  We note that this stability holds for both the case when $\tau_{\bm{v}}, \tau_{\bm{\sigma}}$ are zero (which corresponds to a central flux) and when they are positive (which corresponds to a penalty flux).   }
Integrating the velocity equations of (\ref{eq:dgform}) by parts gives
\begin{align*}
\sum_{D^k \in \Oh}\LRp{\rho \pd{\bm{v}}{t},\bm{w}}_{\Lk} &= \sum_{D^k \in \Oh}-\LRp{ \sum_{i=1}^d \bm{\sigma},\bm{A}_i \pd{\bm{w}}{\bm{x}_i}}_{\Lk}  + \LRa{\bm{A}_n^T\avg{\bm{\sigma}} + \frac{\tau_{\bm{v}}}{2}\bm{A}_n^T\bm{A}_n\jump{\bm{v}},\bm{w}}_{\Ldk} \nonumber\\
\sum_{D^k \in \Oh}\LRp{\bm{C}^{-1} \pd{\bm{\sigma}}{t},\bm{q}}_{\Lk} &= \sum_{D^k \in \Oh}\LRp{\sum_{i=1}^d \bm{A}_i \pd{\bm{v}}{\bm{x}_i},\bm{q}}_{\Lk} + \LRa{\frac{1}{2}\bm{A}_n\jump{\bm{v}} + \frac{\tau_{\sigma}}{2}\bm{A}_n\bm{A}_n^T\jump{\bm{\sigma}},\bm{q}}_{\Ldk} 
\end{align*}
Taking $(\bm{w},\bm{q}) = (\bm{v},\bm{\sigma})$ and adding both equations together yields
\begin{align*}
\sum_{D^k\in \Oh}&\frac{1}{2}\pd{}{t}\LRp{\LRp{\rho \bm{v},\bm{v}}_{\Lk} + \LRp{\bm{C}^{-1} \bm{\sigma},\bm{\sigma}}_{\Lk}} \\
&=\sum_{D^k\in \Omega_h}\LRa{\bm{A}_n^T\avg{\bm{\sigma}} + \frac{\tau_{\bm{v}}}{2}\bm{A}_n^T\bm{A}_n\jump{\bm{v}},\bm{v}}_{\partial D^k}+ \LRa{\frac{1}{2}\bm{A}_n\jump{\bm{v}} + \frac{\tau_{\bm{\sigma}}}{2}\bm{A}_n\bm{A}_n^T\jump{\bm{\sigma}},\bm{\sigma}}_{\partial D^k}  \\
&=\sum_{D^k\in \Omega_h}\sum_{f\in \partial D^k}\int_{f}\LRp{\reviewerTwo{\bm{v}^T}\bm{A}_n^T\avg{\bm{\sigma}} + \frac{\tau_{\bm{v}}}{2}\reviewerTwo{\bm{v}^T}\bm{A}_n^T\bm{A}_n\jump{\bm{v}} + \frac{1}{2}\reviewerTwo{\bm{\sigma}^T}\bm{A}_n\jump{\bm{v}} + \frac{\tau_{\bm{\sigma}}}{2}\reviewerTwo{\bm{\sigma}^T}\bm{A}_n\bm{A}_n^T\jump{\bm{\sigma}}}\diff{\bm{x}},
\end{align*}
\reviewerOne{where the term 
\[
\sum_{D^k\in \Oh}\frac{1}{2}\LRp{\LRp{\rho \bm{v},\bm{v}}_{\Lk} + \LRp{\bm{C}^{-1} \bm{\sigma},\bm{\sigma}}_{\Lk}}
\]
is the total energy of the system.}  Let $\Gamma_h$ denote the set of unique faces in $\Oh$, and let $\Gamma_v,\Gamma_\sigma,\Gamma_{\rm abc}$ denote the parts of the boundary where velocity, traction, and non-reflective boundary conditions are imposed, respectively.  We separate surface terms into contributions from interior shared faces and from boundary faces.  On an interior shared face, we sum contributions from the two adjacent elements to yield
\begin{align*}
&\sum_{f\in \Gamma_h \setminus \partial \Omega} \int_{f}\LRp{\reviewerTwo{\bm{v}^T}\bm{A}_n^T\avg{\bm{\sigma}} + \frac{\tau_{\bm{v}}}{2} \reviewerTwo{\bm{v}^T}\bm{A}_n^T\bm{A}_n\jump{\bm{v}} + \frac{1}{2}\reviewerTwo{\bm{\sigma}^T}\bm{A}_n\jump{\bm{v}} + \frac{\tau_{\bm{\sigma}}}{2}\reviewerTwo{\bm{\sigma}^T}\bm{A}_n\bm{A}_n^T\jump{\bm{\sigma}}}\diff{\bm{x}}\\
&= - \sum_{f \in\Gamma_h \setminus \partial \Omega} \int_{f} \LRp{\frac{\tau_{\bm{v}}}{2}\LRb{\bm{A}_n\jump{\bm{v}}}^2 + \frac{\tau_{\bm{\sigma}}}{2}\LRb{\bm{A}_n^T\jump{\bm{\sigma}}}^2}\diff{\bm{x}}.  
\end{align*}
For faces which lie on the boundary $\Gamma_v$ where velocity boundary conditions are imposed, $\jump{\bm{v}} = -2\bm{v}^-$, $\jump{\bm{A}_n^T\bm{\sigma}} = 0$, and $\bm{A}_n^T\avg{\bm{\sigma}} = \bm{Sn}^-$, implying that 
\begin{align*}
&\sum_{f\in \Gamma_v} \int_f\LRp{ \reviewerTwo{\bm{v}^T}\bm{A}_n^T\avg{\bm{\sigma}} + \frac{\tau_{\bm{v}}}{2}\reviewerTwo{\bm{v}^T}\bm{A}_n^T\bm{A}_n\jump{\bm{v}} + \frac{1}{2}\reviewerTwo{\bm{\sigma}^T}\bm{A}_n\jump{\bm{v}} + \frac{\tau_{\bm{\sigma}}}{2}\reviewerTwo{\bm{\sigma}^T}\bm{A}_n\bm{A}_n^T\jump{\bm{\sigma}}}\diff{\bm{x}}\\
&= \sum_{f\in \Gamma_v} \int_f \LRp{\reviewerTwo{\LRp{\bm{v}^-}^T}\bm{A}_n^T \bm{\sigma}^- - \reviewerTwo{\LRp{\bm{\sigma}^-}^T}\bm{A}_n \bm{v}^-  - \tau_{\bm{v}} \LRb{\bm{A}_n \bm{v}^-}^2}\diff{\bm{x}} = - \sum_{f\in \Gamma_v} \int_f \LRp{\tau_{\bm{v}}  \LRb{\bm{A}_n\bm{v}^-}^2}\diff{\bm{x}}.
\end{align*}
For faces which lie on $\Gamma_\sigma$, $\bm{A}_n^T\jump{\bm{\sigma}} = -2\bm{A}_n^T\bm{\sigma}^-$, $\bm{A}_n^T\avg{\bm{\sigma}} = 0$, and $\jump{\bm{v}} = 0$, yielding a similar contribution
\[
\sum_{f\in \Gamma_\sigma} \int_f \LRp{\reviewerTwo{{\bm{v}}^T}\bm{A}_n^T\avg{\bm{\sigma}} + \frac{\tau_{\bm{v}}}{2}\reviewerTwo{{\bm{v}}^T}\bm{A}_n^T\bm{A}_n\jump{\bm{v}} + \frac{1}{2}\reviewerTwo{{\bm{\sigma}}^T}\bm{A}_n\jump{\bm{v}} + \frac{\tau_{\bm{\sigma}}}{2}\reviewerTwo{{\bm{\sigma}}^T}\bm{A}_n\bm{A}_n^T\jump{\bm{\sigma}}}\diff{\bm{x}} = - \sum_{f\in \Gamma_\sigma} \int_f \LRp{\tau_{\bm{\sigma}}  \LRb{\bm{A}_n^T\bm{\sigma}^-}^2}\diff{\bm{x}}.
\]
Finally, for faces in $\Gamma_{\rm abc}$ we have $\bm{A}_n^T\avg{\bm{\sigma}} =\frac{1}{2} \bm{A}_n^T\bm{\sigma}^- $, $\bm{A}_n^T\jump{\bm{\sigma}} = -\bm{A}_n^T\bm{\sigma}^-$, and $\jump{\bm{v}} = -\bm{v}^-$, yielding
\begin{align*}
\sum_{f\in \Gamma_{\rm abc}} &\int_f \LRp{\reviewerTwo{\bm{v}^T}\bm{A}_n^T\avg{\bm{\sigma}} + \frac{\tau_{\bm{v}}}{2}\reviewerTwo{\bm{v}^T}\bm{A}_n^T\bm{A}_n\jump{\bm{v}} + \frac{1}{2}\reviewerTwo{\bm{\sigma}^T}\bm{A}_n\jump{\bm{v}}  + \frac{\tau_{\bm{\sigma}}}{2}\reviewerTwo{\bm{\sigma}^T}\bm{A}_n\bm{A}_n^T\jump{\bm{\sigma}}}\diff{\bm{x}} =\\
& - \sum_{f\in \Gamma_{\rm abc}} \int_f \LRp{ \frac{\tau_{\bm{v}}}{2}  \LRb{\bm{A}_n\bm{v}^-}^2 + \frac{\tau_{\bm{\sigma}}}{2}  \LRb{\bm{A}_n^T\bm{\sigma}^-}^2}\diff{\bm{x}}.  
\end{align*}
Combining all face contributions together gives the following result:
\begin{theorem}
The DG formulation (\ref{eq:dgform}) is energy stable for $\tau_{\bm{v}},\tau_{\bm{\sigma}} \geq 0$, in the sense that
\begin{align}
&\sum_{D^k\in \Oh} \frac{1}{2}\pd{}{t}\LRp{\LRp{\rho \bm{v},\bm{v}}_{\Lk} + \LRp{\bm{C}^{-1} \bm{\sigma},\bm{\sigma}}_{\Lk}} = - \sum_{f \in\Gamma_h \setminus \partial \Omega} \int_{f}\LRp{ \frac{\tau_{\bm{v}}}{2}\LRb{\bm{A}_n\jump{\bm{v}}}^2 + \frac{\tau_{\bm{\sigma}}}{2}\LRb{\bm{A}_n^T\jump{\bm{\sigma}}}^2}\diff{\bm{x}}\nonumber\\
& - \sum_{f\in \Gamma_v} \int_f \LRp{\tau_{\bm{v}}  \LRb{\bm{A}_n\bm{v}^-}^2}\diff{\bm{x}}  - \sum_{f\in \Gamma_\sigma} \int_f \LRp{\tau_{\bm{\sigma}}  \LRb{\bm{A}_n^T\bm{\sigma}^-}^2}\diff{\bm{x}} - \sum_{f\in \Gamma_{\rm abc}} \int_f \LRp{\frac{\tau_{\bm{v}}}{2} \LRb{\bm{A}_n\bm{v}^-}^2 + \frac{\tau_{\bm{\sigma}}}{2} \LRb{\bm{A}_n^T\bm{\sigma}^-}^2}\diff{\bm{x}} \leq 0.  
\label{eq:thm:stability}
\end{align}
\label{thm:stability}
\end{theorem}

Since $\rho$ and $\bm{C}^{-1}$ are positive definite, the left hand side of (\ref{eq:thm:stability}) is an $L^2$-equivalent norm on $(\bm{v},\bm{\sigma})$, and Theorem~\ref{thm:stability} implies that the magnitude of the DG solution $(\bm{v},\bm{\sigma})$ is non-increasing in time.  This also shows that dissipation present for positive penalization constants $\tau_{\bm{v}},\tau_{\sigma}$ acts on non-conforming components with non-zero jumps $\bm{A}_n^T\jump{\bm{v}}$ and $\bm{A}_n\jump{\bm{\sigma}}$.  In fact, it was shown in \cite{chan2016short} that, in the limit as $\tau_{\bm{v}},\tau_{\bm{\sigma}} \rightarrow\infty$, the eigenspaces of DG discretizations split into a conforming part consisting of $\bm{u},\bm{\sigma}$ which satisfy
\[
\bm{A}_n\jump{\bm{u}} = 0, \qquad \bm{A}_n^T\jump{\bm{\sigma}} = 0
\]
and an non-conforming part (defined through the $L^2$ orthogonal complement) corresponding to eigenvalues contain real parts which approach $-\infty$.  For the linear elastic wave equations, these conditions are equivalent to requirements of $C^0$ continuity for $\bm{u}$ and normal continuity of the stress tensor $\jump{\tilde{\bm{S}}}\bm{n} = 0$.

\subsection{The semi-discrete matrix system for DG}

The solution to (\ref{eq:dgform}) can be approximated by discretizing in space and using an explicit time integrator, which requires only evaluations of local contributions over $D^k$ to the DG formulation
\begin{align}
\LRp{\rho \pd{\bm{v}}{t},\bm{w}}_{\Lk} &= {\LRp{ \sum_{i=1}^d \bm{A}_i^T\pd{\bm{\sigma}}{\bm{x}_i},\bm{w}}_{\Lk}  + \LRa{\frac{1}{2}\bm{A}_n^T\jump{\bm{\sigma}} + \frac{\tau_{\bm{v}}}{2}\bm{A}_n^T\bm{A}_n\jump{\bm{v}},\bm{w}}_{\Ldk}} \nonumber\\
\LRp{\bm{C}^{-1} \pd{\bm{\sigma}}{t},\bm{q}}_{\Lk} &= {\LRp{\sum_{i=1}^d \bm{A}_i \pd{\bm{v}}{\bm{x}_i},\bm{q}}_{\Lk} + \LRa{\frac{1}{2}\bm{A}_n\jump{\bm{v}} + \frac{\tau_{\sigma}}{2}\bm{A}_n\bm{A}_n^T\jump{\bm{\sigma}},\bm{q}}_{\Ldk}}. 
\label{eq:dgloc} 
\end{align}

Let $\LRc{\phi_i}_{i=1}^{N_p}$ be a basis for $P^N\LRp{\widehat{D}}$.\footnote{\reviewerOne{In our implementations, we use nodal basis functions at Warp and Blend interpolation points \cite{warburton2006explicit}.   These basis functions are defined implicitly using an orthogonal polynomial basis on the reference simplex \cite{hesthaven2007nodal}.  However, we note that the implementation and formulation are independent of the specific choice of polynomial basis.  }}  We define the reference mass matrix $\widehat{\bm{M}}$ and the physical mass matrix $\bm{M}$ for an element $D^k$ as
\[
\LRp{\widehat{\bm{M}}}_{ij} = \int_{\widehat{D}}\phi_j\phi_i \diff{\widehat{\bm{x}}}, \qquad \LRp{\bm{M}}_{ij} = \int_{D^k} \phi_j\phi_i \diff{\bm{x}} = \int_{\widehat{D}}\phi_j\phi_i J\diff{ \widehat{\bm{x}}}.
\]
For affine mappings, $J$ is constant and $\bm{M} = J\widehat{\bm{M}}$.  We also define weak differentiation matrices $\bm{S}_{k}$ and face mass matrix $\bm{M}_f$ such that
\[
\LRp{\bm{S}_{k}}_{ij} =  \int_{D^k} \pd{\phi_j}{\bm{x}_k}\phi_i \diff{\bm{x}}, \qquad \LRp{\bm{M}_{f}}_{ij} =  \int_{f} \phi_j \phi_i \diff{\bm{x}} = \int_{\widehat{f}} \phi_j \phi_i J^f\diff{\widehat{\bm{x}}},
\]
where $J^f$ is the Jacobian of the mapping from a reference face $\widehat{f}$ to $f$.  For affinely mapped simplices, $J^f$ is also constant and $\bm{M}_f = J^f \widehat{\bm{M}}_f$, where the definition of the reference face mass matrix $\widehat{\bm{M}}_f$ is analogous to the definition of the reference mass matrix $\widehat{\bm{M}}$.

Finally, we define weighted mass matrices.  Let $w(\bm{x})\in \mathbb{R}$ and $\bm{W}(\bm{x})\in \mathbb{R}^{m\times n}$.  Then, scalar and matrix-weighted mass matrices $\bm{M}_w$ and $\bm{M}_{\bm{W}}$ are defined through
\[
\LRp{\bm{M}_w}_{ij} = \int_{D^k} w(\bm{x}) \phi_j(\bm{x})  \phi_i(\bm{x})\diff{\bm{x}}, \qquad \bm{M}_{\bm{W}} = \LRp{\begin{array}{ccc}
\bm{M}_{\bm{W}_{1,1}}& \ldots & \bm{M}_{\bm{W}_{1,n}}\\
\vdots & \ddots & \vdots\\
\bm{M}_{\bm{W}_{m,1}} & \ldots & \bm{M}_{\bm{W}_{m,n}}\\
\end{array}},
\]
where $\bm{M}_{\bm{W}_{i,j}}$ is the scalar weighted mass matrix weighted by the $(i,j)$ entry of $\bm{W}$. 

Local contributions to the DG variational form may then be evaluated in a quadrature-free manner using these matrices.  Let $\bm{\Sigma}_i,\bm{V}_i$ denote vectors containing degrees of freedom for solution components $\bm{\sigma}_i, \bm{v}_i$ such that 
\begin{align*}
\bm{\sigma}_{i}(\bm{x},t) &= \sum_{j=1}^{N_p} \reviewerOne{\LRp{\bm{\Sigma}_i(t)}_j} \phi_j(\bm{x}), \qquad 1 \leq i \leq N_d\\ 
\bm{v}_{i}(\bm{x},t) &= \sum_{j=1}^{N_p} \reviewerOne{\LRp{\bm{V}_i(t)}_j} \phi_j(\bm{x}), \qquad 1 \leq i \leq d. 
\end{align*}
Then, the local DG formulation can be written as a block system of ordinary differential equations (ODEs)
by concatenating $\bm{\Sigma}_i,\bm{V}_i$ into single vectors $\bm{\Sigma},\bm{V}$ and using the Kronecker product $\otimes$
\begin{align}
\bm{M}_{\rho \bm{I}} \pd{\bm{V}}{t} &= \sum_{i=1}^d  \LRp{\bm{A}_i^T\otimes \bm{S}_{i}} \bm{\Sigma} + \sum_{f\in \partial D^k}\LRp{\bm{I}\otimes \bm{M}_f}\bm{F}_{v}\nonumber\\
\bm{M}_{\bm{C}^{-1}} \pd{\bm{\Sigma}}{t} &= \sum_{i=1}^d \LRp{\bm{A}_i\otimes \bm{S}_{i}}\bm{V}  +  \sum_{f\in \partial D^k}\LRp{\bm{I}\otimes\bm{M}_f}\bm{F}_{\sigma}. 
\label{eq:dgmat}
\end{align}
where $\bm{F}_{v}$ and $\bm{F}_{\sigma}$ denote degrees of freedom for velocity and stress numerical fluxes.  

In order to apply standard time integration methods, we must invert $\bm{M}_{\rho\bm{I}}$ and $\bm{M}_{\bm{C}^{-1}}$ to isolate $\pd{\bm{v}}{t}$ and $\pd{\bm{\sigma}}{t}$ on the left hand side.  While the inversion of $\bm{M}_{\rho\bm{I}}$ and $\bm{M}_{\bm{C}^{-1}}$ can be parallelized from element to element, doing so typically requires either the precomputation and storage of large dense matrix inverses or the on-the-fly construction and solution of a large dense matrix system at every time-step.  The former option requires a large amount of storage, while the latter option is computationally expensive and difficult to parallelize.  This cost can be overcome for $\rho, \bm{C}$ which are constant over an element $D^k$, in which case $\bm{M}_{\rho\bm{I}}$ is block diagonal with identical blocks $\bm{M}_\rho = \rho\bm{M}$, while $\bm{M}_{\bm{C}^{-1}}$ reduces to
\[
\bm{M}_{\bm{C}^{-1}} = \LRp{\begin{array}{ccc}
\bm{C}_{1,1}^{-1}\bm{M}& \ldots & \bm{C}^{-1}_{1,N_d}\bm{M}\\
\vdots & \ddots & \vdots\\
\bm{C}_{N_d,1}^{-1}\bm{M} & \ldots & \bm{C}^{-1}_{N_d,N_d}\bm{M}\\
\end{array}} 
= \LRp{\bm{C}^{-1}\otimes \bm{M}}.
\]
Then, $\bm{M}^{-1}_\rho = \frac{1}{\rho}\bm{M}^{-1} = \frac{1}{J\rho}\widehat{\bm{M}}^{-1}$, and $\bm{M}^{-1}_{\bm{C}^{-1}} = \bm{C}\otimes \bm{M}^{-1} = \bm{C}\otimes \LRp{\frac{1}{J}\widehat{\bm{M}}^{-1}}$, and each matrix inverse can be applied using the inverse of the reference mass matrix $\widehat{\bm{M}}^{-1}$ and the values of $\rho, \bm{C}$, and $J$ over each element.  Applying this observation to (\ref{eq:dgmat}) then yields the following local system of ODEs 
\begin{align*}
\pd{\bm{V}}{t} &= \sum_{i=1}^d  \LRp{\frac{1}{\rho}\bm{A}_i^T\otimes \bm{D}_{i}} \bm{\Sigma} + \sum_{f\in \partial D^k}\LRp{\frac{1}{\rho}\bm{I}\otimes \bm{L}_f}\bm{F}_{v}\\
\pd{\bm{\Sigma}}{t} &= \sum_{i=1}^d \LRp{\bm{C}\bm{A}_i\otimes \bm{D}_{i}}\bm{V}  +  \sum_{f\in \partial D^k}\LRp{\bm{C}\otimes\bm{L}_f}\bm{F}_{\sigma},
\end{align*}
where we have introduced the differentiation matrix $\bm{D}_i = \bm{M}^{-1}\bm{S}_i$ and lift matrix $\bm{L}_f = \bm{M}^{-1}\bm{M}_f$.  For affine elements, both derivative and lift matrices may be applied using only geometric factors and reference derivative and lift matrices.  

Unfortunately, if $\rho$ and $\bm{C}$ vary spatially within an element, the above approach can no longer be used to invert $\bm{M}_\rho$ and $\bm{M}_{\bm{C}^{-1}}$ in an efficient and low-storage manner.  
For isotropic media, one way to address sub-element variations in material parameters is to diagonalize the matrix $\bm{C}$ through a change of variables \cite{etienne2010hp}.  This results in a local system of ODEs with only scalar weighted mass matrices \cite{mercerat2015nodal}, which can be treated using scalar weight-adjusted approximations.  We take a different approach in addressing these issues and approximate the matrix-weighted $L^2$ inner product (and corresponding matrix-weighted mass matrix $\bm{M}_{\bm{C}^{-1}}$) using a weight-adjusted approximation which is low storage, simple to invert, energy stable, and provably high order accurate for spatially varying weights $\rho, \bm{C}$ with sufficiently regularity.

\section{Weight-adjusted inner products for matrix-valued weights}
\label{sec:mwadg}
Weight-adjusted inner products are high order accurate approximations of weighted $L^2$ inner products.  These can be interpreted as generalizations of mass lumping techniques, reducing to mass lumping when integrals are evaluated with appropriate quadrature rules.  These weight-adjusted inner products result in weight-adjusted mass matrices, whose inverses approximate the inverse of a weighted $L^2$ mass matrix.  

We wish to apply weight-adjusted approximations to avoid the inversion of $\bm{M}_\rho$ and $\bm{M}_{\bm{C}^{-1}}$.  Approximating the inverse of $\bm{M}_\rho$ can be done using weight-adjusted approximations for scalar weights \cite{chan2016weight1,chan2016weight2}, which we review in Section~\ref{sec:swadg}.  We then extend scalar weight-adjusted approximations to matrix-valued weights in Section~\ref{sec:mwadg} to approximate the inverse of $\bm{M}_{\bm{C}^{-1}}$.  

\subsection{Scalar weight adjusted inner products}
\label{sec:swadg}

We introduce standard Lebesgue $L^p$ norms and their associated $L^p$ spaces over a general domain $\Omega$
\begin{align*}
\nor{u}_{L^p\LRp{\Omega}} = \LRp{\int_{\Omega} u^p \diff{\bm{x}}}^{1/p}, \qquad  
L^p\LRp{\Omega} = \LRc{u: \Omega\rightarrow \mathbb{R}, \quad \nor{u}_{L^p\LRp{\Omega}} < \infty} 
\end{align*}
for $1 \leq p < \infty$.  For $p = \infty$, these are defined as
\begin{align*}
\nor{u}_{L^{\infty}\LRp{\Omega}} = \inf\LRc{C \geq 0: \LRb{u\LRp{\bm{x}}} \leq C \quad \forall \bm{x}\in \Omega}, \qquad L^{\infty}\LRp{\Omega} &= \LRc{u: \Omega\rightarrow \mathbb{R}, \quad \nor{u}_{L^{\infty}\LRp{\Omega}} < \infty}.
\end{align*}
These induce $L^p$ Sobolev seminorms and norms of degree $s$ 
\begin{align*}
\LRb{u}_{W^{s,p}\LRp{\Omega}} &= \LRp{\sum_{\LRb{\alpha}= s} \nor{ D^{\alpha} u}_{L^p\LRp{\Omega}}^p}^{1/p}, \qquad \LRb{u}_{W^{s,\infty}\LRp{\Omega}} = \max_{\LRb{\alpha}= s} \nor{D^{\alpha}u}_{L^{\infty}\LRp{\Omega}}\\
\nor{u}_{W^{s,p}\LRp{\Omega}} &= \LRp{\sum_{\LRb{\alpha}\leq s} \nor{ D^{\alpha} u}_{L^p\LRp{\Omega}}^p}^{1/p}, \qquad \nor{u}_{W^{s,\infty}\LRp{\Omega}} = \max_{\LRb{\alpha}\leq s} \nor{D^{\alpha}u}_{L^{\infty}\LRp{\Omega}}.
\end{align*}
where $\alpha = \LRc{\alpha_1,\ldots,\alpha_d}$ is a multi-index such that
\begin{align*}
D^{\alpha}u &= \pd{^{\alpha_1}}{x^{\alpha_1}}\pd{^{\alpha_2}}{y^{\alpha_2}} u, \qquad d = 2\\
D^{\alpha}u &= \pd{^{\alpha_1}}{x^{\alpha_1}}\pd{^{\alpha_2}}{y^{\alpha_2}}\pd{^{\alpha_3}}{z^{\alpha_3}} u, \qquad d = 3.
\end{align*}

\reviewerOne{Let $\Pi_N$ denote the $L^2$ projection on the element ${D}^k$.  For affine elements where $J$ is constant such that
\[
\LRp{wu,v}_{\Lk} = \LRp{wu,vJ}_{\widehat{D}} = J\LRp{wu,v}_{\widehat{D}},
\]
$\Pi_N$ is equivalent to the $L^2$ projection on the reference element $\widehat{D}$.}  We define two operators $T_w: L^2\LRp{D^k} \rightarrow P^N\LRp{D^k}$ and $T^{-1}_w: L^2\LRp{D^k} \rightarrow P^N\LRp{D^k}$ such that
\begin{align*}
T_wu &= \Pi_N(wu)\\
\LRp{wT^{-1}_wu,v}_{\Lk} &= \LRp{u,v}_{\Lk}, \qquad \forall v\in P^N\LRp{D^k}.  
\end{align*}
A weighted $L^2$ inner product $\LRp{wu,v}_{\Lk}$ can be approximated by a weight-adjusted inner product
\[
\LRp{wu,v}_{\Lk} = \LRp{T_wu,v}_{\Lk}\approx \LRp{T^{-1}_{1/w}u,v}_{\Lk}.  
\]
based on the observation that $T^{-1}_{1/w}u \approx uw$.  \reviewerOne{The intuition behind this approximation is that, by the definition of $T^{-1}_{1/w}$, 
\[
\LRp{\frac{1}{w}T^{-1}_{1/w}u - u,v}_{\Lk} = 0, \qquad \forall v\in P^N\LRp{D^k}.  
\]
This shows that $\frac{1}{w}T^{-1}_{1/w}u - u$ is orthogonal to all polynomials of degree $N$, implying that $\frac{1}{w}T^{-1}_{1/w}u - u\approx 0$ and $T^{-1}_{1/w}u \approx wu$ (for $w(x), u(x)$ which are smooth and well-represented by polynomials).  
}

This approximation is made precise by the following estimates for approximations of the product $uw$ and weighted moments on affinely mapped elements:
\begin{theorem}[Theorem 5 in \cite{chan2016weight1}]
Let $D^k$ be quasi-regular with representative size $h = {\rm diam}\LRp{D^k}$.  For $N \geq 0$, $w\in W^{N+1,\infty}\LRp{D^k}$, and $u \in W^{N+1,2}\LRp{D^k}$, 
\begin{align}
\nor{{u}{w}-T_{w}u}_{L^2\LRp{D^k}} &\leq C_wh^{N+1} \nor{u}_{W^{N+1,2}\LRp{D^k}},\\
\nor{{u}{w}-T^{-1}_{1/w}u}_{L^2\LRp{D^k}} &\leq C_wh^{N+1}   \nor{u}_{W^{N+1,2}\LRp{D^k}}.
\end{align}
where $C_w = C\nor{w}_{L^{\infty}\LRp{D^k}}\nor{\frac{1}{w}}_{L^{\infty}\LRp{D^k}} \nor{w}_{W^{N+1,\infty}\LRp{D^k}}$.
\end{theorem}

These results rely on a scalar weighted interpolation estimate derived in \cite{warburton2013low, chan2016weight2} for a general non-affine element $D^k$.  
\begin{theorem}[\reviewerTwo{Theorem 1 in \cite{chan2016weight2}.}]
Let $D^k$ be a quasi-regular element with representative size $h = {\rm diam}\LRp{D^k}$. For $N \geq 0$, $w\in W^{N+1,\infty}\LRp{D^k}$, and $u\in W^{N+1,2}\LRp{D^k}$, 
\begin{align*}
\nor{u - \frac{1}{w}\Pi_N(wu)}_{\Lk} &\leq C h^{N+1} \nor{\frac{1}{\sqrt{J}}}_{\LinfDk}\nor{\frac{\sqrt{J}}{w}}_{\LinfDk}
\nor{w}_{W^{N+1,\infty}\LRp{D^k}}\nor{u}_{W^{N+1,2}\LRp{D^k}}.
\end{align*}
\label{thm:wproj}
\end{theorem}

\subsection{Extension to matrix weights}

We now generalize weight-adjusted inner products to the case of matrix-valued weights.  We first define appropriate generalizations of norms used in Section~\ref{sec:swadg} to vector-valued functions.  Let $D^{\bm{\alpha}}\bm{v}$ denote component-wise differentiation of $\bm{v}$ with respect to a $d$-dimensional multi-index $\bm{\alpha}$.  Then, vector $L^p$ Sobolev norms for $\bm{v}(\bm{x}) \in \mathbb{R}^m$ can be defined as 
\begin{align*}
\LRb{\bm{v}}_{W^{k,p}}^p &= {\sum_{i=1}^m \LRb{\bm{v}_i}_{W^{k,p}}^p}, \qquad \nor{\bm{v}}_{W^{k,p}}^p = {\sum_{i=1}^m \nor{\bm{v}_i}_{W^{k,p}}^p} \qquad 1 \leq p < \infty,\\
\LRb{\bm{v}}_{W^{k,\infty}} &= \max_i \LRb{\bm{v}_i}_{W^{k,\infty}}, \qquad \nor{\bm{v}}_{W^{k,\infty}} = \max_i \nor{\bm{v}_i}_{W^{k,\infty}}.
\end{align*}
The corresponding Sobolev spaces $W^{k,p}$ and $W^{k,\infty}$ are defined similarly to the scalar case.  

Let $\bm{W}(\bm{x})$ be a matrix-valued weight function which is pointwise symmetric positive-definite 
\[
0 < w_{\min} \leq  \nor{\bm{W}(\bm{x})}_2 \leq w_{\max} < \infty, \qquad 0 < \tilde{w}_{\min} \leq  \nor{\bm{W}^{-1}(\bm{x})}_2 \leq \tilde{w}_{\max} < \infty, \qquad \forall \bm{x} \in \Omega.
\]
We define a $k$th order Sobolev norm for $\bm{W}(\bm{x})$ in terms of the induced $p$-norm
\begin{align*}
\nor{\bm{W}(\bm{x})}_{k,p,\infty}^p &= \sum_{\LRb{\bm{\alpha}} \leq k} \sup_{\bm{x}} \nor{D^{\bm{\alpha}}\bm{W}(\bm{x})}^p_p
\end{align*}
where $D^{\bm{\alpha}}\bm{W}(\bm{x})$ again denotes component-wise differentiation.  While this norm is not sub-multiplicative, the following bound holds
\begin{align*}
\nor{\bm{W}{\bm{v}}}_{W^{k,p}}^p &= \sum_{\LRb{\alpha}\leq k} \nor{D^{\alpha}\LRp{\bm{W}\bm{v}}}_{L^p}^p 
\leq C_N \int  \sum_{\LRb{\alpha}\leq k} \sum_{\LRb{\beta}\leq \LRb{\alpha}} \nor{\LRp{D^{\beta}\bm{W}}\LRp{D^{\alpha-\beta}\bm{v}} }_{p}^p \diff{\bm{x}}\\
&\leq C_N\int \LRp{\sum_{\LRb{\alpha}\leq k} \nor{\LRp{D^{\alpha}\bm{W}}}_p}^p \LRp{\sum_{\LRb{\alpha}\leq k} \nor{D^{\alpha}\bm{v}}_{p} }^p\diff{\bm{x}}\\
&\leq C_N \nor{\bm{W}}_{k,p,\infty}^p \nor{\bm{v}}_{k,p}^p,
\end{align*}
where we have used Leibniz's rule, Cauchy-Schwarz, and the arithmetic-geometric mean inequality.  

The following theorem extends Theorem~\ref{thm:wproj} to matrix weights by computing weighted interpolation estimates for the quantity $\bm{W}^{-1}\Pi_N\LRp{\bm{W}\bm{v}}$.  
\begin{theorem}
Let $D^k$ be a quasi-regular element with representative size $h = {\rm diam}\LRp{D^k}$. For $N \geq 0$, $\bm{W} \in \LRp{W^{N+1,\infty}\LRp{D^k}}^{d\times d}$, and $\bm{v}\in \LRp{W^{N+1,2}\LRp{D^k}}^d$, 
\[
\nor{\bm{v} - \bm{W}^{-1}\Pi_N\LRp{\bm{W}\bm{v}}}_{\Lk} \leq Ch^{N+1} \nor{{\sqrt{J}}}_{\LinfDk}\nor{\frac{1}{\sqrt{J}}}_{\LinfDk}\tilde{w}_{\max}\nor{\bm{W}}_{N+1,2,\infty} \nor{\bm{v}}_{W^{N+1,2}\LRp{D^k}}
\]
\label{thm:wwproj}
\end{theorem}
\begin{proof}
The proof is similar to the scalar case.  Using vector-valued versions of Bramble-Hilbert and a scaling argument for quasi-regular elements yields
\begin{align*}
\nor{\bm{v} - \bm{W}^{-1}\Pi_N\LRp{\bm{W}\bm{v}}}_{\Lk} &\leq C_1 \nor{\sqrt{J}}_{\LinfDk} \sup_{\bm{x}}\nor{\bm{W}^{-1}}_{2} \nor{\bm{W}\bm{v} - \Pi_N\LRp{\bm{W}\bm{v}}}_{\reviewerTwo{L^2\LRp{\Dhat}}}\\
&\leq C_1 \nor{{\sqrt{J}}}_{\LinfDk} \sup_{\bm{x}}\nor{\bm{W}^{-1}}_{2} \LRb{\bm{W}\bm{v}}_{W^{N+1,2}\LRp{\Dhat}}\\
&\leq C_2 h^{N+1}\nor{{\sqrt{J}}}_{\LinfDk}\nor{\frac{1}{\sqrt{J}}}_{\LinfDk} \sup_{\bm{x}}\nor{\bm{W}^{-1}}_{2} \nor{\bm{W}\bm{v}}_{W^{N+1,2}\LRp{D^k}}\\
&\leq C_3 h^{N+1}\nor{{\sqrt{J}}}_{\LinfDk}\nor{\frac{1}{\sqrt{J}}}_{\LinfDk} \tilde{w}_{\max} \nor{\bm{W}}_{{N+1,2,\infty}}\nor{\bm{v}}_{W^{N+1,2}\LRp{D^k}}.
\end{align*}
\end{proof}

\subsubsection{Weight-adjusted approximations with matrix weights}
\label{sec:mwadgprop}

Let $\Pi_N\bm{u}$ be defined as the $L^2$ projection applied to each component of the vector-valued function $\bm{u}$.  We then define the operator $T_{\bm{W}}$ analogously to the scalar case
\[
T_{\bm{W}}\bm{v} = \Pi_N\LRp{\bm{W}\bm{v}}.  
\]
The inverse operator $T^{-1}_{\bm{W}}$ is defined implicitly via
\[
\LRp{\bm{W}T^{-1}_{\bm{W}}\bm{v},\reviewerOne{\bm{\delta v}}}_{\Lk} = \LRp{\bm{v},\reviewerOne{\bm{\delta v}}}_{\Lk}, \qquad \forall \reviewerOne{\bm{\delta v}} \in \LRp{P^N\LRp{D^k}}^m.
\]
This definition is a straightforward generalization of $T^{-1}_{w}$ to matrix-valued weights $\bm{W}$.  Conveniently, all properties of $T_{w},T^{-1}_{w}$ for scalar $w(\bm{x})$ \cite{chan2016weight1} carry over to matrix weights $\bm{W}(\bm{x})$ as well.
\begin{lemma}
Let $\Pi_N$ denote the component-wise $L^2$ projection, and let $\bm{W} \in \LRp{L^{\infty}}^{m\times m}$.  Then, $T_{\bm{W}}$ satisfies the following properties:
\begin{enumerate}
\vspace{.5em}
\item $T_{\bm{W}}^{-1}T_{\bm{W}} = \Pi_N$ \label{p1}
\vspace{.5em}
\item $\Pi_N T_{\bm{W}}^{-1} = T_{\bm{W}}^{-1}\Pi_N  =  T_{\bm{W}}^{-1}$ \label{p2}
\vspace{.5em}
\item $\nor{T^{-1}_{\bm{W}}}_{\Lk} \leq \tilde{w}_{\max}$. \label{p4}
\vspace{.5em}
\item $\LRp{T^{-1}_{\bm{W}} \bm{v},\bm{w}}_{\Lk}$ forms an inner product on $\LRp{P^N}^m \times \LRp{P^N}^m$, which is equivalent to the $L^2$ inner product with equivalence constants $C_1 = \tilde{w}_{\min}, C_2 = \tilde{w}_{\max}$. \label{p3}
\vspace{.5em}
\end{enumerate}
\label{lemma:props}
\end{lemma}
\begin{proof}
The proofs of properties~\ref{p1} and \ref{p2} are consequences of the definition of $T_{\bm{W}}$, and are identical to proofs for the scalar case.  Property~\ref{p4} is a straightforward extension from the scalar case.  Let $\bm{v}\in \LRp{P^N}^m$.  Then, 
\[
\LRp{T^{-1}_{\bm{W}} \bm{v},\bm{v}}_{\Lk} = \LRp{\bm{W}^{-1}\bm{W}T^{-1}_{\bm{W}} \bm{v},\bm{v}}_{\Lk} \leq \sup_{\bm{x}}\nor{\bm{W}^{-1}(\bm{x})}\LRp{\bm{W}T^{-1}_{\bm{W}} \bm{v},\bm{v}}_{\Lk} = \tilde{w}_{\max} \nor{\bm{v}}_{\Lk}^2.
\]
Property~\ref{p3} then simply requires the lower bound
\[
\LRp{T^{-1}_{\bm{W}} \bm{v},\bm{v}}_{\Lk} = \LRp{\bm{W}^{-1}\bm{W}T^{-1}_{\bm{W}} \bm{v},\bm{v}}_{\Lk} \geq \inf_{\bm{x}}\nor{\bm{W}^{-1}(\bm{x})}\nor{\bm{v}}_{\Lk}^2 = \tilde{w}_{\min}\nor{\bm{v}}_{\Lk}^2 .  
\]
\end{proof}

Using Theorem~\ref{thm:wwproj}, we may also show that the matrix-valued weight-adjusted inner product is also high order accurate for sufficiently regular $\bm{W}$.
\begin{theorem}
Let $D^k$ be a quasi-regular element with representative size $h = {\rm diam}\LRp{D^k}$. For $N > 0$
\[
\nor{ \bm{W}\bm{v} - {T}^{-1}_{\bm{W}^{-1}}\bm{v} }_{\Lk } \leq C_{\bm{W}} h^{N+1} \nor{\bm{v}}_{W^{N+1,2}\LRp{D^k}}
\]
with constant $C_{\bm{W}}$ depending on $\bm{W}$ and $N$
\[
C_{\bm{W}} = C_N \nor{\sqrt{J}}_{\LinfDk}\nor{\frac{1}{\sqrt{J}}}_{\LinfDk} \tilde{w}_{\max}{w}_{\max} \nor{\bm{W}}_{{N+1,2,\infty}}.  
\]
\label{thm:west}
\end{theorem}
\begin{proof}
The proof follows the scalar case.  The triangle inequality gives
\[
\nor{ \bm{W}\bm{v} - {T}^{-1}_{\bm{W}^{-1}}\bm{v} }_{\Lk } \leq \nor{ \bm{W}\bm{v} - \Pi_N(\bm{W}\bm{v})}_{\Lk} + \nor{\Pi_N(\bm{W}\bm{v}) - T^{-1}_{\bm{W}^{-1}}\bm{v} }_{\Lk }.  
\]
The former term is bounded by interpolation estimates and by arguments used in the proof of Theorem~\ref{thm:wwproj}.  The latter term is bounded 
\begin{align*}
\nor{\Pi_N(\bm{W}\bm{v})- T^{-1}_{\bm{W}^{-1}}\bm{v} }_{\Lk} 
&= \reviewerTwo{\nor{T^{-1}_{\bm{W}^{-1}} \Pi_N \LRp{T_{\bm{W}^{-1}}\Pi_N(\bm{W}\bm{v})}- T^{-1}_{\bm{W}^{-1}}\Pi_N\bm{v} }_{\Lk } }\\
&\leq C_N \nor{T^{-1}_{\bm{W}^{-1}}} \nor{ T_{\bm{W}^{-1}}\Pi_N(\bm{W}\bm{v})-\Pi_N\bm{v} }_{\Lk }\\
&\leq C_N \nor{T^{-1}_{\bm{W}^{-1}}} \nor{ \Pi_N\LRp{\bm{W}^{-1}\Pi_N(\bm{W}\bm{v})} - \Pi_N\bm{v} }_{\Lk }\\
&\leq C_N {w}_{\max} \nor{ \bm{W}^{-1}\Pi_N(\bm{W}\bm{v}) - \bm{v} }_{\Lk }.
\end{align*}
where we have used $\nor{T_{\bm{W}^{-1}}^{-1}} \leq {w}_{\max}$ (Lemma~\ref{lemma:props}) and $\nor{\Pi_N}_{\Lk} = 1$.  
An application of Theorem~\ref{thm:wwproj} to bound $\nor{ \bm{W}^{-1}\Pi_N(\bm{W}\bm{v}) - \bm{v} }_{\Lk }$ completes the proof.
\end{proof}

We note that these estimates are tight, in the sense that estimates for scalar weight-adjusted inner products are recovered when the matrix weight is taken to be $\bm{W} = w\bm{I}$.  


\subsubsection{Approximation of weighted mass matrix inverses}
\label{sec:wadgimplement}
The advantage of using weight-adjusted inner products is that the corresponding weight-adjusted mass matrices are straightforward to invert.  For scalar weights, the weight-adjusted mass matrix approximates the weighted $L^2$ mass matrix and its inverse by
\[
\bm{M}_w \approx \bm{M}\bm{M}^{-1}_{1/w}\bm{M}, \qquad \bm{M}^{-1}_w \approx \bm{M}^{-1}\bm{M}_{1/w}\bm{M}^{-1}.
\]
By evaluating $\bm{M}_{1/w}$ in a matrix-free fashion, the inverse of the weight-adjusted mass matrix $\bm{M}^{-1}\bm{M}_{1/w}\bm{M}^{-1}$ \reviewerOne{yields a low storage implementation using a sufficiently accurate quadrature rule}.  \reviewerOne{Let $\widehat{\bm{x}}_i, \widehat{w}_i$ denote quadrature points and weights on the reference element, and let $\bm{V}_q$ denote the matrix 
\[
\LRp{\bm{V}_q}_{ij} = \phi_j(\widehat{\bm{x}}_i).  
\]
whose columns correspond to evaluations of basis functions at quadrature points.  Then, for affine elements, ${\bm{M}} = J\widehat{\bm{M}} = J\bm{V}_q^T{\rm diag}(\widehat{w}_i) \bm{V}_q$, where $\widehat{\bm{M}}$ is the reference mass matrix and $J$ is the determinant of the Jacobian of the reference-to-physical mapping, which is constant for affine mappings.  Additionally, 
\[
\bm{M}_{1/w} = J\bm{V}_q^T{\rm diag}(\widehat{w}_i/ w(\widehat{\bm{x}}_i)) \bm{V}_q,
\]
where $w(\widehat{\bm{x}}_i)$ denotes the evaluation of the weight function $w(\bm{x})$ at quadrature points.  Thus, for a vector $\bm{u}$, the inverse of the weight-adjusted mass matrix can be applied as follows 
\[
\bm{M}^{-1}\bm{M}_{1/w}\bm{M}^{-1}\bm{u} = \bm{M}^{-1} \bm{V}_q^T{\rm diag}(\widehat{w}_i) {\rm diag}\LRp{\frac{1}{ w(\bm{x}_i)}} \bm{V}_q \bm{M}^{-1}\bm{u} = \bm{P}_q {\rm diag}\LRp{\frac{1}{w(\widehat{\bm{x}}_i)}} \bm{V}_q \frac{1}{J} \widehat{\bm{M}}^{-1}\bm{u},
\]
where we have introduced the quadrature-based $L^2$ projection operator on the reference element $\bm{P}_q = \widehat{\bm{M}}^{-1} \bm{V}_q^T{\rm diag}(\widehat{w}_i)$.  

In the context of DG using explicit time-stepping, the factor of $\widehat{\bm{M}}^{-1}$ can be premultiplied into the right hand side (i.e.\ the evaluation of the spatial discretization).  Then, applying the weight-adjusted mass matrix requires only storage of two reference matrices $\bm{V}_q$ and $\bm{P}_q$ and the values of the weight function at quadrature points $w(\widehat{\bm{x}}_i)$.  We assume that the number of quadrature points is $O(N^3)$, which is true for most simplicial quadratures \cite{xiao2010quadrature, karniadakis2013spectral}.  

The overall storage cost of applying the weight-adjusted mass matrix using the above implementation is $O(N^3)$ per element, while the pre-computation and storage of DG operators involving inverses of weighted mass matrices requires $O(N^6)$ storage per element.  However, we note that unlike the aforementioned quadrature-based implementation of WADG, the strategy of precomputation and storage used by Mercerat and Glinsky in \cite{mercerat2015nodal} can accomodate arbitrarily high accuracy quadrature rules without any increase in computational cost.  

Finally, we note that the computational impact of storage costs vary from architecture to architecture.  As pointed out in \cite{warburton2013low, chan2016weight1, chan2016weight2}, the limited storage of accelerator architectures such as GPUs limits the maximum feasible problem size, and decreasing storage costs with respect to the degree $N$ allows one to run higher order simulations on larger meshes before running out of memory.  However, for distributed parallelism implementations of DG on large supercomputing clusters, storage limitations may be less of an issue.  We limit our focus to GPU computations in this work, and present results comparing the computational cost of WADG to several alternatives in Section~\ref{sec:comp}.  
}

For weight-adjusted inner products with matrix-valued weights, the corresponding weight-adjusted mass matrices approximate weighted $L^2$ mass matrices and inverses in a similar fashion
\begin{align}
\bm{M}_{\bm{W}} \approx \LRp{\bm{I} \otimes \bm{M}} \bm{M}^{-1}_{\bm{W}^{-1}} \LRp{\bm{I} \otimes \bm{M}} \nonumber\\
\bm{M}_{\bm{W}}^{-1} \approx \LRp{\bm{I} \otimes \bm{M}^{-1}} \bm{M}_{\bm{W}^{-1}} \LRp{\bm{I} \otimes \bm{M}^{-1}}.
\label{eq:wadgmc}
\end{align}

We note that, when $\bm{W}$ is constant over $D^k$, $\bm{M}_{\bm{W}}$ reduces to the Kronecker product of the inverse stiffness tensor and the local mass matrix 
\[
\bm{M}_{\bm{W}} = \bm{W}\otimes \bm{M}, \qquad \bm{M}_{\bm{W}}^{-1} = \bm{W}^{-1}\otimes \bm{M}^{-1}.
\] 
In this case, we also have $\bm{M}_{\bm{W}^{-1}} = \bm{W}^{-1}\otimes \bm{M}$, and substituting this explicit inverse into the weight adjusted mass matrix inverse $\LRp{\bm{I} \otimes \bm{M}^{-1}} \bm{M}_{\bm{W}^{-1}} \LRp{\bm{I} \otimes \bm{M}^{-1}}$ in (\ref{eq:wadgmc}) recovers the exact inversion of $\bm{M}_{\bm{W}}$.

\section{An energy stable weight-adjusted discontinuous Galerkin formulation for elastic wave propagation}
\label{sec:wadgelas}
We construct a weight-adjusted DG method by simply replacing the weighted $L^2$ inner products appearing in the left hand side of the local DG formulation (\ref{eq:dgloc}) with weight-adjusted approximations
\begin{align*}
\LRp{T^{-1}_{1/\rho} \pd{\bm{v}}{t},\bm{w}}_{\Lk} &= \LRp{ \sum_{i=1}^d \bm{A}_i^T\pd{\bm{\sigma}}{\bm{x}_i},\bm{w}}_{\Lk}  + \LRa{\frac{1}{2}\bm{A}_n^T\jump{\bm{\sigma}} + \frac{\tau_{\bm{v}}}{2}\bm{A}_n\bm{A}_n^T\jump{\bm{v}},\bm{w}}_{\Ldk}\\
\LRp{T^{-1}_{\bm{C}} \pd{\bm{\sigma}}{t},\bm{q}}_{\Lk} &= \LRp{\sum_{i=1}^d \bm{A}_i \pd{\bm{v}}{\bm{x}_i},\bm{q}}_{\Lk} + \LRa{\frac{1}{2}\bm{A}_n\jump{\bm{v}} + \frac{\tau_{\sigma}}{2}\bm{A}_n^T\bm{A}_n\jump{\bm{\sigma}},\bm{q}}_{\Ldk}.
\end{align*}
Since the right hand side of the WADG formulation is identical to the right hand side of the DG formulation (\ref{eq:dgform}), WADG preserves a variant of the energy stability in Theorem~\ref{thm:stability}  
\begin{align*}
&\sum_{D^k\in \Oh} \frac{1}{2}\pd{}{t}\LRp{\LRp{T^{-1}_{1/\rho} \bm{v},\bm{v}}_{\Lk} + \LRp{T^{-1}_{\bm{C}} \bm{\sigma},\bm{\sigma}}_{\Lk}} \leq - \sum_{f \in\Gamma_h \setminus \partial \Omega} \int_{f} \LRp{\frac{\tau_{\bm{v}}}{2}\LRb{\bm{A}_n\jump{\bm{v}}}^2 + \frac{\tau_{\bm{\sigma}}}{2}\LRb{\bm{A}_n^T\jump{\bm{\sigma}}}^2}\diff{\bm{x}}\nonumber\\
& - \sum_{f\in \Gamma_v} \int_f \LRp{\tau_{\bm{v}}  \LRb{\bm{A}_n\bm{v}^-}^2}\diff{\bm{x}}  - \sum_{f\in \Gamma_\sigma} \int_f \LRp{\tau_{\bm{\sigma}}  \LRb{\bm{A}_n^T\bm{\sigma}^-}^2}\diff{\bm{x}} - \sum_{f\in \Gamma_{\rm abc}} \int_f \LRp{\frac{\tau_{\bm{v}}}{2} \LRb{\bm{A}_n\bm{v}^-}^2 + \frac{\tau_{\bm{\sigma}}}{2} \LRb{\bm{A}_n^T\bm{\sigma}^-}^2 }\diff{\bm{x}}\leq 0.  
\end{align*}

The use of weight-adjusted inner products replaces the weighted $L^2$ mass matrices in (\ref{eq:dgmat}) by their weight-adjusted approximations.  
Inverting these weight-adjusted mass matrices yields the following local system of ODEs for $\bm{V},\bm{\Sigma}$
\begin{align*}
\pd{\bm{V}}{t} &= \LRp{\bm{I}\otimes\bm{M}^{-1}} \bm{M}_{\LRp{\rho^{-1}\bm{I}}} \LRp{\sum_{i=1}^d  \LRp{\bm{A}_i^T\otimes \bm{D}_{i}} \bm{\Sigma} + \sum_{f\in \partial D^k}\LRp{\bm{I}\otimes \bm{L}_f}\bm{F}_{v}} \nonumber\\
\pd{\bm{\Sigma}}{t} &= \LRp{\bm{I}\otimes\bm{M}^{-1}} \bm{M}_{\bm{C}}  \LRp{\sum_{i=1}^d \LRp{\bm{A}_i\otimes \bm{D}_{i}}\bm{V}  +  \sum_{f\in \partial D^k}\LRp{\bm{I}\otimes\bm{L}_f}\bm{F}_{\sigma}}. 
\label{eq:wadgmat}
\end{align*}
In practice, the matrices $\LRp{\bm{I}\otimes\bm{M}^{-1}} \bm{M}_{\LRp{\rho^{-1}\bm{I}}}$ and $\LRp{\bm{I}\otimes\bm{M}^{-1}} \bm{M}_{\bm{C}}$ are applied in a matrix-free fashion using reference element matrices and values of $\rho,\bm{C}$ at quadrature points.  After fusing operations together, this procedure can be boiled down to multiplication by two rectangular matrices \cite{chan2016weight2}.  

One drawback of the analysis presented in this work is that accuracy of the weight-adjusted approximation is not guaranteed in the incompressible limit $\mu / \lambda \rightarrow 0$.  In this case, the stiffness matrix $\bm{C}$ becomes singular, and the constant $\tilde{w}_{\max}$ in the upper bound on $\nor{\bm{C}^{-1}}$ blows up.  However, numerical experiments in Section~\ref{sec:incomp} suggest that, while taking $\mu/\lambda \approx 0$ (or $\bm{C}$ nearly singular) results in larger relative errors for $\bm{\sigma}$, the accuracy of the WADG solution for $\bm{v}$ does not degrade as significantly for near-incompressible materials.  

\subsection{Energy stability on curvilinear meshes}
\label{sec:curvi}
We have shown the energy stability of the DG formulation (\ref{eq:dgform}) on meshes of affine elements by assuming exact integration of all terms.  However, energy stability can still be guaranteed if integrals are evaluated inexactly using quadrature.  Instead of discretizing the ``strong'' DG formulation (\ref{eq:dgform}), we discretize the ``skew-symmetric'' formulation, where the right hand side of the velocity equations is integrated by parts,\footnote{The choice of which equation to integrate by parts is arbitrary, since integrating the stress equations by parts also results in an energy stable formulation. } resulting in the (local) formulation 
\begin{align*}
\int_{\Dhat} T^{-1}_{(\rho J)^{-1}} \pd{\bm{v}}{t}\bm{w} \diff{\widehat{\bm{x}}}&= \int_{\Dhat}\LRp{\sum_{i=1}^d \bm{A}_i^T \pd{\bm{\sigma}}{\bm{x}_i}, {\bm{w}}}J\diff{\widehat{\bm{x}}}  + \sum_{f \in \partial D^k} \int_{\widehat{f}}\LRp{\frac{1}{2}\bm{A}_n^T\jump{\bm{\sigma}} + \frac{\tau_{\bm{v}}}{2}\bm{A}_n^T\bm{A}_n\jump{\bm{v}}}\bm{w}J^f\diff{\widehat{\bm{x}}} \nonumber\\
\int_{\Dhat}  T^{-1}_{\LRp{J^{-1}\bm{C}}} \pd{\bm{\sigma}}{t}\bm{q} \diff{\widehat{\bm{x}}}&= -\int_{\Dhat} \LRp{\sum_{i=1}^d \bm{v},\bm{A}_i^T\pd{\bm{q}}{\bm{x}_i}}J \diff{\widehat{\bm{x}}} + \sum_{f\in \partial D^k}\int_{\widehat{f}}\LRp{\frac{1}{2}\bm{A}_n\avg{\bm{v}} + \frac{\tau_{\sigma}}{2}\bm{A}_n\bm{A}_n^T\jump{\bm{\sigma}}}\bm{q}J^f\diff{\widehat{\bm{x}}}
\end{align*}
where we have incorporated spatial variations of $J$ into the definition of the weights on the left hand side.  

The proof of energy stability in Theorem~\ref{thm:stability} \reviewerOne{follows \cite{chan2016weight2}, requiring} only algebraic manipulations of this formulation.  \reviewerOne{The proof} does not require that integration-by-parts holds discretely.  This implies that a discrete version of energy stability is still guaranteed in the presence of inexact quadrature, where integrals in Theorem~\ref{thm:stability} are replaced with quadrature approximations.  This is especially important for discretizations on curvilinear meshes, where the exact integration of spatially varying geometric factors and Jacobians can be either prohibitively expensive for high order curvilinear mappings or impossible for rational mappings \cite{engvall2016isogeometric, michoski2015foundations}.  

We note that, to ensure energy stability on curved and non-affine elements, the ``skew-symmetric'' formulation must be evaluated explicitly using quadrature, which is typically more expensive than quadrature-free evaluations used to evaluate the ``strong'' DG formulation (\ref{eq:dgform}).  These costs can be slightly reduced for most curvilinear meshes by evaluating the DG formulation using the ``skew-symmetric'' formulation on curvilinear elements and the more efficient ``strong'' DG formulation on affine elements \cite{chan2016weight2}.  

\subsection{Convergence analysis}

Using estimates from Section~\ref{sec:mwadgprop}, we can extend the semi-discrete convergence analysis in \cite{houston2002discontinuous,warburton2013low,chan2016weight1} to linear elastic wave propagation on meshes of affine elements.  Techniques in \cite{warburton2013low} can be used to extend this analysis to curvilinear elements.  

Let $\bm{U},\bm{U}_h$ denote the exact and discrete WADG solutions, respectively.  We will assume that $\bm{U},\pd{\bm{U}}{t}$ are sufficiently regular such that
\[
\nor{\bm{U}}_{W^{N+1,2}\LRp{\Oh}} < \infty, \qquad \nor{\pd{\bm{U}}{t}}_{W^{N+1,2}\LRp{\Oh}} < \infty,
\]
where we define $\nor{\bm{U}}_{W^{N+1,2}\LRp{\Oh}}^2 = \sum_k \nor{\bm{U}}_{W^{N+1,2}\LRp{D^k}}^2$.


In terms of group variables $\bm{U} = \LRp{\bm{v},\bm{\sigma}}$ and $\bm{V} = \LRp{\bm{w},\bm{q}} \in \LRp{V_h\LRp{\Oh}}^d\times\LRp{V_h\LRp{\Oh}}^{N_d}$, the WADG formulation can be written as
\begin{align*}
&\LRp{T^{-1}_{\bm{A}_0^{-1}}\pd{\bm{U}}{t},\bm{V}}_{\L} + a(\bm{U},\bm{V}) + b(\bm{U},\bm{V}) = (\bm{f},\bm{V})\\
a(\bm{U},\bm{V}) &= \sum_{D^k\in \Oh} \LRp{-\LRp{ \sum_{i=1}^d \bm{\sigma},\bm{A}_i \pd{\bm{w}}{\bm{x}_i}}_{\Lk}  + \LRp{\sum_{i=1}^d \bm{A}_i \pd{\bm{v}}{\bm{x}_i},\bm{q}}_{\Lk}} \\
b(\bm{U},\bm{V}) &= \sum_{D^k\in \Oh}\LRp{ \LRa{\bm{A}_n^T\avg{\bm{\sigma}} + \frac{\tau_{\bm{v}}}{2}\bm{A}_n^T\bm{A}_n\jump{\bm{v}},\bm{w}}_{\Ldk} + \LRa{\frac{1}{2}\bm{A}_n\jump{\bm{v}} + \frac{\tau_{\sigma}}{2}\bm{A}_n\bm{A}_n^T\jump{\bm{\sigma}},\bm{q}}_{\Ldk} },
\end{align*}
where $\bm{A}_0(\bm{x}) = {\rm diag}\LRp{\rho(\bm{x})\bm{I}^{d\times d},\bm{C}^{-1}(\bm{x})}$.  Energy stability implies that
\begin{align*}
&b(\bm{U},\bm{U}) = \sum_{D^k\in\Oh} \LRp{\frac{\tau_{\bm{v}}}{2}\nor{\jump{\bm{A}_n\bm{v}}}_{\Ldk}^2 + \frac{\tau_{\bm{\sigma}}}{2}\nor{\jump{\bm{A}^T_n\bm{\sigma}}}_{\Ldk}^2}\\
&\LRp{T^{-1}_{\bm{A}_0^{-1}}\pd{\bm{U}}{t},\bm{U}}_{\L} + b(\bm{U},\bm{U}) =  (\bm{f},\bm{U}).
\end{align*}
Since the DG formulation (\ref{eq:dgform}) is consistent, these solutions satisfy
\begin{align}
\LRp{\bm{A}_0\pd{\bm{U}}{t},\bm{V}}_{\L} + a(\bm{U},\bm{V}) + b(\bm{U},\bm{V}) &= (\bm{f},\bm{V}) \nonumber\\
\LRp{T_{\bm{A}_0^{-1}}^{-1}\pd{\bm{U}_h}{t},\bm{V}}_{\L} + a(\bm{U}_h,\bm{V}) + b(\bm{U}_h,\bm{V}) &= (\bm{f},\bm{V})
\label{eq:dgvswadg}
\end{align}
for all $\bm{V} \in \LRp{V_h\LRp{\Oh}}^d\times\LRp{V_h\LRp{\Oh}}^{N_d}$.  We decompose the error $\bm{U}-\bm{U}_h$ into a projection error $\bm{\epsilon}$ and discretization error $\bm{\eta}$.
\[
\bm{U}-\bm{U}_h  = \LRp{\Pi_N\bm{U}-\bm{U}_h}-\LRp{\Pi_N\bm{U}-\bm{U}} = \bm{\eta} - \bm{\epsilon}.
\]
We assume that $\bm{U}_h(\bm{x},0)$ is the $L^2$ projection of the exact initial condition, such that $\left.\bm{\eta}\right|_{t=0} = 0$.  We also introduce a consistency error $\bm{\delta} = \bm{A}_0\bm{U} - T_{\bm{A}_0^{-1}}^{-1}\bm{U}$ resulting from the approximation of $\bm{A}_0\bm{U}$ by a weight-adjusted inner product
\[
\bm{A}_0\pd{\bm{U}}{t} - T_{\bm{A}_0^{-1}}^{-1}\pd{\bm{U}_h}{t} = \pd{}{t}\LRp{\bm{A}_0\bm{U} - T_{\bm{A}_0^{-1}}^{-1}\bm{U}} + \pd{}{t}T_{\bm{A}_0^{-1}}^{-1}\LRp{\Pi_N\bm{U}-\bm{U}_h} = \pd{\bm{\delta}}{t} + \pd{}{t}\LRp{T_{\bm{A}_0^{-1}}^{-1}\bm{\eta}}
\]  
where we have used that $\bm{T}_{\bm{A}_0^{-1}}^{-1} = \bm{T}_{\bm{A}_0^{-1}}^{-1}\Pi_N$.  Subtracting the DG and WADG formulations in (\ref{eq:dgvswadg}) and setting $\bm{V} = \bm{\eta}$ then yields 
\begin{align}
\frac{1}{2}\pd{}{t}\LRp{T^{-1}_{\bm{A}_0^{-1}}\bm{\eta},\bm{\eta}}_{\L} + b(\bm{\eta},\bm{\eta}) &= \LRp{\reviewerTwo{-}\pd{\bm{\delta}}{t},\bm{\eta}}_{\L} + a(\bm{\epsilon},\bm{\eta}) + b(\bm{\epsilon},\bm{\eta}),
\label{eq:erreqn}
\end{align}
where we have used that $a(\bm{\eta},\bm{\eta}) = 0$ by skew symmetry.  

We bound $a(\bm{\epsilon},\bm{\eta}) + b(\bm{\epsilon},\bm{\eta})$ on the right hand side by integrating by parts the stress equation and using the component-wise $L^2$ orthogonality of $\bm{\epsilon}$ to derivatives of $\bm{\eta}$.  This reduces the term $a(\bm{\epsilon},\bm{\eta}) + b(\bm{v},\bm{\eta})$ to surface contributions over each element, which can be combined with contributions from neighboring elements to yield
\begin{align*}
 \sum_{D^k\in \Oh}&\LRp{ \LRa{\bm{A}_n^T\avg{\bm{\epsilon}_\sigma} + \frac{\tau_{\bm{v}}}{2}\bm{A}_n^T\bm{A}_n\jump{\bm{\epsilon}_v},\bm{\eta}_v}_{\Ldk} + \LRa{\bm{A}_n\avg{\bm{\epsilon}_v} + \frac{\tau_{\sigma}}{2}\bm{A}_n\bm{A}_n^T\jump{\bm{\epsilon}_\sigma},\bm{\eta}_\sigma}_{\Ldk}}\\
&= \frac{1}{2}\sum_{D^k\in \Oh}\LRp{ 
\LRa{\avg{\bm{\epsilon}_\sigma} - \frac{\tau_{\bm{v}}}{2}\bm{A}_n\jump{\bm{\epsilon}_v},\bm{A}_n\jump{\bm{\eta}_v}}_{\Ldk} + \LRa{\avg{\bm{\epsilon}_v} - \frac{\tau_{\bm{\sigma}}}{2}\bm{A}_n^T\jump{\bm{\epsilon}_\sigma},\bm{A}_n^T\jump{\bm{\eta}_{\sigma}}}_{\Ldk}}\\ 
&\leq \frac{1}{2}\sum_{D^k\in \Oh}\nor{\avg{\bm{\epsilon}_\sigma} - \frac{\tau_{\bm{v}}}{2}\bm{A}_n\jump{\bm{\epsilon}_v}}_{\Ldk} \nor{\bm{A}_n\jump{\bm{\eta}_v}}_{\Ldk} + \nor{\avg{\bm{\epsilon}_v} - \frac{\tau_{\bm{\sigma}}}{2}\bm{A}_n^T\jump{\bm{\epsilon}_\sigma}}_{\Ldk}\nor{\bm{A}_n^T\jump{\bm{\eta}_{\sigma}}}_{\Ldk}\\ 
&\leq C_\tau \sum_{D^k\in \Oh}\nor{\bm{\epsilon}}_{\Ldk} \LRp{\frac{\tau_{\bm{v}}}{2}\nor{\bm{A}_n\jump{\bm{\eta}_v}}_{\Ldk}^2 + \frac{\tau_{\bm{\sigma}}}{2}\nor{\bm{A}_n^T\jump{\bm{\eta}_{\sigma}}}_{\Ldk}^2}^{1/2}
\end{align*}
where $C_\tau$ is proportional to $\max\LRp{\tau_{\bm{v}}, \tau_{\bm{\sigma}}}$.  Using Young's inequality with $\alpha = C_\tau / 2$ yields the following bound 
\[
\LRb{a(\bm{\epsilon},\bm{\eta}) + b(\bm{\epsilon},\bm{\eta})} \leq {b(\bm{\eta},\bm{\eta})} + \frac{C_\tau^2 }{4} \sum_{D^k\in \Oh}\nor{\bm{\epsilon}}_{\Ldk}^2.
\]
Applying this to (\ref{eq:erreqn}) and using Cauchy-Schwarz on $\LRp{\pd{\bm{\delta}}{t},\bm{\eta}}_{\L}$ then yields
\[
\frac{1}{2}\pd{}{t}\LRp{T^{-1}_{\bm{A}_0^{-1}}\bm{\eta},\bm{\eta}}_{\L} + b(\bm{\eta},\bm{\eta}) \leq \nor{\pd{\bm{\delta}}{t}}_{\L} \nor{\bm{\eta}}_{\L} +  b(\bm{\eta},\bm{\eta}) + \sum_{D^k \in \Oh}\frac{ C_\tau^2}{4} \nor{\bm{\epsilon}}_{\Ldk}^2.
\]
We eliminate factors of $b(\bm{\eta},\bm{\eta})$ on both sides and bound right hand side terms.  The trace term is bounded using a standard $hp$ trace inequality \cite{warburton2003constants} and an interpolation estimate
\[
\sum_{D^k \in \Oh} \nor{\bm{\epsilon}}_{\Ldk}^2 \leq \sum_{D^k \in \Oh} C h^{-1} \nor{\bm{\epsilon}}_{\Lk}^2 = Ch^{-1} \nor{\bm{\epsilon}}^2_{\L} \leq Ch^{2N+1} \nor{\bm{U}}^2_{W^{N+1,2}}.  
\]
Since $\bm{A}_0$ is independent of $t$, Theorem~\ref{thm:west} gives
\[
\nor{\pd{\bm{\delta}}{t}}_{\L} = \nor{\bm{A}_0 \pd{\bm{U}}{t} - T_{\bm{A}_0^{-1}}^{-1}\pd{\bm{U}}{t}}_{\L} \leq C h^{N+1} {A}_{\max} \nor{\bm{A}_0}_{W^{N+1,\infty}(\Oh)}\nor{\pd{\bm{U}}{t}}_{\L}\reviewerTwo{,}
\]
where $A_{\max} = \max\LRp{\frac{\rho_{\max}}{\rho_{\min}}, c_{\max} \tilde{c}_{\max}}$.  Then, integrating from $[0,T]$ and using Lemma~\ref{lemma:props} yields 
\begin{equation}
A_{\min} \nor{\bm{\eta}}^2_{\L} \leq C\int_0^T  h^{N+1}A_{\max}\nor{\bm{A}_0}_{W^{N+1,\infty}(\Oh)}\nor{\pd{\bm{U}}{t}}_{\L} \nor{\bm{\eta}}_{\L} + h^{2N+1} \nor{\bm{U}}^2_{W^{N+1,2}(\Omega)} \diff{t}\reviewerTwo{,}
\label{eq:intermed}
\end{equation}
where $A_{\min} =\min\LRp{ \rho_{\min}, \tilde{c}_{\min}}$.  Applying the modified Gronwall's inequality (Lemma 1.10 in \cite{dolejvsi2015discontinuous}; see also \cite{chan2016weight1}) to (\ref{eq:intermed}) then yields
\begin{align*}
\nor{\bm{\eta}}_{\L} &\leq \frac{1}{A_{\min}} \int_0^T CA_{\max} h^{N+1}\nor{\bm{A}_0}_{W^{N+1,\infty}(\Oh)}\nor{\pd{\bm{U}}{t}}_{\L}\diff{t} + \sup_{t\in [0,T]} \sqrt{\int_0^T C h^{2N+1} \nor{\bm{U}}^2_{W^{N+1,2}(\Omega)}\diff{t}}\\
&\leq \frac{C T h^{N+1/2}}{A_{\min}} \sup_{t \in [0,T]}\LRp{h^{1/2} A_{\max}\nor{\bm{A}_0}_{W^{N+1,\infty}(\Oh)}\nor{\pd{\bm{U}}{t}}_{\L}  +  \nor{\bm{U}}_{W^{N+1,2}(\Oh)}}.
\end{align*}
The triangle inequality gives the final estimate
\[
\nor{\bm{U}-\bm{U}_h} \leq \LRp{C_1 + C_2 T} h^{N + 1/2}\sup_{t \in [0,T]}\LRp{ h^{1/2}\nor{\bm{A}_0}_{W^{N+1,\infty}(\Oh)}\nor{\pd{\bm{U}}{t}}_{W^{N+1,2}(\Oh)}+\nor{\bm{U}}_{W^{N+1,2}(\Oh)}}\reviewerTwo{,}
\]
where $C_2$ depends on $A_{\min}, A_{\max}$.  \reviewerThree{From this estimate, we expect $L^2$ errors to decrease proportionally to $O(h^{N+1/2})$ under mesh refinement, which mirrors theoretical results given in \cite{warburton2013low, chan2016weight1}.  Optimal rates of $O(h^{N+1})$ are often observed in practice.  However, we do also observe $O(h^{N+1/2})$ rates of convergence for $N = 1,\ldots,5$ for certain problems, which suggests that the theoretical estimate is tight. }

\section{Numerical experiments}
\label{sec:numerical}
The following sections present several numerical experiments validating the stability and accuracy of the proposed method in two and three dimensions.  The energy stability of the method is also confirmed for examples with sub-cell variations in heterogeneous media and curvilinear meshes.  The convergence of the new DG formulation in piecewise constant isotropic media is confirmed using analytic solutions, while the convergence of the method for high order approximations of heterogeneous media is confirmed using a fine grid reference solution.  Finally, the method is applied to problems with anisotropy and stiffness matrices $\bm{C}$ with sub-element variations.  

In all experiments, we follow \cite{chan2016weight1} and compute application of weight-adjusted mass matrices using a quadrature exact for polynomials of degree $(2N+1)$ \cite{xiao2010quadrature}.  Time integration is performed using the low-storage 4th order five-stage Runge-Kutta scheme of Carpenter and Kennedy \cite{carpenter1994fourth}, and the time-step is chosen based on the global \reviewerOne{estimate}
\begin{equation}
dt = \min_{k} \frac{C_{\rm CFL}}{  \sup_{\bm{x}\in\Omega}\nor{\bm{C}(\bm{x})}_{2} C_N \nor{J^f}_{L^{\infty}\LRp{\partial D^k}}\nor{J^{-1}}_{L^{\infty}\LRp{D^k}}}, 
\label{eq:dt}
\end{equation}
where $C_N = O(N^2)$ is the order-dependent constant in the surface polynomial trace inequality \cite{chan2015gpu} and $C_{\rm CFL}$ is a tunable global CFL constant.  \reviewerOne{This estimate is derived by bounding the eigenvalues of the spatial DG discretization matrix appearing in the semi-discrete system of ODEs.  We note that the usual factor of $h$ arises through the term 
\[
\frac{1}{\nor{J^f}_{L^{\infty}\LRp{\partial D^k}}\nor{J^{-1}}_{L^{\infty}\LRp{D^k}}} = O(h)
\]
due to the fact that $\nor{J^{-1}}_{L^{\infty}\LRp{D^k}} = O(h^{-d})$ and $\nor{J^f}_{L^{\infty}\LRp{D^k}} = O(h^{d-1})$ in $d$ dimensions.  
} 

Finally, in all following experiments, we use $\tau_{\bm{v}} = \tau_{\bm{\sigma}} = 1$ \reviewerOne{and $C_{\rm CFL} = 1$} unless specified otherwise. \reviewerOne{We have arbitrarily chosen these parameters for simplicity, though a more nuanced choice of penalty parameters and CFL constant may improve numerical and computational performance for certain problems.  We note that because $dt$ is derived through an upper bound on the spectral radius of the discretization matrix, this estimate of the timestep is rather conservative, and we have observed that it is possible to take $C_{\rm CFL} > 1$ without losing stability under our choice of timestepping scheme. }

\subsection{Spectra and choice of penalty parameter}
\label{sec:scaling}

We first verify the energy stability of the WADG method for arbitrary heterogeneous media.  We use a stiffness matrix constructed using similarity transforms, such that at every quadrature point, $\bm{C}(\bm{x}) = \bm{U}\bm{D}\bm{U}^{T}$, where $\bm{D}$ is a diagonal matrix with random positive entries $d_{\min} \leq \bm{D}_{ii} \leq d_{\max}$ and $\bm{U}$ is a random unitary matrix.  Let $\bm{A}_h$ denote the matrix induced by the global semi-discrete DG formulation, such that the time evolution of the solution $\bm{v},\bm{\sigma}$ is governed by
\[
\pd{\bm{Q}}{t} = \bm{A}_h\bm{Q},
\]
with $\bm{Q}$ denotes a vector of degrees of freedom for $(\bm{v},\bm{\sigma})$.  Figure~\ref{fig:spectra} shows computed eigenvalues of $\bm{A}_h$ for $[d_{\min},d_{\max}] = [1/10,1]$ and $[\reviewerThree{10^{-5}},1]$, $\tau_{\bm{v}}=\tau_{\bm{\sigma}} = 0$ and $\tau_{\bm{v}}=\tau_{\bm{\sigma}} = 1$ under discretization parameters $N=4$ and $h = 1/4$.  In both cases, the largest real part of any eigenvalue is $O(10^{-14})$, verifying the energy stability of the WADG discretization for arbitrary media.  

\begin{figure}
\centering
\subfloat[$\tau_{\bm{v}} = \tau_{\bm{\sigma}} = 0$]{\includegraphics[width=.475\textwidth]{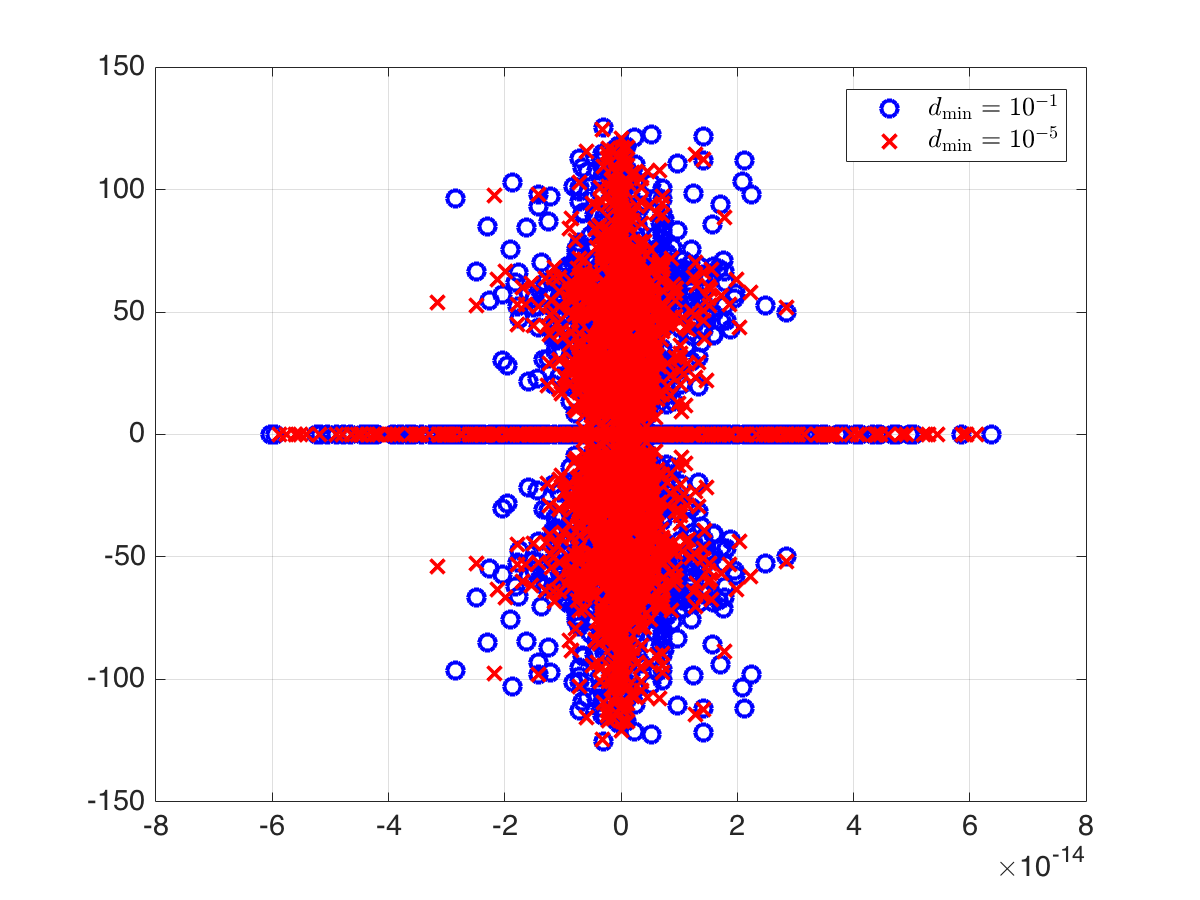}}
\subfloat[$\tau_{\bm{v}} = \tau_{\bm{\sigma}} = 1$]{\includegraphics[width=.475\textwidth]{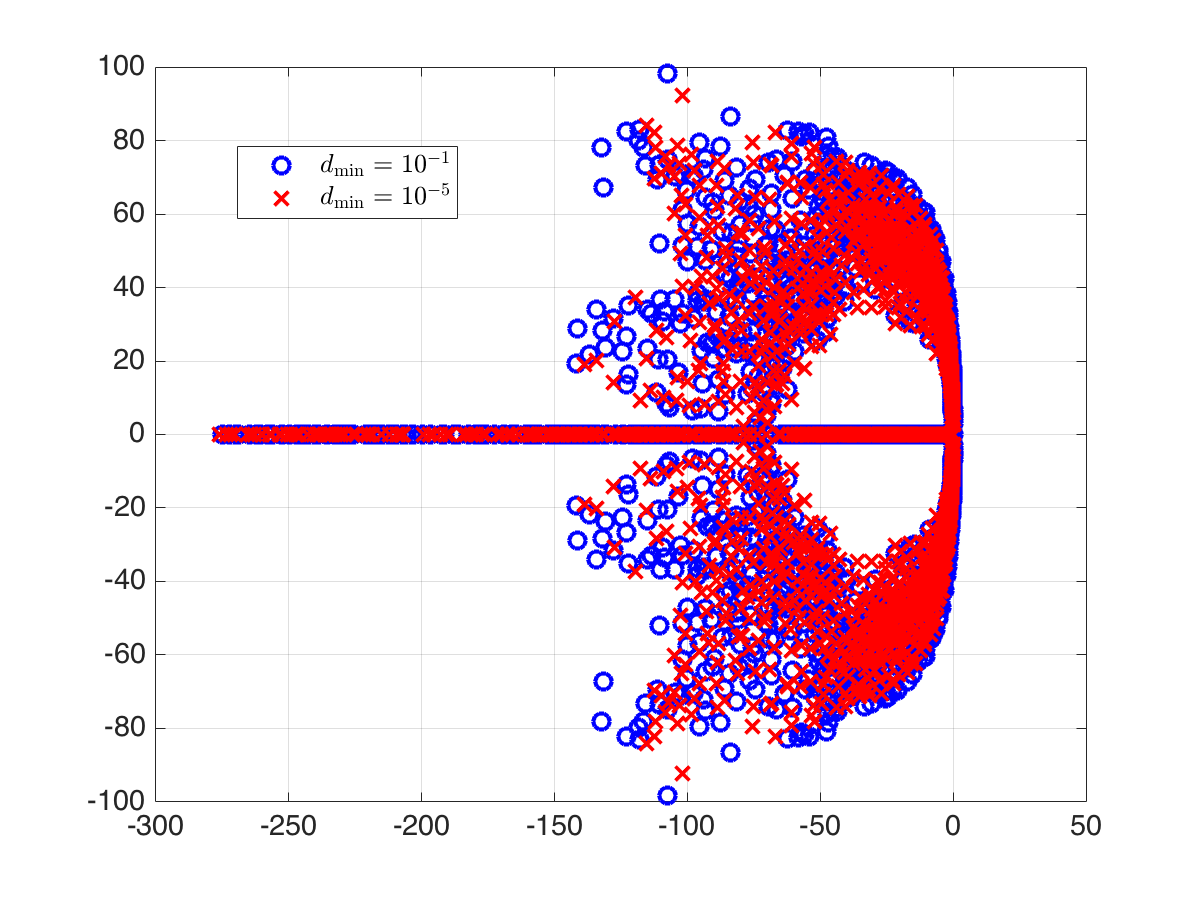}}
\caption{Spectra for $N=4$ and $h = 1/4$, using randomly chosen $\bm{C}(\bm{x})$ with eigenvalues between $[d_{\min}, 1]$ at each quadrature point.  For both $\tau_{\bm{v}} = \tau_{\bm{\sigma}} = 0$ and $\tau_{\bm{v}} = \tau_{\bm{\sigma}} = 1$, the largest real part of the spectra is $O(10^{-14})$.}
\label{fig:spectra}
\end{figure}

For practical simulations, the choice of $\tau_{\bm{v}}, \tau_{\bm{\sigma}}$ remains to be specified.  Taking $\tau_{\bm{v}},\tau_{\bm{\sigma}} > 0$ results in damping of under-resolved spurious components of the solution \cite{chan2016short}; however, a naive selection of these penalty parameters can result in an overly restrictive time-step restriction for stability.  We wish to choose $\tau_{\bm{v}},\tau_{\bm{\sigma}}$ as large as possible without increasing \reviewerOne{the value of $\nor{\bm{A}_h}$ when using a central flux (i.e.\ $\tau_{\bm{v}} = \tau_{\bm{\sigma}} = 0$).  For example, in Figure~\ref{fig:spectra}, we observe that the value of $\nor{\bm{A}_h}$ for a central flux  is roughly half as large as the value of $\nor{\bm{A}_h}$ when taking $\tau_{\bm{v}} = \tau_{\bm{\sigma}} = 1$.  We note that the growth in $\nor{\bm{A}_h}$ when $\tau_{\bm{v}} = \tau_{\bm{\sigma}} = 1$ is due to the large negative real part of the extremal eigenvalues of $\bm{A}_h$, which mirrors observations in \cite{chan2016short} that a subset of the eigenvalues of $\bm{A}_h$ approach $-\infty$ as the penalty parameter increases.  

Initial numerical experiments suggest that varying the penalty parameters spatially and scaling $\tau_{\bm{v}}$ and $\tau_{\bm{\sigma}}$ independently of each other can offset the artificial stiffness induced by a naive choice of penalty parameters.  For example, material coefficients can be taken into account by scaling the penalty parameters such that flux terms are dimensionally consistent.  One example of such a scaling is
\[
\tau_{\bm{v}} = \gamma_v \sup_{\bm{x} \in f} \sqrt{\nor{\avg{\bm{C}(\bm{x})}}\avg{\rho(\bm{x})}}, \qquad \tau_{\bm{\sigma}} =  \gamma_\sigma  \sup_{\bm{x} \in f} \frac{1}{\sqrt{\nor{\avg{\bm{C}(\bm{x})}}\avg{\rho(\bm{x})}}},
\]  
where the supremum is taken locally over each face and $\gamma_v, \gamma_\sigma$ are dimensionless constants.  We note that the optimal choice of scaling is outside of the scope of this current paper, and will be explored in future work.
}

\subsection{Analytic solutions} 

Next, we study the accuracy and convergence of weight-adjusted DG method for several analytical solutions in linear elasticity.  In all cases, the solution is expressed \reviewerOne{in} terms of the displacement vector $\bm{u}(\bm{x},t) = (u_1,\ldots,u_d)$.  Initial conditions for velocity and stress are computed through 
\[
\bm{v}(\bm{x},t) = \pd{\bm{u}}{t}, \qquad \bm{\sigma} = \bm{C}\frac{1}{2}\LRp{\Grad\bm{u} + \Grad\bm{u}^T}.  
\]
Unless otherwise stated, we report relative $L^2$ errors for all components of the solution $\bm{U} = (\bm{v},\bm{\sigma})$ 
\[
\frac{\nor{\bm{U}-\bm{U}_h}_{L^2\LRp{\Omega}}}{\nor{\bm{U}}_{L^2\LRp{\Omega}}} = \frac{\LRp{\sum_{i=1}^m \nor{\bm{U}_i-\bm{U}_{i,h}}^2_{L^2\LRp{\Omega}}}^{1/2}}{\LRp{\sum_{i=1}^m\nor{\bm{U}_i}^2_{L^2\LRp{\Omega}}}^{1/2}}.
\]

\subsubsection{Harmonic oscillation of a square} 

We first examine convergence  on a unit square domain with  $\lambda = \mu = \rho = 1$.   The components of the displacement vector are given by
\begin{align*}
u_1(x,y,t) &= \cos(\omega\pi t)\cos(\pi x)\sin(\pi y)\\
u_2(x,y,t) &= -\cos(\omega\pi t)\sin(\pi x)\cos(\pi y),
\end{align*}
where $\omega = \sqrt{2\mu}$.  Zero traction boundary conditions are imposed.  Figure~\ref{fig:harmonic_osc} shows $L^2$ errors computed at time $T = 5$, using uniform triangular meshes constructed by bisecting a uniform mesh of quadrilaterals along the diagonal.

 \reviewerOne{
For $N = 1,\ldots, 5$, $O(h^{N+1})$ rates of convergence are observed when using the penalty flux with $\tau_{\bm{v}} = \tau_{\bm{\sigma}} = 1$.  When using a central flux (with $\tau_{\bm{v}} = \tau_{\bm{\sigma}} = 0$), we observe a so-called ``even-odd'' pattern \cite{hesthaven2007nodal, mercerat2015nodal}, where the convergence rate is $O(h^N)$ for $N$ odd and between $O(h^{N+1/2})$ and $O(h^{N+1})$ for $N$ even.  This behavior improves upon the theoretical estimate derived in \cite{delcourte2015analysis}, which is $O(h^N)$ for a sufficiently small time-step size.  
 
 For quasi-uniform meshes, the optimal rate of convergence of spatial $L^2$ errors under uniform mesh refinement is $O(h^{N+1})$ \cite{hesthaven2007nodal}, which is greater than the $O(h^{N+1/2})$ rate of convergence which can be proven for dissipative DG discretizations on general meshes \cite{cockburn2008optimal} (though optimal rates of convergence are often observed in numerical experiments).  We note that for $N = 4$ and $N=5$, we observe results for both fluxes which are better than the 4th order accuracy of our time-stepping scheme.  This is most likely due to the benign nature of the solution and the choice of timestep (\ref{eq:dt}), which scales as $O(h/N^2)$.   For $N=4, 5$, the results of Figure~\ref{fig:harmonic_osc} suggest that the resulting timestep is small enough such that temporal errors of $O(dt^4)$ are small relative to spatial discretization errors of $O(h^{N+1})$. }

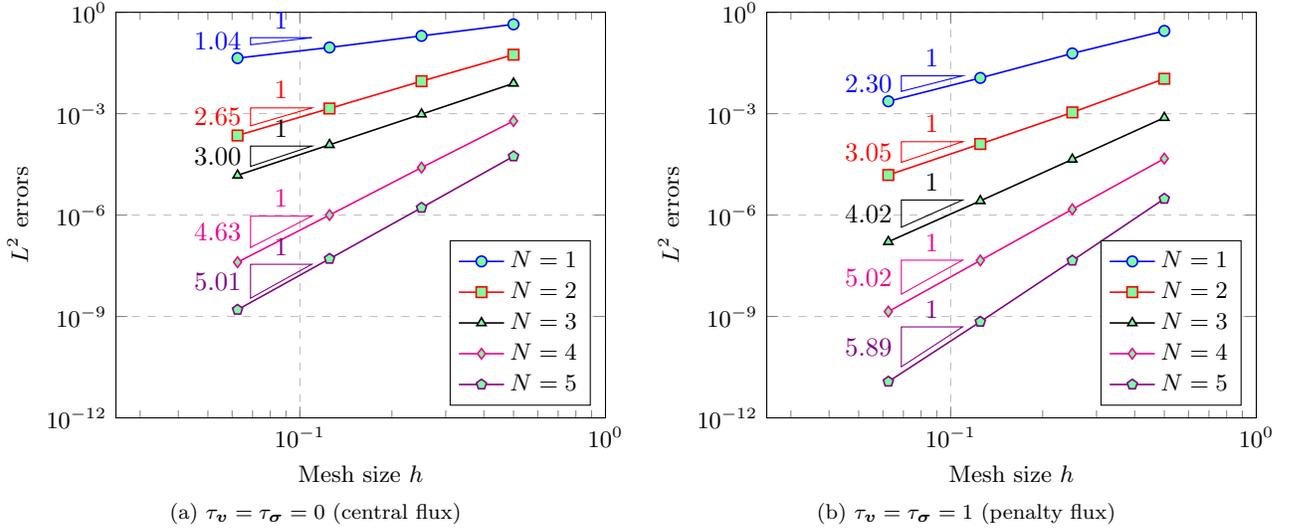
\begin{figure}
\centering
\subfloat[\reviewerOne{$\tau_{\bm{v}} = \tau_{\bm{\sigma}} = 0$ (central flux)}]{
\begin{tikzpicture}
\begin{loglogaxis}[
	legend cell align=left,
	width=.49\textwidth,
    xlabel={Mesh size $h$},
       ylabel={$L^2$ errors}, 
    xmin=.025, xmax=1,
    ymin=1e-12, ymax=1,
    legend pos=south east,
    xmajorgrids=true,
    ymajorgrids=true,
    grid style=dashed,
] 


\addplot[color=blue,mark=*,semithick, mark options={fill=markercolor}]
coordinates{(0.5,0.439397)(0.25,0.200219)(0.125,0.0909291)(0.0625,0.0440705)};
\logLogSlopeTriangleFlip{0.4}{0.125}{0.92}{1.04}{blue}

\addplot[color=red,mark=square*,semithick, mark options={fill=markercolor}]
coordinates{(0.5,0.0555882)(0.25,0.00915474)(0.125,0.0014232)(0.0625,0.000226124)};
\logLogSlopeTriangleFlip{0.4}{0.125}{0.72}{2.65}{red}

\addplot[color=black,mark=triangle*,semithick, mark options={fill=markercolor}]
coordinates{(0.5,0.00789363)(0.25,0.000967494)(0.125,0.000120029)(0.0625,1.499e-05)};
\logLogSlopeTriangleFlip{0.4}{0.125}{0.62}{3.00}{black}

\addplot[color=magenta,mark=diamond*,semithick, mark options={fill=markercolor}]
coordinates{(0.5,0.000607677)(0.25,2.54019e-05)(0.125,1.00184e-06)(0.0625,4.03515e-08)};
\logLogSlopeTriangleFlip{0.4}{0.125}{0.42}{4.63}{magenta}

\addplot[color=violet,mark=pentagon*,semithick, mark options={fill=markercolor}]
coordinates{(0.5,5.47387e-05)(0.25,1.64848e-06)(0.125,5.08294e-08)(0.0625,1.58192e-09)};
\logLogSlopeTriangleFlip{0.4}{0.125}{0.295}{5.01}{violet}

\legend{$N=1$,$N=2$,$N=3$,$N=4$,$N=5$}
\end{loglogaxis}
\end{tikzpicture}
}
\subfloat[$\tau_{\bm{v}} = \tau_{\bm{\sigma}} = 1$ (penalty flux)]{
\begin{tikzpicture}
\begin{loglogaxis}[
	legend cell align=left,
	width=.49\textwidth,
    xlabel={Mesh size $h$},
       ylabel={$L^2$ errors}, 
    xmin=.025, xmax=1,
    ymin=1e-12, ymax=1,
    legend pos=south east,
    xmajorgrids=true,
    ymajorgrids=true,
    grid style=dashed,
] 

\addplot[color=blue,mark=*,semithick, mark options={fill=markercolor}]
coordinates{(0.5,0.281186)(0.25,0.0607224)(0.125,0.0114231)(0.0625,0.00232631)};
\logLogSlopeTriangleFlip{0.4}{0.125}{0.805}{2.30}{blue}

\addplot[color=red,mark=square*,semithick, mark options={fill=markercolor}]
coordinates{(0.5,0.010807)(0.25,0.00109063)(0.125,0.000126864)(0.0625,1.53214e-05)};
\logLogSlopeTriangleFlip{0.4}{0.125}{0.63}{3.05}{red}

\addplot[color=black,mark=triangle*,semithick, mark options={fill=markercolor}]
coordinates{(0.5,0.000760654)(0.25,4.44169e-05)(0.125,2.63258e-06)(0.0625,1.62546e-07)};
\logLogSlopeTriangleFlip{0.4}{0.125}{0.47}{4.02}{black}

\addplot[color=magenta,mark=diamond*,semithick, mark options={fill=markercolor}]
coordinates{(0.5,4.68589e-05)(0.25,1.48545e-06)(0.125,4.55009e-08)(0.0625,1.39823e-09)};
\logLogSlopeTriangleFlip{0.4}{0.125}{0.305}{5.02}{magenta}

\addplot[color=violet,mark=pentagon*,semithick, mark options={fill=markercolor}]
coordinates{(0.5,3.03577e-06)(0.25,4.50412e-08)(0.125,6.98053e-10)(0.0625,1.17913e-11)};
\logLogSlopeTriangleFlip{0.4}{0.125}{0.126}{5.89}{violet}

\legend{$N=1$,$N=2$,$N=3$,$N=4$,$N=5$}
\end{loglogaxis}
\end{tikzpicture}
}
\caption{Convergence of $L^2$ errors for harmonic oscillation.}
\label{fig:harmonic_osc}
\end{figure}

\subsubsection{Rayleigh and Lamb waves}  
\label{sec:rlamb}
Next, we examine the convergence of WADG for Rayleigh and Lamb waves, both of which test the imposition of traction-free boundary conditions.  

Rayleigh waves are elastic surface waves which decay exponentially away from the surface.  These waves are given by the displacement vector
\begin{align*}
\bm{u}\LRp{x,y,t} &= e^{-\omega x \sqrt{1-\xi^2}}
\LRp{
\begin{array}{c}
\cos(\omega(y+c_r t)) \\
\sqrt{1-\xi^2}\sin(\omega(y+c_r t))
\end{array}
}\\
&+ \LRp{\frac{\xi^2}{2}-1}e^{-\omega x\sqrt{1-\frac{\xi^2\mu}{2\mu + \lambda}}}
\LRp{
\begin{array}{c}
\cos(\omega(y + c_r t))\\
{\sin(\omega(y+c_r t))}/{\sqrt{1-\frac{\xi^2\mu}{2\mu + \lambda}}}
\end{array}
},
\end{align*}
where $\omega$ is the wavespeed, $c_r$ is the Rayleigh phase velocity $c_r = \xi \sqrt{\mu}$, and $\xi$ satisfies
\[
\sqrt{1-\xi^2}\sqrt{1 - \frac{\xi^2\mu}{2\mu+\lambda}} - \LRp{\frac{\xi^2}{2}-1}^2 = 0.
\]
In our computations, we use $\rho = \mu = \lambda = 1$, $\xi = 0.949554083888034$, and $\omega = 2\pi$ \cite{wilcox2010high}.  We solve on the domain $[0,2]\times[0,1]$ using a sequence of uniform triangular meshes, and enforce traction-free boundary conditions at $x=0$ and exact Dirichlet boundary conditions at $x=2$.  Periodic boundary conditions are applied at $y = 0$ and $y = 1$.  

Lamb waves are supported by elastic waveguides with traction-free (free surface) boundary conditions at the top and bottom of the domain.  The displacement of these waves is given by
\begin{align*}
{u}_1\LRp{x,y,t} &= \LRp{-kB_1 \cos(p y)- qB_2 \cos(q y)} \sin(k x - \omega t)\\
{u}_2\LRp{x,y,t} &= \LRp{-pB_1 \sin(py) + kB_2 \sin(qy)} \cos(kx - \omega t)
\end{align*}
where $k$ is the wavenumber and $\omega$ is the frequency, and the constants $p$ and $q$ are defined as
\[
p^2 = \frac{\omega^2}{2\mu + \lambda} - k^2, \qquad q^2 = \frac{\omega^2}{\mu} - k^2.
\]
The wavenumber $k$ and frequency $\omega$ are related through a dispersion relation.  The ratio of the amplitudes $B_1/ B_2$ can be determined using other parameters, implying that $B_1, B_2$ are unique up to a scaling constant.  In our experiments, we use $\rho = \mu = 1$, $\lambda = 2$, $k = 2\pi$.  For these values, $\omega = 13.137063197233$, $B_1 = 126.1992721468$ and $B_2 = 53.88807700007$ \cite{wilcox2010high}.  We solve on the domain $[-1,1]\times[-1/2,1/2]$, with traction-free boundary conditions at $y = \pm 1/2$ and periodic boundary conditions at $x = \pm 1$.  

\reviewerOne{Figures~\ref{fig:rayleigh} and \ref{fig:lamb} show $L^2$ errors for Rayleigh and Lamb waves at time $T = 5$, respectively.  As with the harmonic oscillation solution, both central and penalty fluxes are considered.  For penalty fluxes, the computed convergence rates fall between the optimal rate of $O(h^{N+1})$ and theoretical rate of $O(h^{N+1/2})$ \cite{johnson1986analysis}.  For central fluxes, we observe the odd-even pattern for Lamb waves.  However, for Rayleigh waves, we observe the theoretical $O(h^N)$ rate of convergence.

As with the harmonic oscillation problem, the accuracy of the numerical method is theoretically limited by the 4th order accuracy of the time-stepping scheme.  The observed higher order accuracy for $N=4$ and $N=5$ again suggests that the solution is smooth in time and the time-step is small enough to render temporal discretization errors small relative to spatial discretization errors.   }
\begin{figure}
\centering
\subfloat[\reviewerOne{Central flux}]{
\begin{tikzpicture}
\begin{loglogaxis}[
	legend cell align=left,
	width=.475\textwidth,
    xlabel={Mesh size $h$},
    xmin=.025, xmax=1,
    ymin=5e-9, ymax=1.5,
    legend pos=south east,
    xmajorgrids=true,
    ymajorgrids=true,
    grid style=dashed,
] 

\addplot[color=blue,mark=*,semithick, mark options={fill=markercolor}]
coordinates{(0.5,1.20336)(0.25,1.14525)(0.125,0.621593)(0.0625,0.22184)};
\logLogSlopeTriangle{.425}{0.11}{0.9}{1.49}{blue}

\addplot[color=red,mark=square*,semithick, mark options={fill=markercolor}]
coordinates{(0.5,0.916679)(0.25,0.104515)(0.125,0.0201473)(0.0625,0.00488285)};
\logLogSlopeTriangleFlip{.38}{0.1}{0.735}{2.04}{red}

\addplot[color=black,mark=triangle*,semithick, mark options={fill=markercolor}]
coordinates{(0.5,0.160567)(0.25,0.0153951)(0.125,0.00219232)(0.0625,0.000289318)};
\logLogSlopeTriangleFlip{.38}{0.1}{0.595}{2.92}{black}

\addplot[color=magenta,mark=diamond*,semithick, mark options={fill=markercolor}]
coordinates{(0.5,0.0318274)(0.25,0.00183785)(0.125,0.000115977)(0.0625,6.94596e-06)};
\logLogSlopeTriangleFlip{.38}{0.1}{0.41}{4.06}{magenta}

\addplot[color=violet,mark=pentagon*,semithick, mark options={fill=markercolor}]
coordinates{(0.5,0.00476904)(0.25,0.000209483)(0.125,8.28278e-06)(0.0625,2.77453e-07)};
\logLogSlopeTriangleFlip{0.37}{0.1}{0.245}{4.90}{violet}

\legend{$N=1$,$N=2$,$N=3$,$N=4$,$N=5$}
\end{loglogaxis}
\end{tikzpicture}
}
\subfloat[Penalty flux]{
\begin{tikzpicture}
\begin{loglogaxis}[
    legend cell align=left,
    width=.475\textwidth,
    xlabel={Mesh size $h$},
    ylabel={$L^2$ errors}, 
    xmin=.025, xmax=1,
    ymin=5e-9, ymax=1.5,
    legend pos=south east,
    xmajorgrids=true,
    ymajorgrids=true,
    grid style=dashed,
] 

\addplot[color=blue,mark=*,semithick, mark options={fill=markercolor}]
coordinates{(0.5,0.987525)(0.25,0.856084)(0.125,0.272343)(0.0625,0.0495455)};
\logLogSlopeTriangleFlip{.39}{0.11}{0.855}{2.46}{blue}

\addplot[color=red,mark=square*,semithick, mark options={fill=markercolor}]
coordinates{(0.5,0.762022)(0.25,0.0822137)(0.125,0.00510445)(0.0625,0.000517961)};
\logLogSlopeTriangleFlip{.38}{0.1}{0.625}{3.30}{red}

\addplot[color=black,mark=triangle*,semithick, mark options={fill=markercolor}]
coordinates{(0.5,0.133471)(0.25,0.00400974)(0.125,0.000280342)(0.0625,1.93442e-05)};
\logLogSlopeTriangleFlip{.38}{0.1}{0.47}{3.86}{black}

\addplot[color=magenta,mark=diamond*,semithick, mark options={fill=markercolor}]
coordinates{(0.5,0.0130378)(0.25,0.000405004)(0.125,1.36821e-05)(0.0625,4.57998e-07)};
\logLogSlopeTriangleFlip{.38}{0.1}{0.28}{4.90}{magenta}

\addplot[color=violet,mark=pentagon*,semithick, mark options={fill=markercolor}]
coordinates{(0.5,0.00157091)(0.25,3.42541e-05)(0.125,5.7373e-07)(0.0625,9.09776e-09)};
\logLogSlopeTriangleFlip{.38}{0.1}{0.085}{5.98}{violet}

\legend{$N=1$,$N=2$,$N=3$,$N=4$, $N=5$}
\end{loglogaxis}
\end{tikzpicture}
}
\caption{Convergence of $L^2$ errors for the Rayleigh wave solution.}
\label{fig:rayleigh}
\end{figure}
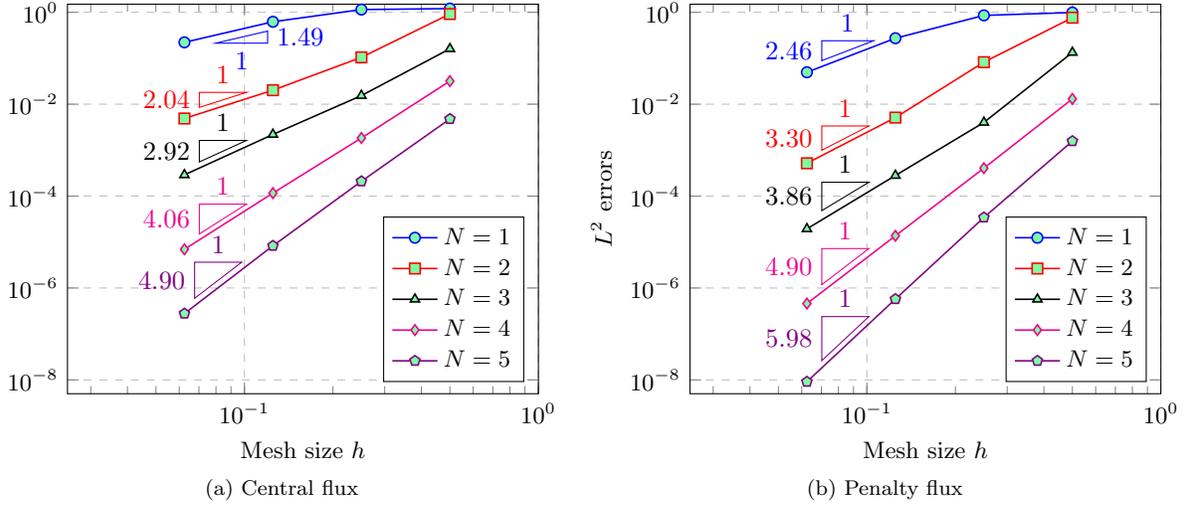

\begin{figure}
\centering
\subfloat[\reviewerOne{Central flux}]{
\begin{tikzpicture}
\begin{loglogaxis}[
    legend cell align=left,
    width=.475\textwidth,
    xlabel={Mesh size $h$},
    ylabel={$L^2$ errors}, 
    xmin=.025, xmax=1,
    ymin=1e-7, ymax=2.5,
     legend pos=south east,
    xmajorgrids=true,
    ymajorgrids=true,
    grid style=dashed,
] 


\addplot[color=blue,mark=*,semithick, mark options={fill=markercolor}]
coordinates{(0.5,1.43804)(0.25,1.5293)(0.125,1.11806)(0.0625,0.5205)};
\logLogSlopeTriangle{.425}{0.11}{0.9}{1.10}{blue}

\addplot[color=red,mark=square*,semithick, mark options={fill=markercolor}]
coordinates{(0.5,1.14227)(0.25,1.03142)(0.125,0.0764885)(0.0625,0.0100553)};
\logLogSlopeTriangleFlip{.38}{0.1}{0.72}{2.93}{red}

\addplot[color=black,mark=triangle*,semithick, mark options={fill=markercolor}]
coordinates{(0.5,1.80664)(0.25,0.0693555)(0.125,0.00751621)(0.0625,0.00118495)};
\logLogSlopeTriangleFlip{.38}{0.1}{0.585}{2.67}{black}

\addplot[color=magenta,mark=diamond*,semithick, mark options={fill=markercolor}]
coordinates{(0.5,0.258573)(0.25,0.0188506)(0.125,0.0011296)(0.0625,5.81723e-05)};
\logLogSlopeTriangleFlip{.38}{0.1}{0.425}{4.28}{magenta}

\addplot[color=violet,mark=pentagon*,semithick, mark options={fill=markercolor}]
coordinates{(0.5,0.115954)(0.25,0.00287139)(0.125,0.000102016)(0.0625,3.87019e-06)};
\logLogSlopeTriangleFlip{0.37}{0.1}{0.25}{4.72}{violet}

\legend{$N=1$,$N=2$,$N=3$,$N=4$, $N=5$}
\end{loglogaxis}
\end{tikzpicture}
}
\subfloat[Penalty flux]{
\begin{tikzpicture}
\begin{loglogaxis}[
	legend cell align=left,
	width=.475\textwidth,
    xlabel={Mesh size $h$},
    xmin=.025, xmax=1,
    ymin=1e-7, ymax=2.5,
    legend pos=south east,
    xmajorgrids=true,
    ymajorgrids=true,
    grid style=dashed,
] 

\addplot[color=blue,mark=*,semithick, mark options={fill=markercolor}]
coordinates{(0.5,1.00019)(0.25,1.00515)(0.125,0.633842)(0.0625,0.161762)};
\logLogSlopeTriangleFlip{.39}{0.11}{0.875}{1.97}{blue}

\addplot[color=red,mark=square*,semithick, mark options={fill=markercolor}]
coordinates{(0.5,1.01296)(0.25,0.503696)(0.125,0.0316637)(0.0625,0.00188428)};
\logLogSlopeTriangleFlip{.38}{0.1}{0.63}{4.07}{red}

\addplot[color=black,mark=triangle*,semithick, mark options={fill=markercolor}]
coordinates{(0.5,0.865781)(0.25,0.0445608)(0.125,0.00158453)(0.0625,0.000113704)};
\logLogSlopeTriangleFlip{.38}{0.1}{0.465}{3.80}{black}

\addplot[color=magenta,mark=diamond*,semithick, mark options={fill=markercolor}]
coordinates{(0.5,0.343325)(0.25,0.00491357)(0.125,0.000156595)(0.0625,5.13222e-06)};
\logLogSlopeTriangleFlip{.38}{0.1}{0.28}{4.93}{magenta}

\addplot[color=violet,mark=pentagon*,semithick, mark options={fill=markercolor}]
coordinates{(0.5,0.0682545)(0.25,0.000718077)(0.125,1.37715e-05)(0.0625,2.30844e-07)};
\logLogSlopeTriangleFlip{0.37}{0.1}{0.0985}{5.90}{violet}

\legend{$N=1$,$N=2$,$N=3$,$N=4$,$N=5$}
\end{loglogaxis}
\end{tikzpicture}
}
\caption{Convergence of $L^2$ errors for the Lamb wave solution.}
\label{fig:lamb}
\end{figure}
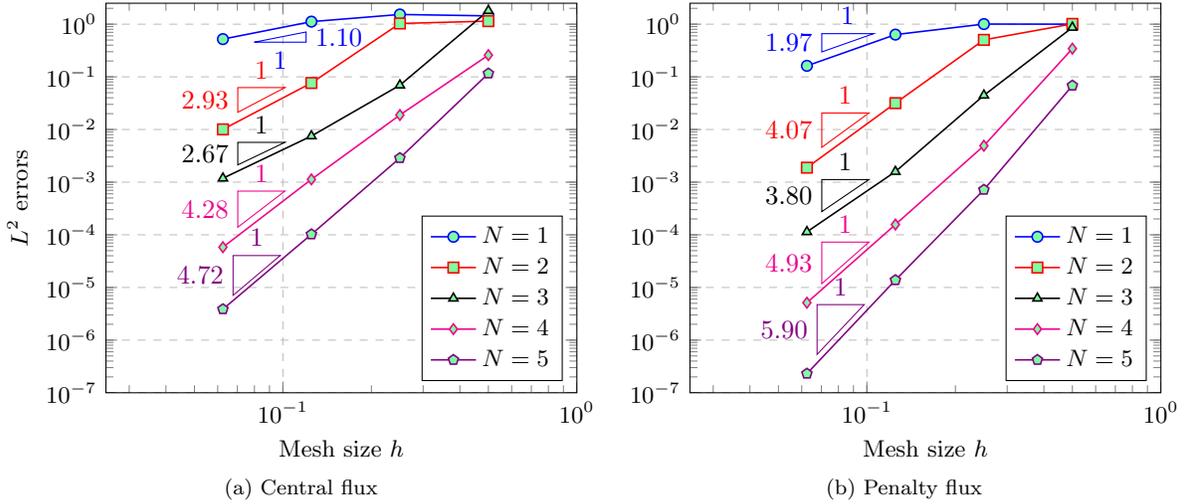

\subsubsection{Rayleigh waves in near-incompressible materials}
\label{sec:incomp}

As noted in Section~\ref{sec:wadgelas}, error estimates for isotropic elasticity no longer hold in the incompressible limit $\mu / \lambda \rightarrow 0$ due to the fact that $\bm{C}$ becomes singular.  We use the propagation of Rayleigh waves to examine the behavior of WADG for near-incompressible materials.  We follow \cite{sjogreen2012fourth, appelo2015energy} and fix $\lambda = 1$ and set $\mu = 1, .1, .01, .001, .0001$.  Since the Rayleigh wave propagates with speed proportional to $\sqrt{\mu}$, we compute $L^2$ errors at the final time $1 / (4\sqrt{\mu})$ to ensure a fair comparison between solutions at different values of $\mu$.  

Table~\ref{tab:incomp} shows relative errors for $\bm{v}$, $\bm{\sigma}$ at different orders and mesh sizes using the penalty flux.  The relative errors for $\bm{\sigma}$ grow as $\mu/\lambda \rightarrow 0$.  This is not surprising, as the constant in the error estimates of Section~\ref{sec:mwadgprop} depends on $\nor{\bm{C}}$, which blows up in the incompressible limit.  However, because the magnitude of $\bm{\sigma}$ also decreases as $\mu/\lambda\rightarrow 0$, relative errors for $\bm{v}$ remain roughly the same magnitude for near-incompressible materials.  Decreasing $\mu$ by four orders of magnitude results in up to a ten-fold increase in relative error for $\bm{\sigma}$, but less than a two-fold increase in error for $\bm{v}$.  

\begin{table}
\centering
\subfloat[Relative error in $\bm{v}$]{
\begin{tabular}{|c||c|c|c|c|c|}
\hline
 & $\mu = 1$ & $\mu = .1$ & $\mu = .01$ & $\mu = .001$ & $\mu = .0001$  \\
\hhline{|=|=|=|=|=|=|}
$N = 2$, $h = 1/4$ &  3.1063e-02  &  2.6972e-02  & 3.0858e-02 &  4.0103e-02  & 5.6835e-02\\
\hline
$N = 3$, $h = 1/4$ & 3.1677e-03 &  2.6848e-03 &  2.9854e-03 &  3.6960e-03 &  4.9362e-03\\
\hline
$N = 4$, $h = 1/4$ & 2.8726e-04 &  2.4990e-04 &  2.9142e-04 &  3.9774e-04 &  5.2469e-04\\      

\hhline{|=|=|=|=|=|=|}
$N = 2$, $h = 1/8$ & 3.1819e-03  & 2.5476e-03 &  2.8877e-03 &  3.5608e-03  & 4.7010e-03 \\
\hline
$N = 3$, $h = 1/8$ & 1.8867e-04  & 1.6925e-04  & 1.8509e-04  & 2.2520e-04  & 2.8301e-04\\
\hline
$N = 4$, $h = 1/8$ & 8.4999e-06  & 7.5750e-06 &  8.1094e-06  & 1.0760e-05  & 1.4886e-05\\
\hline         
\end{tabular}
}\\
\subfloat[Relative error in $\bm{\sigma}$]{
\begin{tabular}{|c||c|c|c|c|c|}
\hline
 & $\mu = 1$ & $\mu = .1$ & $\mu = .01$ & $\mu = .001$& $\mu = .0001$ \\
\hhline{|=|=|=|=|=|=|}
$N = 2$, $h = 1/4$ & 6.8250e-02  & 7.4130e-02 &  1.2104e-01 &  2.1333e-01 &  4.1451e-01\\
\hline
$N = 3$, $h = 1/4$ & 9.2980e-03  & 1.0685e-02 &  1.8109e-02 &  3.2046e-02 &  5.7612e-02\\
\hline
$N = 4$, $h = 1/4$ & 9.7251e-04  & 1.1687e-03  & 2.1941e-03 &  4.1378e-03 &  6.8404e-03\\
\hhline{|=|=|=|=|=|=|}
$N = 2$, $h = 1/8$ & 9.8138e-03  & 1.1429e-02  & 2.0331e-02 &  3.7889e-02 & 7.6516e-02\\
\hline
$N = 3$, $h = 1/8$ & 6.6596e-04  & 7.8157e-04  & 1.5353e-03  & 3.1239e-03 &  5.6341e-03\\
\hline
$N = 4$, $h = 1/8$ & 3.4272e-05  & 4.2639e-05 &  8.8393e-05  & 1.9764e-04 &  3.6052e-04 \\
\hline
\end{tabular}
}
\caption{Behavior of WADG for linear elastic wave propagation in the incompressible limit as $\mu / \lambda \rightarrow 0$.  Errors are shown for various orders and mesh resolutions using the penalty flux.  }
\label{tab:incomp}
\end{table}

\subsubsection{Stoneley waves} 

A Stoneley wave is supported along the interface between two solids \cite{stoneley1924elastic}.  Like Rayleigh waves, Stoneley waves decay exponentially away from the interface, and test the effectiveness of numerical fluxes across interfaces.  We follow \cite{wilcox2010high,appelo2015energy} and use discontinuous media defined by
\[
(\rho,\lambda,\mu) = \begin{cases}
(10 ,3,3) & y > 0\\
(1,1,1) & y < 0.
\end{cases}.
\]
The displacement vector for a Stoneley wave is then given by
\begin{align*}
u_1(x,y,t) &= \begin{cases}
{\rm Re}\LRp{\LRp{ikB_1e^{-kb_{1p}y} + kb_{1s}B_2 e^{-kb_{1s}y}}e^{i(ky-\omega t)}}, & y > 0\\
{\rm Re}\LRp{\LRp{-k b_{1p}B_1e^{-k b_{1p} y} + i kB_2e^{-kb_{1s}y}}e^{i(kx-\omega t)}}, & y < 0
\end{cases}\\
u_2(x,y,t) &= \begin{cases}
{\rm Re}\LRp{\LRp{ikB_3 e^{k b_{2p} y} - kb_{2s}B_4 e^{k b_{2s}y }} e^{i(kx-\omega t)}}, & y > 0\\
{\rm Re}\LRp{\LRp{k b_{2p} B_3 e^{k b_{2p}y} + ikB_4 e^{k b_{2s}y}}e^{i(kx-\omega t)}}, & y < 0
\end{cases},
\end{align*}
where $c_{st}$ is the Stoneley wave speed, and 
\[
k = \omega/ c_{st}, \qquad b_{jp} = \sqrt{1 - \frac{c_{st}^2}{(2\mu_j + \lambda_j)/\rho_j}}, \qquad  b_{js} = \sqrt{1 - \frac{c_{st}^2}{(\mu_j)/\rho_j}}, \qquad j = 1,2.
\]
The Stoneley wave speed $c_{st}$ can be determined based material parameters and interface conditions, and the amplitudes $B_1,B_2,B_3,B_4$ are determined from $c_{st}$ up to scaling by a constant.  For the parameters used in this study, we take $c_{st} = 0.546981324213884$, $B_1 = i0.2952173626624, B_2 = -0.6798795208473, B_3 = i0.5220044931212$, and $B_4 = -0.9339639688697$ \cite{wilcox2010high}.  We assume $k =1$, which gives $\omega=c_{st}$.  

We solve on the domain $[-1,1]\times [-5,5]$, and enforce Dirichlet boundary conditions at all boundaries using the exact solution.  Figure~\ref{fig:stoneley} shows $L^2$ errors for two uniform meshes of triangles constructed by bisecting a quadrilateral mesh of $K_{\rm 1D} \times 5K_{\rm 1D}$ elements. Figure~\ref{subfig:stoneley1} shows errors at time $T=5$ when $K_{\rm 1D}$ is even and the mesh is fitted to the interface at $y = 0$, while Figure~\ref{subfig:stoneley2} shows errors when $K_{\rm 1D}$ is odd and the interface cuts through element interiors.  

When the mesh is fitted to the interface, computed convergence rates \reviewerOne{using penalty fluxes match the theoretical $O(h^{N+1/2})$ rate.  When using central fluxes, we observe $O(h^{N})$ rates of convergence.  This matches the behavior observed when using central fluxes for the Rayleigh wave problem, instead of the odd-even pattern of convergence observed when using central fluxes for the harmonic oscillation and Lamb wave solutions.}

When the mesh is not fitted to the interface exactly, we compute the application of the weight-adjusted mass matrix inverse using a quadrature rule from Xiao and Gimbutas \cite{xiao2010quadrature} which is exact for degree $2N+1$ polynomials.  Since the values of $\rho,\mu$, and $\lambda$ are positive at all quadrature points, the method is energy stable.  However, since the exact solution is discontinuous, the error in elements cut by the interface is $O(1)$, resulting in $L^2$ errors which converge at rate $O(h^{1/2})$ \reviewerOne{for both penalty and central fluxes.  
We have also computed errors on a sequence of unfitted unstructured uniform meshes, as well as on a sequence of uniform meshes with randomly perturbed vertex positions.  In both cases, similar $O(h^{1/2})$ rates of convergence were observed.}  We note that, when using piecewise constant approximations of $\mu$ and $\lambda$, we observe the same $O(h^{1/2})$ convergence rate, though errors are roughly twice as large in magnitude.  
\begin{figure}[!h]
\centering
\subfloat[Central fluxes, fitted interface]{
\begin{tikzpicture}
\begin{loglogaxis}[
    legend cell align=left,
    legend style={legend pos=south east, font=\tiny},
    width=.475\textwidth,
    xlabel={Mesh size $h$},
    ylabel={$L^2$ errors}, 
    xmin=.025, xmax=1,
    ymin=1e-11, ymax=1.5,
    xmajorgrids=true,
    ymajorgrids=true,
    grid style=dashed,
] 

\addplot[color=blue,mark=*,semithick, mark options={fill=markercolor}]
coordinates{(0.5,0.503752)(0.25,0.299303)(0.125,0.158574)(0.0625,0.0815037)};
\logLogSlopeTriangleFlip{.4}{0.125}{0.9}{0.96}{blue}

\addplot[color=red,mark=square*,semithick, mark options={fill=markercolor}]
coordinates{(0.5,0.0548857)(0.25,0.0122212)(0.125,0.0029031)(0.0625,0.000712684)};
\logLogSlopeTriangleFlip{.4}{0.125}{0.725}{2.03}{red}

\addplot[color=black,mark=triangle*,semithick, mark options={fill=markercolor}]
coordinates{(0.5,0.00552463)(0.25,0.000886284)(0.125,0.000117076)(0.0625,1.5002e-05)};
\logLogSlopeTriangleFlip{.4}{0.125}{0.575}{2.96}{black}

\addplot[color=magenta,mark=diamond*,semithick, mark options={fill=markercolor}]
coordinates{(0.5,0.000342868)(0.25,2.08939e-05)(0.125,1.27168e-06)(0.0625,7.89726e-08)};
\logLogSlopeTriangleFlip{.4}{0.125}{0.375}{4.01}{magenta}

\addplot[color=violet,mark=pentagon*,semithick, mark options={fill=markercolor}]
coordinates{(0.5,2.59555e-05)(0.25,1.01846e-06)(0.125,3.39333e-08)(0.0625,1.0916e-09)};
\logLogSlopeTriangleFlip{0.4}{0.125}{0.215}{ 4.96}{violet}

\legend{$N=1$,$N=2$,$N=3$,$N=4$,$N=5$}
\end{loglogaxis}
\end{tikzpicture}
\label{subfig:stoneleycentral}
}
\subfloat[Penalty fluxes, fitted interface]{
\begin{tikzpicture}
\begin{loglogaxis}[
    legend cell align=left,
    legend style={legend pos=south east, font=\tiny},
    width=.475\textwidth,
    xlabel={Mesh size $h$},
    ylabel={$L^2$ errors}, 
    xmin=.025, xmax=1,
    ymin=1e-11, ymax=1.5,
    xmajorgrids=true,
    ymajorgrids=true,
    grid style=dashed,
] 

\addplot[color=blue,mark=*,semithick, mark options={fill=markercolor}]
coordinates{(0.5,0.291816)(0.25,0.110431)(0.125,0.0409099)(0.0625,0.0140493)};
\logLogSlopeTriangleFlip{.4}{0.125}{0.85}{1.54}{blue}

\addplot[color=red,mark=square*,semithick, mark options={fill=markercolor}]
coordinates{(0.5,0.0360738)(0.25,0.00627358)(0.125,0.00107612)(0.0625,0.0001848)};
\logLogSlopeTriangleFlip{.4}{0.125}{0.675}{2.54}{red}

\addplot[color=black,mark=triangle*,semithick, mark options={fill=markercolor}]
coordinates{(0.5,0.00234802)(0.25,0.000185264)(0.125,1.51976e-05)(0.0625,1.28801e-06)};
\logLogSlopeTriangleFlip{.4}{0.125}{0.485}{3.56}{black}

\addplot[color=magenta,mark=diamond*,semithick, mark options={fill=markercolor}]
coordinates{(0.5,0.000142893)(0.25,5.99393e-06)(0.125,2.55178e-07)(0.0625,1.10265e-08)};
\logLogSlopeTriangleFlip{.4}{0.125}{0.31}{4.53}{magenta}

\addplot[color=violet,mark=pentagon*,semithick, mark options={fill=markercolor}]
coordinates{(0.5,6.55362e-06)(0.25,1.31822e-07)(0.125,2.7319e-09)(0.0625,5.82544e-11)};
\logLogSlopeTriangleFlip{0.4}{0.125}{0.11}{ 5.55}{violet}
\legend{$N=1$,$N=2$,$N=3$,$N=4$,$N=5$}
\end{loglogaxis}
\end{tikzpicture}
\label{subfig:stoneley1}
}\\
\subfloat[Central flux, non-fitted interface]{
\begin{tikzpicture}
\begin{loglogaxis}[
    legend cell align=left,
    legend style={legend pos=south east, font=\tiny},
    width=.475\textwidth,
    xlabel={Mesh size $h$},
    xmin=.025, xmax=1,
    ymin=1e-2, ymax=1,
    xmajorgrids=true,
    ymajorgrids=true,
    grid style=dashed,
] 

\addplot[color=blue,mark=*,semithick, mark options={fill=markercolor}]
coordinates{(0.5,0.593576)(0.25,0.369792)(0.125,0.213907)(0.0625,0.130846)};
\addplot[color=red,mark=square*,semithick, mark options={fill=markercolor}]
coordinates{(0.5,0.245724)(0.25,0.170804)(0.125,0.118617)(0.0625,0.0830996)};
\addplot[color=black,mark=triangle*,semithick, mark options={fill=markercolor}]
coordinates{(0.5,0.142495)(0.25,0.102073)(0.125,0.0727388)(0.0625,0.0514866)};
\addplot[color=magenta,mark=diamond*,semithick, mark options={fill=markercolor}]
coordinates{(0.5,0.160473)(0.25,0.113053)(0.125,0.0790769)(0.0625,0.0551781)};
\logLogSlopeTriangle{0.425}{0.125}{0.325}{1/2}{black}

\legend{$N=1$,$N=2$,$N=3$,$N=4$}
\end{loglogaxis}
\end{tikzpicture}
\label{subfig:stoneleycentralnonfitted}
}
\subfloat[Penalty flux, non-fitted interface]{
\begin{tikzpicture}
\begin{loglogaxis}[
    legend cell align=left,
    legend style={legend pos=south east, font=\tiny},
    width=.475\textwidth,
    xlabel={Mesh size $h$},
    xmin=.025, xmax=1,
    ymin=1e-2, ymax=1,
    xmajorgrids=true,
    ymajorgrids=true,
    grid style=dashed,
] 

\addplot[color=blue,mark=*,semithick, mark options={fill=markercolor}]
coordinates{(0.5,0.1507)(0.25,0.0796)(0.125,0.0495)(0.0625,0.0331)}; 
\addplot[color=red,mark=square*,semithick, mark options={fill=markercolor}]
coordinates{(0.5,0.0820)(0.25,0.0594)(0.125,0.0427)(0.0625, 0.0307)};
\addplot[color=black,mark=triangle*,semithick, mark options={fill=markercolor}]
coordinates{(0.5, 0.0640)(0.25,0.0435)(0.125,0.0306)(0.0625,0.0216)};
\addplot[color=magenta,mark=diamond*,semithick, mark options={fill=markercolor}]
coordinates{(0.5, 0.0593)(0.25, 0.0436)(0.125, 0.0318)(0.0625,0.0229)};
\logLogSlopeTriangle{0.425}{0.125}{0.15}{1/2}{black}

\legend{$N=1$,$N=2$,$N=3$,$N=4$}
\end{loglogaxis}
\end{tikzpicture}
\label{subfig:stoneley2}
}
\caption{Convergence of WADG for a Stoneley wave using a fitted mesh aligned with the interface (Figures~\ref{subfig:stoneleycentral} and \ref{subfig:stoneley1}) and a non-fitted mesh where the interface does not lie exactly on an element boundary (Figure~\ref{subfig:stoneleycentralnonfitted} and \ref{subfig:stoneley2}). }
\label{fig:stoneley}
\end{figure}
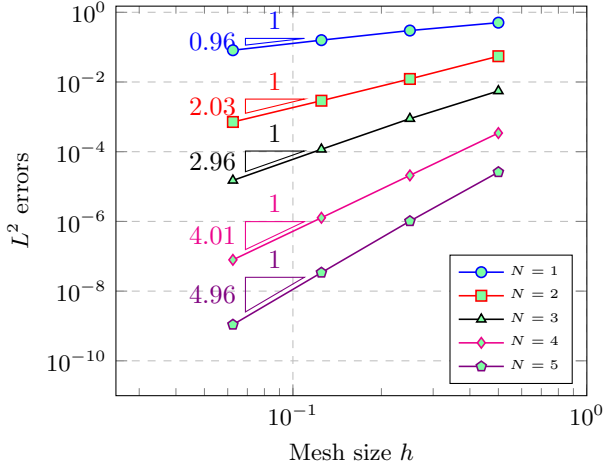
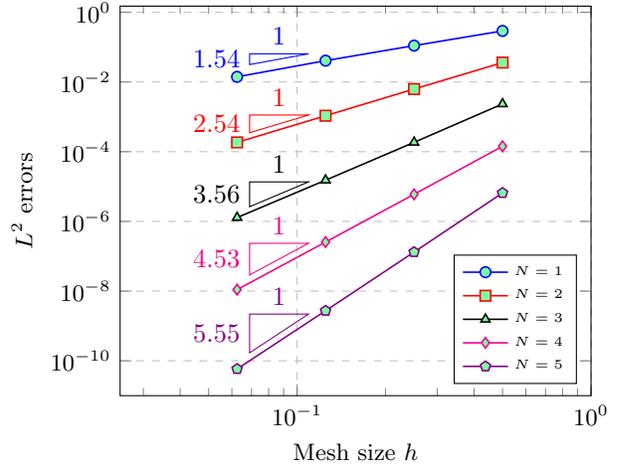
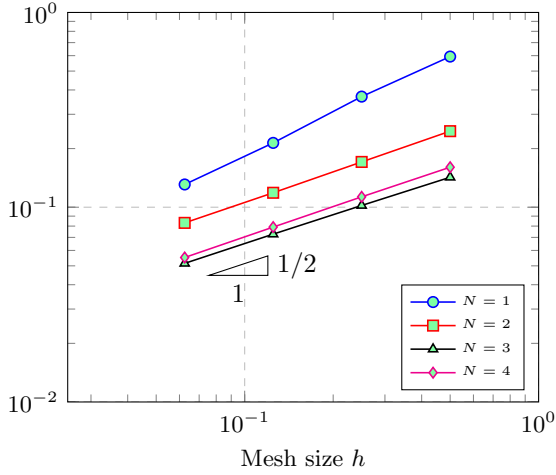
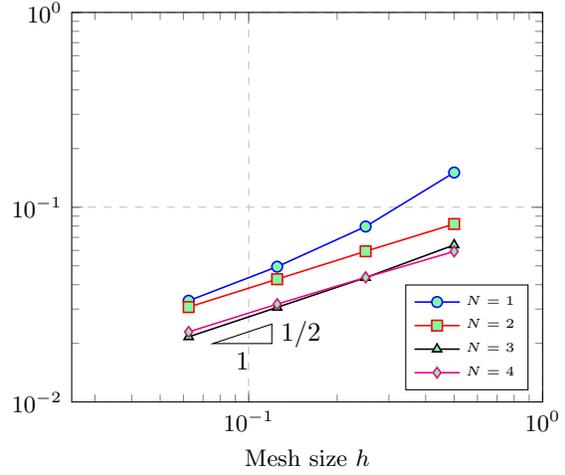

\subsubsection{Convergence to \reviewerOne{manufactured} and reference solutions}

\reviewerOne{To check the accuracy of the method for problems with smoothly varying heterogeneous media, we follow \cite{mercerat2015nodal} and consider a manufactured solution.  We assume isotropic media, and incorporate variations into the stiffness matrix $\bm{C}$ by taking $\lambda$ such that
\[
\lambda(x,y) = \lambda_0 + \tilde{\lambda}(x,y),
\]
where $\lambda_0$ is a constant.  We assume the displacement solution $\bm{u}$ is given as a plane wave 
\[
u_1(x,y,t) = \cos\LRp{k (x- c_P t)}, \qquad u_2(x,y,t) = \cos\LRp{k (x- c_S t)},
\]
where $c_P = \sqrt{(2\mu+\lambda_0)/\rho}$ and $c_S = \sqrt{\mu/\rho}$ are the P- and S-wave velocities corresponding to Lame parameters $\mu, \lambda_0$.  

This plane wave is the solution of the homogeneous elastic wave equations with $\tilde{\lambda}(x,y) = 0$.  However, this is not true if $\tilde{\lambda}(x,y) \neq 0$ varies spatially.  In order to test the convergence of our method when $\bm{C}$ contains smoothly varying coefficients, we modify our equations by adding source terms $\bm{f}_{\bm{\sigma}}$ such that the plane wave solution satisfies
\begin{align}
\rho \pd{\bm{v}}{t} &= \sum_{i=1}^d \bm{A}_i^T \pd{\bm{\sigma}}{\bm{x}_i}\nonumber\\
\bm{C}^{-1} \pd{\bm{\sigma}}{t} &= \sum_{i=1}^d \bm{A}_i \pd{\bm{v}}{\bm{x}_i} + \bm{f}_{\bm{\sigma}}.
\label{eq:symelas}
\end{align}
Using the fact that the plane wave is the solution to the homogeneous equations, it is straightforward to show that the source terms are 
\[
\bm{f}_{\bm{\sigma}} = -\tilde{\lambda} \Grad\cdot\bm{v}\left(\begin{array}{c}1\\1\\0\end{array}\right), 
\]
where $\bm{v} = \pd{\bm{u}}{t}$ is the velocity of the exact plane wave solution.  These source terms are computed using the same quadrature rule used for WADG.  

Figure~\ref{fig:manufactured} shows the convergence of $L^2$ errors for a plane wave manufactured solution with $k = \pi$.  We set $\rho = 1, \mu = 1, \lambda_0 = 2$, and $\tilde{\lambda}(x,y) = \frac{1}{2}\sin(2\pi x)\sin(2\pi y)$, and compute errors at final time $T = 5$ for $N=1,\ldots,5$.  We observe that the $L^2$ errors convergence at a rate between the theoretical $O(h^{N+1/2})$ and optimal $O(h^{N+1})$ rates for the penalty flux with $\tau_{\bm{v}} = \tau_{\bm{\sigma}} = 1$.  For central fluxes, we observe an even-odd pattern of convergence, with rates near $O(h^N)$ for $N$ odd and $O(h^{N+1})$ for $N$ even.   
}

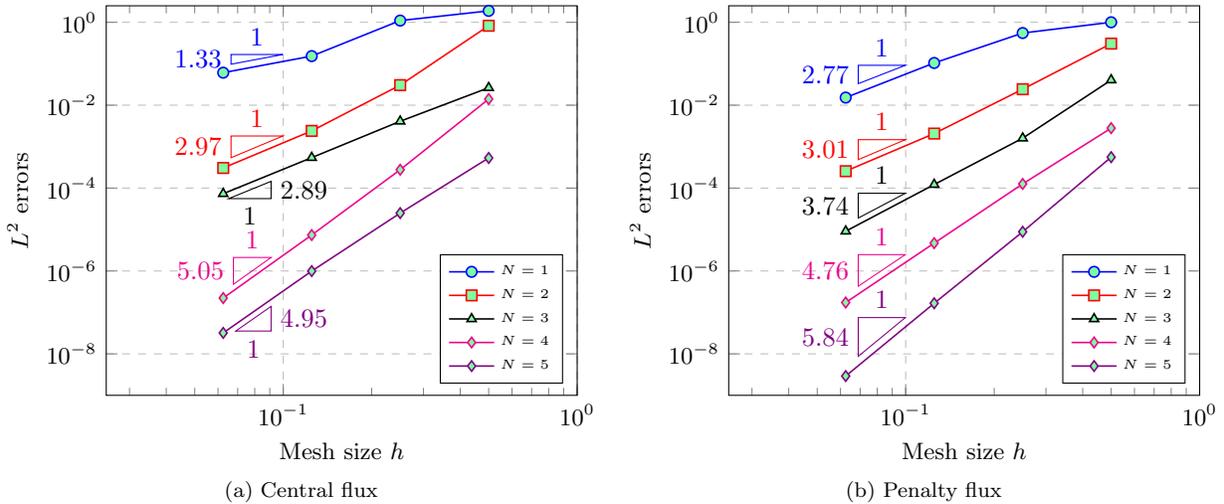
\begin{figure}
\centering
\subfloat[\reviewerOne{Central flux}]{
\begin{tikzpicture}
\begin{loglogaxis}[
    legend cell align=left,
    legend style={legend pos=south east, font=\tiny},
    width=.475\textwidth,
    xlabel={Mesh size $h$},
    ylabel={$L^2$ errors}, 
    xmin=.025, xmax=1,
    ymin=1e-9, ymax=2.5,
    xmajorgrids=true,
    ymajorgrids=true,
    grid style=dashed,
] 


\addplot[color=blue,mark=*,semithick, mark options={fill=markercolor}]
coordinates{(0.5,1.88416)(0.25,1.09941)(0.125,0.152765)(0.0625,0.0608485)};
\logLogSlopeTriangleFlip{0.375}{0.11}{0.85}{1.33}{blue}

\addplot[color=red,mark=square*,semithick, mark options={fill=markercolor}]
coordinates{(0.5,0.825136)(0.25,0.0304815)(0.125,0.00239865)(0.0625,0.000305336)};
\logLogSlopeTriangleFlip{0.375}{0.11}{0.61}{2.97}{red}

\addplot[color=black,mark=triangle*,semithick, mark options={fill=markercolor}]
coordinates{(0.5,0.0263573)(0.25,0.00408897)(0.125,0.000542327)(0.0625,7.31536e-05)};
\logLogSlopeTriangle{0.35}{0.09}{0.505}{2.89}{black}

\addplot[color=magenta,mark=diamond*,semithick, mark options={fill=markercolor}]
coordinates{(0.5,0.0140713)(0.25,0.000277044)(0.125,7.36781e-06)(0.0625,2.22198e-07)};
\logLogSlopeTriangleFlip{0.35}{0.08}{0.285}{5.05}{magenta}

\addplot[color=violet,mark=diamond*,semithick, mark options={fill=markercolor}]
coordinates{(0.5,0.000532981)(0.25,2.49055e-05)(0.125,9.93481e-07)(0.0625,3.20304e-08)};
\logLogSlopeTriangle{0.35}{0.075}{0.165}{4.95}{violet}

\legend{$N=1$,$N=2$,$N=3$,$N=4$,$N=5$}
\end{loglogaxis}
\end{tikzpicture}
}
\subfloat[\reviewerOne{Penalty flux}]{
\begin{tikzpicture}
\begin{loglogaxis}[
    legend cell align=left,
    legend style={legend pos=south east, font=\tiny},
    width=.475\textwidth,
    xlabel={Mesh size $h$},
    ylabel={$L^2$ errors}, 
    xmin=.025, xmax=1,
    ymin=1e-9, ymax=2.5,
    xmajorgrids=true,
    ymajorgrids=true,
    grid style=dashed,
] 

\addplot[color=blue,mark=*,semithick, mark options={fill=markercolor}]
coordinates{(0.5,1.00045)(0.25,0.550577)(0.125,0.104337)(0.0625,0.0152592)};
\logLogSlopeTriangleFlip{0.375}{0.1}{0.8}{2.77}{blue}

\addplot[color=red,mark=square*,semithick, mark options={fill=markercolor}]
coordinates{(0.5,0.305771)(0.25,0.0242762)(0.125,0.00206544)(0.0625,0.000255504)};
\logLogSlopeTriangleFlip{0.375}{0.1}{0.605}{3.01}{red}

\addplot[color=black,mark=triangle*,semithick, mark options={fill=markercolor}]
coordinates{(0.5,0.0404346)(0.25,0.00159769)(0.125,0.000121514)(0.0625,9.09561e-06)};
\logLogSlopeTriangleFlip{0.375}{0.1}{0.455}{3.74}{black}

\addplot[color=magenta,mark=diamond*,semithick, mark options={fill=markercolor}]
coordinates{(0.5,0.00279312)(0.25,0.000126062)(0.125,4.69835e-06)(0.0625,1.73129e-07)};
\logLogSlopeTriangleFlip{0.375}{0.1}{0.28}{4.76}{magenta}

\addplot[color=violet,mark=diamond*,semithick, mark options={fill=markercolor}]
coordinates{(0.5,0.00055782)(0.25,8.82004e-06)(0.125,1.66264e-07)(0.0625,2.91186e-09)};
\logLogSlopeTriangleFlip{0.375}{0.1}{0.10}{5.84}{violet}

\legend{$N=1$,$N=2$,$N=3$,$N=4$,$N=5$}
\end{loglogaxis}
\end{tikzpicture}
}
\caption{\reviewerOne{Convergence of WADG under mesh refinement to a manufactured solution with smoothly varying heterogeneous media for $N=1,\ldots,5$.}}
\label{fig:manufactured}
\end{figure}

We also examine the accuracy of the WADG method for smoothly varying heterogeneous media by comparing against a \reviewerOne{reference spectral element method solution of degree $N = 50$} on a unit square $[-1,1]^2$.  \reviewerOne{Zero traction} boundary conditions are enforced weakly through numerical fluxes \cite{wilcox2010high}.  We use a heterogeneous isotropic medium with $\rho,\lambda, \mu$ set to
\[
\rho(\bm{x}) = 1, \qquad \lambda(\bm{x}) = 1 + .25\sin(\pi x)\sin(\pi y), \qquad \mu(\bm{x}) = 1 + .25\cos(\pi x) \cos(\pi y).  
\]
Initial stresses are set to zero, while the initial velocity is set to 
\[
v_1(\bm{x},0) = \cos(\pi x)\sin(\pi y), \qquad v_2(\bm{x},0) = -\sin(\pi x)\cos(\pi y).
\]
Figure~\ref{fig:refsol} shows $L^2$ errors with respect to the reference solution at time $T=1/2$ for different mesh sizes and orders of approximation.  Computed convergence rates fall between the optimal $O(h^{N+1})$ and predicted $O(h^{N+1/2})$ when using the penalty flux with penalty parameters set to $1$.  
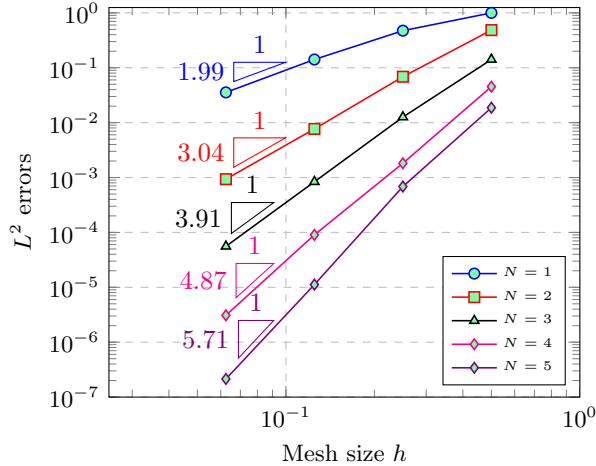
\begin{figure}
\centering
\begin{tikzpicture}
\begin{loglogaxis}[
    legend cell align=left,
    legend style={legend pos=south east, font=\tiny},
    width=.475\textwidth,
    xlabel={Mesh size $h$},
    ylabel={$L^2$ errors}, 
    xmin=.025, xmax=1,
    ymin=1e-7, ymax=1.25,
    xmajorgrids=true,
    ymajorgrids=true,
    grid style=dashed,
] 

\addplot[color=blue,mark=*,semithick, mark options={fill=markercolor}]
coordinates{(0.5,1)(0.25,0.473741)(0.125,0.141342)(0.0625,0.0355409)};    
\logLogSlopeTriangleFlip{0.375}{0.11}{0.81}{1.99}{blue}

\addplot[color=red,mark=square*,semithick, mark options={fill=markercolor}]
coordinates{(0.5,0.484875)(0.25,0.0686547)(0.125,0.00767686)(0.0625,0.000933392)};   
\logLogSlopeTriangleFlip{0.375}{0.11}{0.59}{3.04}{red}

\addplot[color=black,mark=triangle*,semithick, mark options={fill=markercolor}]
coordinates{(0.5,0.141415)(0.25,0.0126263)(0.125,0.0008361)(0.0625,5.57276e-05)};    
\logLogSlopeTriangleFlip{0.35}{0.09}{0.42}{3.91}{black}

\addplot[color=magenta,mark=diamond*,semithick, mark options={fill=markercolor}]
coordinates{(0.5,0.0452398)(0.25,0.00182005)(0.125,9.0346e-05)(0.0625,3.08518e-06)};   
\logLogSlopeTriangleFlip{0.35}{0.08}{0.255}{4.87}{magenta}

\addplot[color=violet,mark=diamond*,semithick, mark options={fill=markercolor}]
coordinates{(0.5,0.0186937)(0.25,0.000684602)(0.125,1.12013e-05)(0.0625,2.13796e-07)};    
\logLogSlopeTriangleFlip{0.35}{0.075}{0.1}{5.71}{violet}

\legend{$N=1$,$N=2$,$N=3$,$N=4$,$N=5$}
\end{loglogaxis}
\end{tikzpicture}
\caption{Convergence of WADG under mesh refinement to a reference $N=50$ spectral method solution with smoothly varying heterogeneous media for $N=1,\ldots,5$ using a penalty flux.}
\label{fig:refsol}
\end{figure}

\reviewerOne{
\subsubsection{Curvilinear meshes}

We now present numerical experiments verifying the stability and accuracy of the formulation presented in Section~\ref{sec:curvi} for curvilinear meshes.  We use isoparametric mappings in the following experiments, where the mapping from the reference element to each physical element is a polynomial of degree $N$.  We construct these mappings by following \cite{hesthaven2007nodal}.  Starting from a uniform triangular mesh on the Lamb wave problem domain $\Omega = [0,2]\times [0,1]$, we place high order Warp and Blend interpolation nodes on each element \cite{warburton2006explicit}.  The physical locations $(x_i,y_i)$ (for $i = 1,\ldots,N_p K$) of these nodes are then perturbed to produce new nodal positions $(\tilde{x}_i,\tilde{y}_i)$ where
\[
\tilde{x}_i = x_i + \frac{1}{10} \cos\LRp{\frac{\pi}{2}x} \cos\LRp{3\pi y}, \qquad 
\tilde{y}_i = y_i + \frac{1}{20} \sin(\pi x) \cos(3\pi y).
\]
These new nodal positions $(\tilde{x}_i,\tilde{y}_i)$ now define a coordinate mapping from the reference element to a curved physical element, producing the warped mesh in Figure~\ref{fig:curvieig}.  This mesh warping is constructed such that the $x$ and $y$ deformations of each element are of roughly the same magnitude, while leaving the positions of nodes on the boundary unchanged.  

Figure~\ref{fig:curvieig} shows eigenvalues of the DG discretization matrix for $N=3$ for both a uniform (affine) mesh and a warped curvilinear mesh.  We use the quadrature-based skew-symmetric formulation introduced in Section~\ref{sec:curvi}, and consider both central and penalty fluxes (with penalty parameters set uniformly to $1$).  We observe that for both the central and penalty fluxes, all eigenvalues contain non-positive real parts (up to machine precision), indicating that the semi-discrete system is energy stable.  The introduction of the curvilinear warping appears to result in a magnification of the real and imaginary parts of larger magnitude eigenvalues.  
\begin{figure}
\centering
\subfloat[$\tau_{\bm{v}}=\tau_{\bm{\sigma}} = 0$ (central flux)]{\includegraphics[width=.49\textwidth]{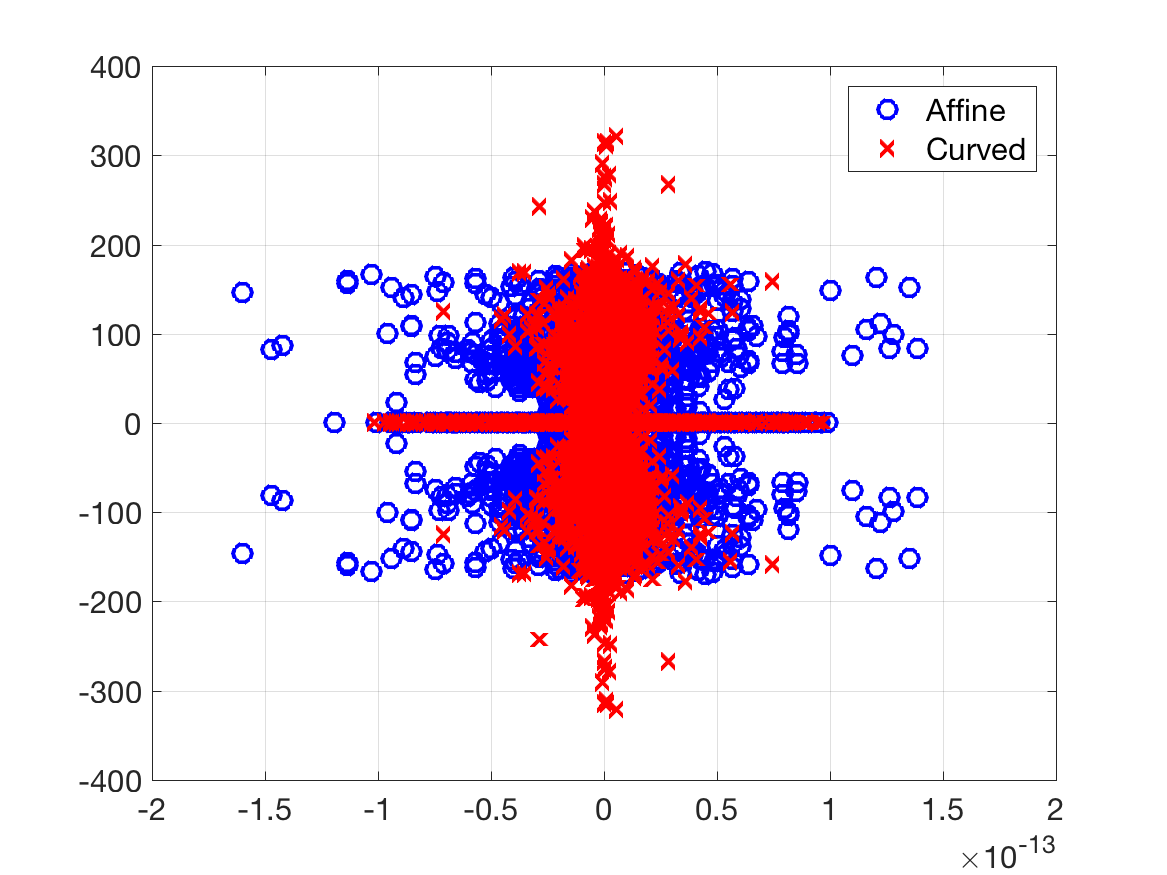}}
\subfloat[$\tau_{\bm{v}}=\tau_{\bm{\sigma}} = 1$ (penalty flux)]{\includegraphics[width=.49\textwidth]{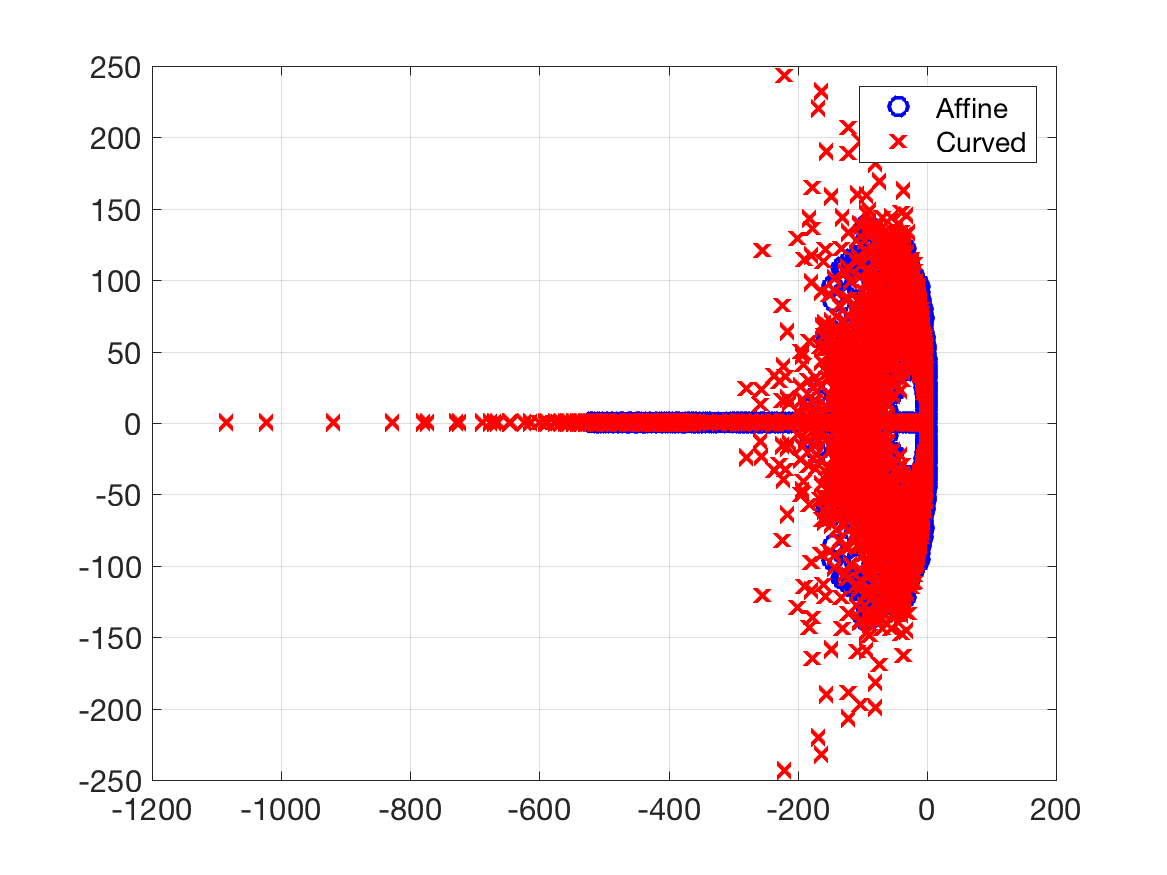}}\\
\subfloat[Warped mesh]{\includegraphics[width=.375\textwidth]{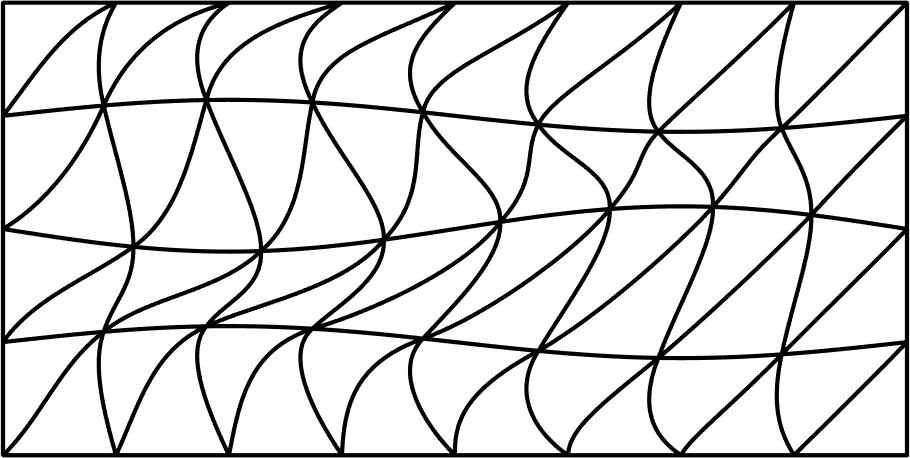}}
\caption{Spectra of the DG discretization matrix for  central and penalty fluxes on a warped curvilinear mesh of degree $N=3$.}
\label{fig:curvieig}
\end{figure}

We also compute $L^2$ errors on a sequence of refined curvilinear meshes for $N=2,\ldots,5$, skipping $N=1$ as it reduces to the affine case.  These curvilinear meshes are constructed using the warping procedure described previously.  Errors for both central and penalty fluxes are shown in Figure~\ref{fig:curvierr}.  We observe rates of convergence of $L^2$ errors which are consistent with the rates observed for affine meshes in Section~\ref{sec:rlamb}.  
}

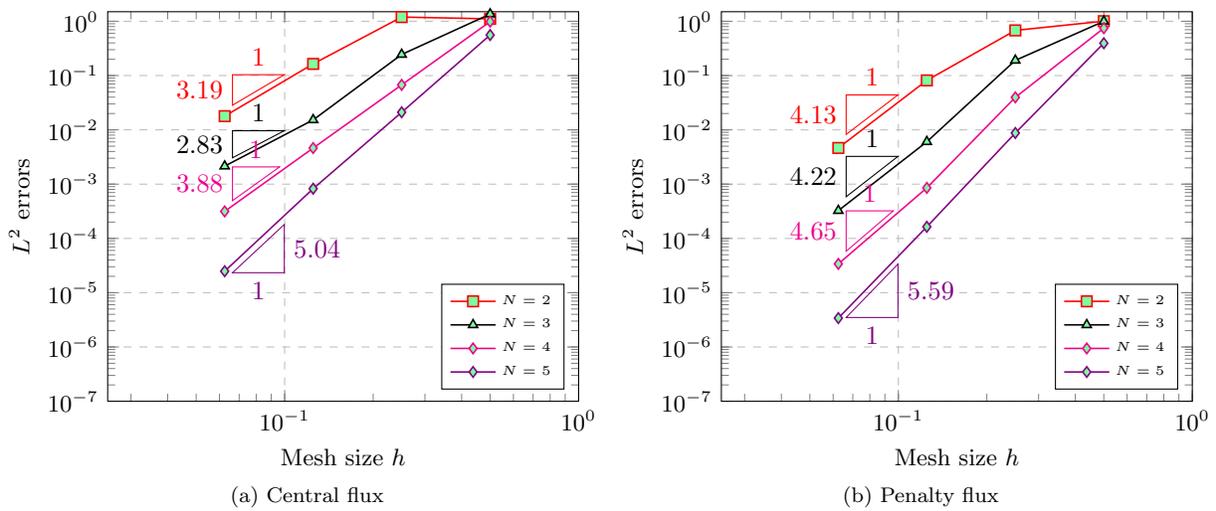
\begin{figure}
\centering
\subfloat[Central flux]{\begin{tikzpicture}
\begin{loglogaxis}[
    legend cell align=left,
    legend style={legend pos=south east, font=\tiny},
    width=.475\textwidth,
    xlabel={Mesh size $h$},
    ylabel={$L^2$ errors}, 
    xmin=.025, xmax=1,
    ymin=1e-7, ymax=1.5,
    xmajorgrids=true,
    ymajorgrids=true,
    grid style=dashed,
] 


\addplot[color=red,mark=square*,semithick, mark options={fill=markercolor}]
coordinates{(0.5,1.10945)(0.25,1.19601)(0.125,0.164019)(0.0625,0.017878)};
\logLogSlopeTriangleFlip{0.375}{0.11}{0.76}{3.19}{red}

\addplot[color=black,mark=triangle*,semithick, mark options={fill=markercolor}]
coordinates{(0.5,1.3575)(0.25,0.244502)(0.125,0.0152567)(0.0625,0.0021426)};
\logLogSlopeTriangleFlip{0.375}{0.11}{0.625}{2.83}{black}

\addplot[color=magenta,mark=diamond*,semithick, mark options={fill=markercolor}]
coordinates{(0.5,0.985545)(0.25,0.0673766)(0.125,0.00462058)(0.0625,0.000314475)};
\logLogSlopeTriangleFlip{0.365}{0.1}{0.515}{3.88}{magenta}

\addplot[color=violet,mark=diamond*,semithick, mark options={fill=markercolor}]
coordinates{(0.5,0.558713)(0.25,0.0210648)(0.125,0.000823737)(0.0625,2.48818e-05)};
\logLogSlopeTriangle{0.375}{0.11}{0.33}{5.04}{violet}

\legend{$N=2$,$N=3$,$N=4$,$N=5$}
\end{loglogaxis}
\end{tikzpicture}
}
\subfloat[Penalty flux]{\begin{tikzpicture}
\begin{loglogaxis}[
    legend cell align=left,
    legend style={legend pos=south east, font=\tiny},
    width=.475\textwidth,
    xlabel={Mesh size $h$},
    ylabel={$L^2$ errors}, 
    xmin=.025, xmax=1,
    ymin=1e-7, ymax=1.5,
    xmajorgrids=true,
    ymajorgrids=true,
    grid style=dashed,
] 

\addplot[color=red,mark=square*,semithick, mark options={fill=markercolor}]
coordinates{(0.5,1.00509)(0.25,0.679352)(0.125,0.0814806)(0.0625,0.00463964)};
\logLogSlopeTriangleFlip{0.375}{0.11}{0.685}{4.13}{red}

\addplot[color=black,mark=triangle*,semithick, mark options={fill=markercolor}]
coordinates{(0.5,0.998103)(0.25,0.190567)(0.125,0.00603429)(0.0625,0.000324808)};
\logLogSlopeTriangleFlip{0.375}{0.11}{0.525}{4.22}{black}

\addplot[color=magenta,mark=diamond*,semithick, mark options={fill=markercolor}]
coordinates{(0.5,0.746305)(0.25,0.0396791)(0.125,0.000852651)(0.0625,3.39164e-05)};
\logLogSlopeTriangleFlip{0.365}{0.1}{0.385}{4.65}{magenta}

\addplot[color=violet,mark=diamond*,semithick, mark options={fill=markercolor}]
coordinates{(0.5,0.394381)(0.25,0.00875065)(0.125,0.000163618)(0.0625,3.40714e-06)};
\logLogSlopeTriangle{0.375}{0.11}{0.215}{5.59}{violet}

\legend{$N=2$,$N=3$,$N=4$,$N=5$}
\end{loglogaxis}
\end{tikzpicture}
}
\caption{Convergence of WADG for the Lamb wave problem on curvilinear meshes.}
\label{fig:curvierr}
\end{figure}

%

\subsection{Application examples}

We next demonstrate the accuracy and flexibility of WADG for several application-based problems in linear elasticity with heterogeneity and anisotropy.  \reviewerOne{All computations are done using penalty parameters $\tau_{\bm{v}} = \tau_{\bm{\sigma}} = 1$ unless specified otherwise.}

\subsubsection{Stiff inclusion}

The stiff inclusion problem  is a common test of methods for linear elastic wave propagation \cite{leveque2002finite, kaser2006arbitrary, appelo2015energy}, where an inclusion with higher wavespeed is embedded within a non-stiff region.  Waves which reach this region of high wavespeed are transmitted through the inclusion, bouncing back and forth within the region.  This vibration then produces waves which propagate outward from the inclusion.  

We solve on a domain $[-1,1]\times [-.5, .5]$ with a rectangular inclusion located at $[-.5,.5]\times[-.1,.1]$.  Outside of the inclusion, material parameters are taken to be 
\[
\rho = 1, \qquad \mu = 1,\qquad \lambda = 2.
\]
Within the inclusion, material parameters are taken to be
\[
\rho = 1, \qquad \mu = 100,\qquad \lambda = 200,
\]
such that the wave speed in the rectangular inclusion is ten times that of the wave speed outside.  A pulse is generated through velocity boundary conditions at $x = -1$
\[
v_1(x,y,t) = \begin{cases}
\sin(\pi t / t_0), & t < t_0\\
0, & t \geq t_0
\end{cases}, \qquad
v_2(x,y,t) = 0.
\]
In our experiments, we take $t_0 = .025$.  Traction free boundary conditions are enforced at all other domain boundaries.  

\begin{figure}
\centering
\subfloat[$\LRb{\bm{\sigma}_{xx}+\bm{\sigma}_{yy}}$]{\includegraphics[width=.475\textwidth]{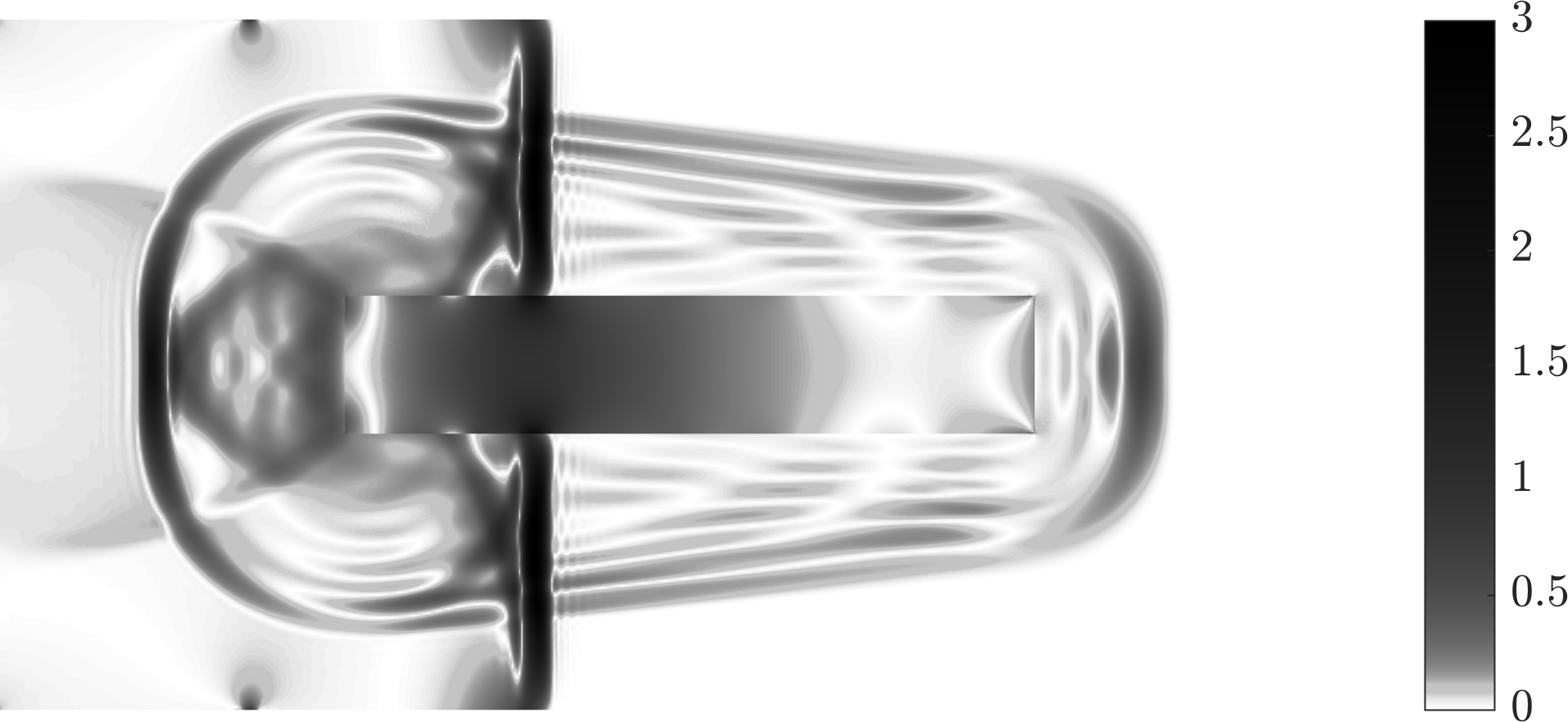}}
\hspace{2em}
\subfloat[$\LRb{\bm{\sigma}_{xy}}$]{\includegraphics[width=.475\textwidth]{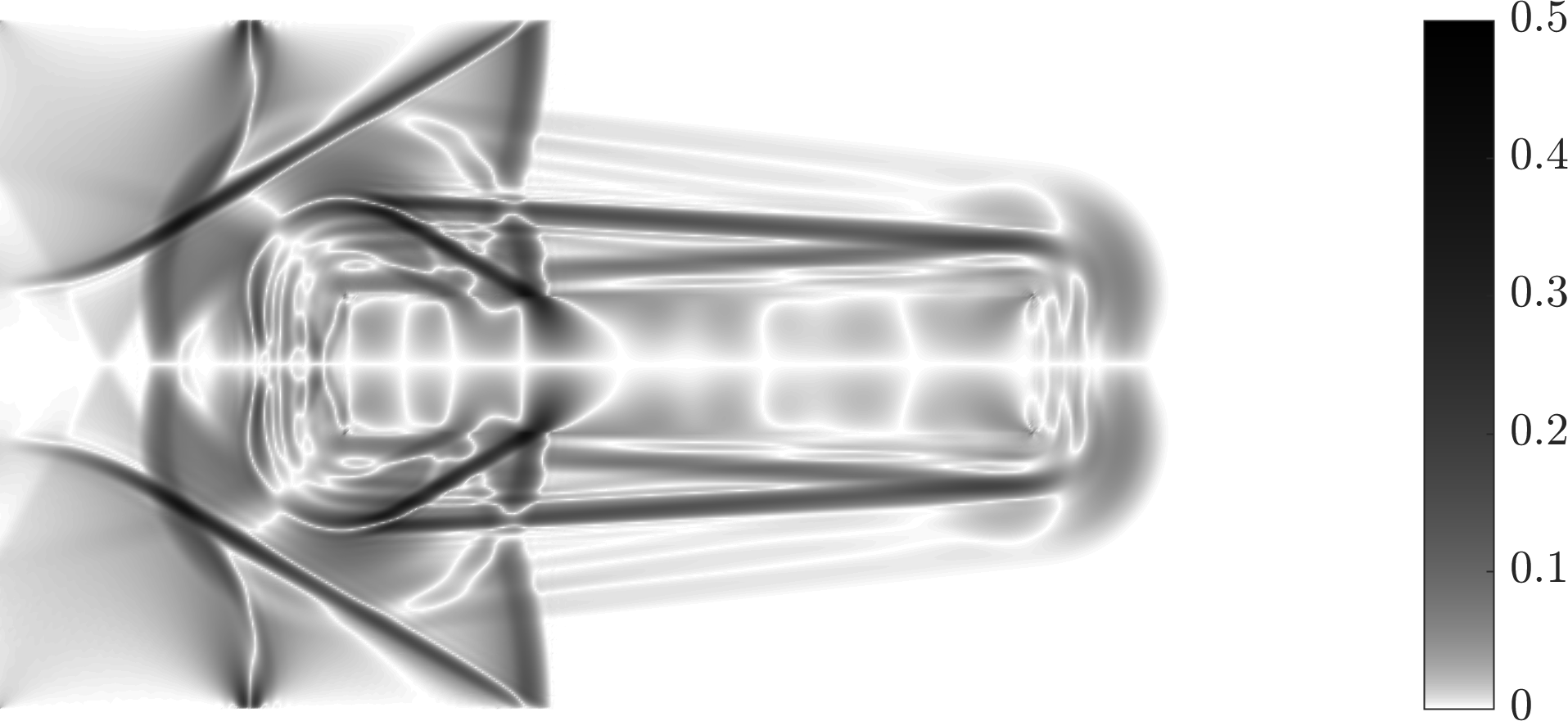}}
\caption{Values of different solution fields at $T=.4$ for the stiff inclusion problem of Leveque \cite{leveque2002finite} with order of approximation $N=5$.}
\label{fig:incl}
\end{figure}


We construct a uniform triangular mesh by dividing each element of a quadrilateral meshes along the diagonal to produce triangular meshes, using $100 \times 50$ quadrilateral elements in the $x$ and $y$ coordinates, respectively.  \reviewerOne{This is the same mesh resolution and polynomial degree used by K\"{a}ser and Dumbser in \cite{kaser2006arbitrary}, and provides roughly the same number of degrees of freedom in the $x$ and $y$ directions as the numerical setup used by \cite{appelo2015energy}. }
Figure~\ref{fig:incl} shows values of $\LRb{\bm{\sigma}_{xx} + \bm{\sigma}_{yy}}$ at final time $T=.4$.  Following the approach \reviewerOne{taken in the literature  \cite{leveque2002finite, kaser2006arbitrary, appelo2015energy}}, a nonlinear color scale is used in order to distinguish small-amplitude waves \reviewerOne{and produce a \textit{schlieren}-style image.}.  The results show qualitatively good agreement with results in the literature.

\subsubsection{Heterogeneous anisotropic material}

We next examine a model wave propagation problem in heterogeneous anisotropy media \cite{carcione1988wave, komatitsch2000simulation, de2007arbitrary}.   The density $\rho = 7100$ is constant over the domain, while the entries of the stiffness matrix $\bm{C}$ are taken to be
\begin{align*}
\bm{C}_{11} &= .165, \quad \bm{C}_{12} = .05, \quad \bm{C}_{22} = .062, \quad \bm{C}_{33} = .0396, \qquad x < 0\\
\bm{C}_{11} &= .165, \quad \bm{C}_{12} = .0858, \quad \bm{C}_{22} = .165, \quad \bm{C}_{33} = .0396, \qquad x > 0,
\end{align*}
with the remaining entries determined by symmetry or set to zero if unspecified.  For $ x < 0$, this corresponds to an anisotropic material, while for $x > 0$, this corresponds to an isotropic material with $\mu = .0396, \lambda = .0858$.

The computational domain is taken to be $[-.32, .32]^2$, and we use $N=5$ and a triangular mesh of $32768$ elements constructed by subdividing a grid of $128\times 128$ uniform quadrilaterals.  \reviewerOne{In order to provide a fair comparison to results in the literature, the degree $N$ and the mesh size are chosen based on the numerical setups used in \cite{komatitsch2000simulation, de2007arbitrary}.  Komatitsch, Barnes and Tromp use a $130\times 130$ grid of uniform quadrilateral elements of degree $N=5$ in \cite{komatitsch2000simulation}.  The authors of \cite{de2007arbitrary} use an unstructured mesh of 37944 elements of degree $N=5$, where the average triangle edge length matches the edge length of the triangles in our mesh.}  Forcing is applied to the $y$-component of the velocity by a Ricker wavelet point source 
\[
f(\bm{x},t) = \LRp{1 - 2(\pi f_0 (t-t_0))^2} e^{-(\pi f_0 (t-t_0))^2}\delta(x-x_0),
\]
where $x_0 = -.02$, $f_0 = .17$, and $t_0 = 1/f_0$.\footnote{All values and units are adapted from \cite{komatitsch2000simulation,de2007arbitrary}, and correspond to units of meters, kg, and microseconds.}  

\reviewerOne{We take the penalty parameters to be $\tau_{\bm{v}} = \tau_{\bm{\sigma}} = 1/2$.  While there is little visual difference between taking the penalty parameters to be $1$ instead of $1/2$, we observe that taking a smaller penalty parameter makes it possible to use a larger timestep ($C_{\rm CFL} = 10$) without blowing up.  Reducing the penalty parameter further does not appear to allow a significant increase in the maximum stable timestep.  This suggests that the naive choice of $\tau_{\bm{v}} = \tau_{\sigma} = 1$ is not optimal with respect to the maximum stable timestep and stiffness of the semi-discrete system, as discussed in Section~\ref{sec:scaling}. }

Figure~\ref{fig:aniso} shows the $y$-component of velocity $\bm{v}_2$ at times $T = 30 \mu s$ (zoomed in) and $T= 60 \mu s$.  Both results show qualitative agreement with reference results from \cite{carcione1988wave, komatitsch2000simulation, de2007arbitrary}.  
\begin{figure}
\centering
\subfloat[$T=30 \mu s$ (zoomed in)]{\includegraphics[height=18em]{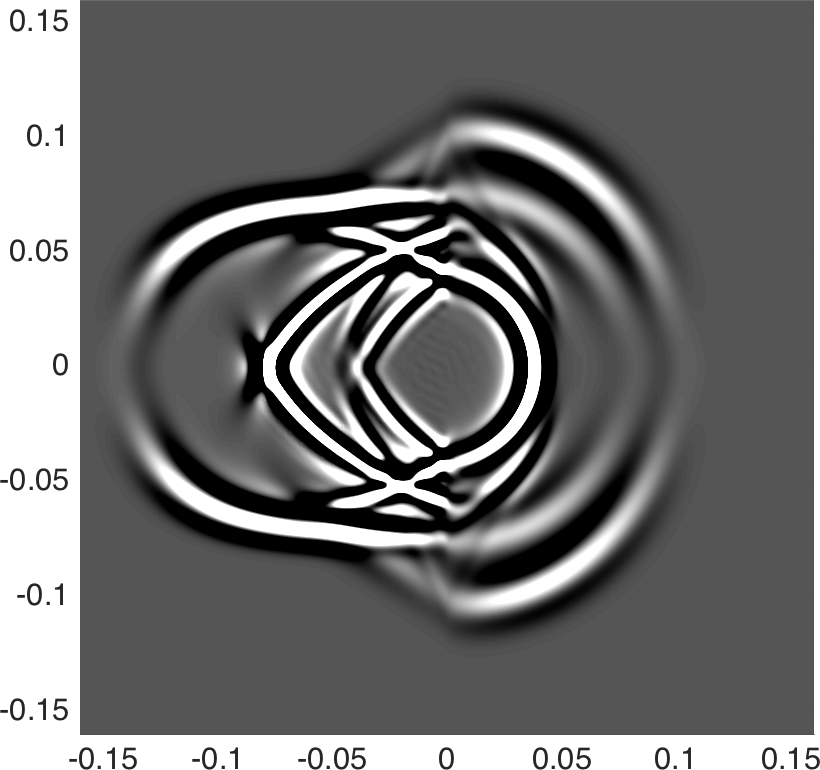}}
\hspace{3em}
\subfloat[$T=60 \mu s$]{\includegraphics[height=18.25em]{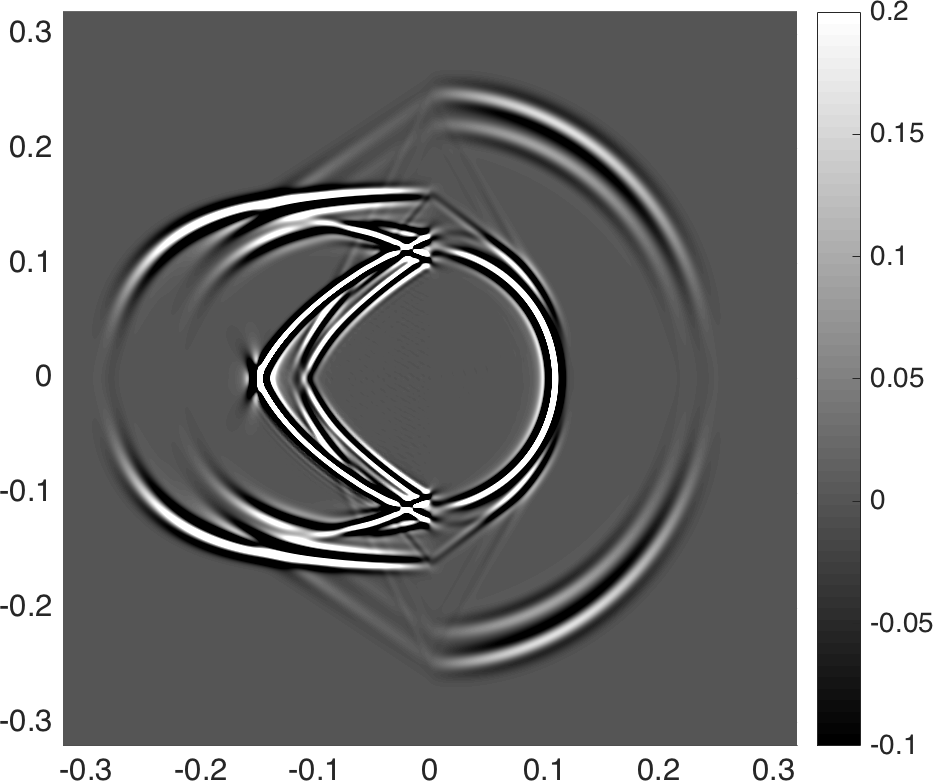}}
\caption{An example of wave propagation in heterogeneous anisotropic media.  The vertical component of the velocity $\bm{v}_2$ is shown at $T = 30$ and $T=60$ microseconds. }
\label{fig:aniso}
\end{figure}

\reviewerOne{
\subsection{A three-dimensional example and computational results}
}
\label{sec:comp}

We now present a three-dimensional example of elastic wave propagation in heterogeneous media with sub-element variations and a discontinuity across an interface.  We consider isotropic elastic wave propagation on the cube $[-.5,.5]^3$ with a discontinuity in material coefficients across $z = 0$
\[
\rho = 1, \qquad \mu(\bm{x}) = \begin{cases}
2 + w(\bm{x}), & z < 0\\
1 + w(\bm{x}), & z > 0
\end{cases},
\qquad
\lambda(\bm{x}) = \begin{cases}
2, & z < 0\\
1, & z > 0
\end{cases}
\]
where $w(\bm{x}) = .5 \cos(3\pi x)\cos(3\pi y)\cos(3\pi z)$.  Forcing is applied to the $x$-component of velocity through a smoothed point source and Ricker wavelet
\[
f(\bm{x},t) = \LRp{1 - 2(\pi f_0 (t-t_0))^2} e^{-(\pi f_0 (t-t_0))^2}e^{- \LRp{a \nor{\bm{x}-\bm{x}_0}}^2}
\]
where $\bm{x}_0 = (0,0,.1)^T$, $a = 100$, $f_0 = 10$, and $t_0 = 1/f_0$.  
\begin{figure}
\centering
\hspace{.5em}
\subfloat[Computational mesh]{\includegraphics[width=.275\textwidth]{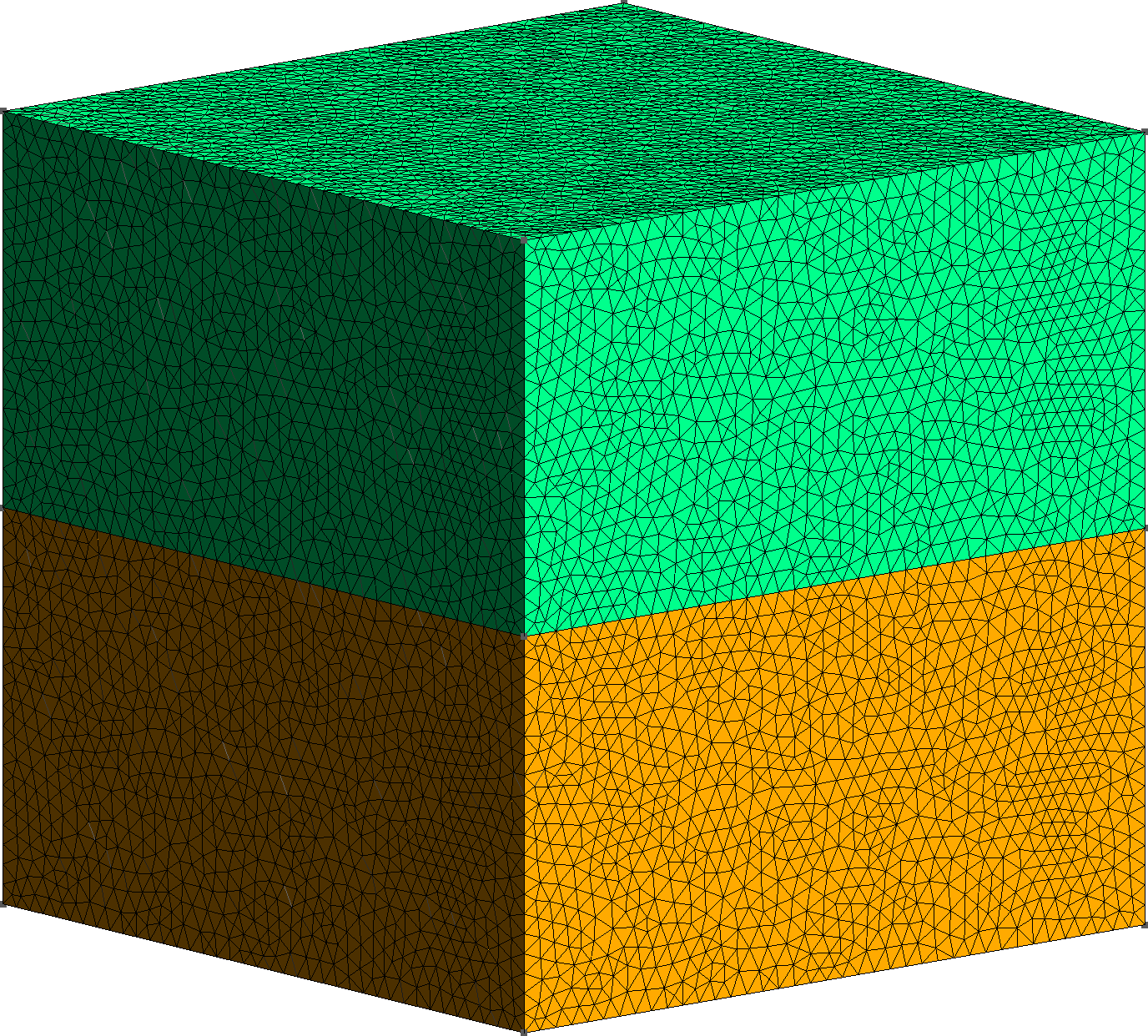}\label{subfig:mesh}}
\hspace{.1em}
\subfloat[Piecewise constant coefficients]{\includegraphics[width=.35\textwidth]{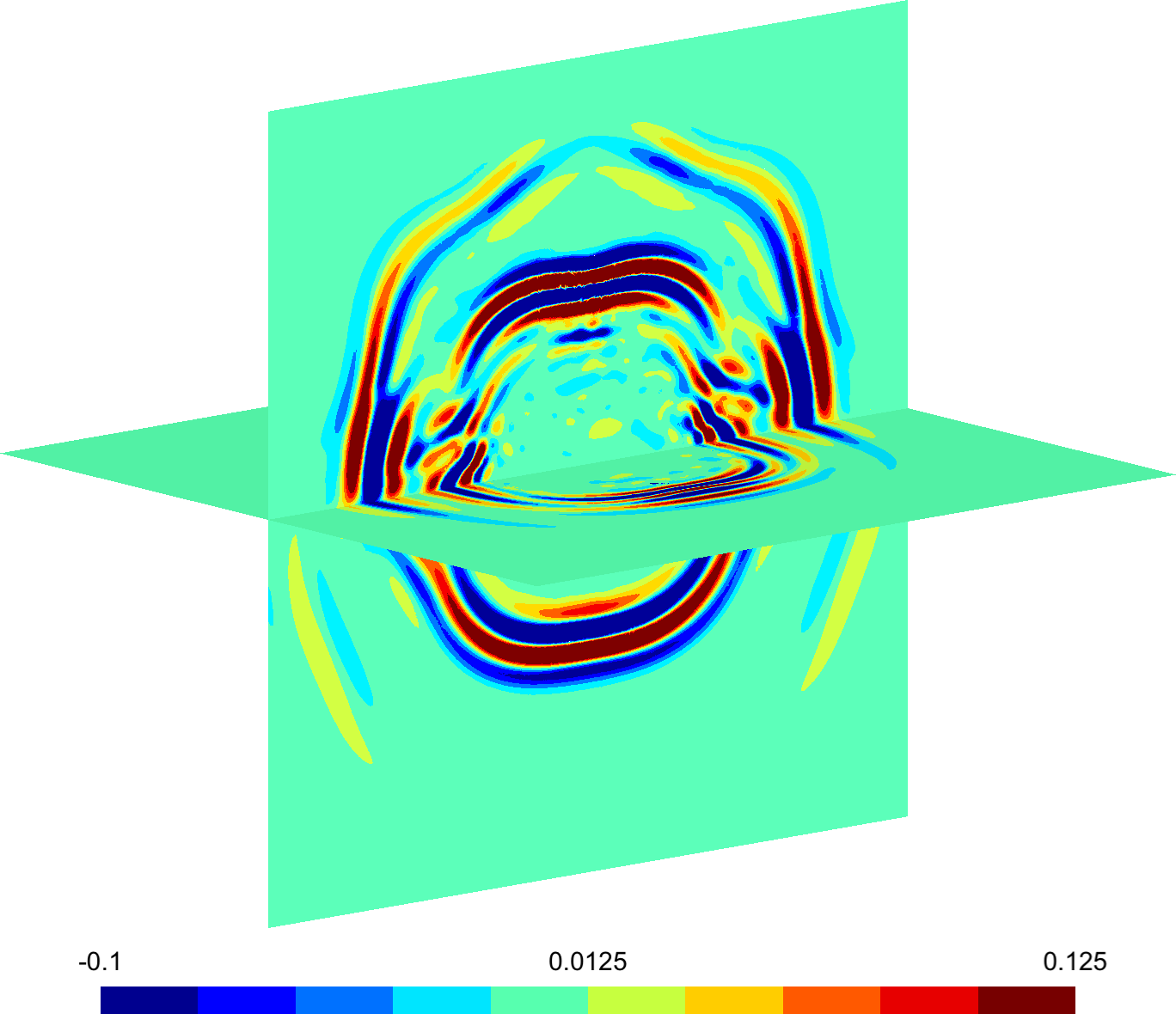}}
\subfloat[Smooth coefficients]{\includegraphics[width=.34\textwidth]{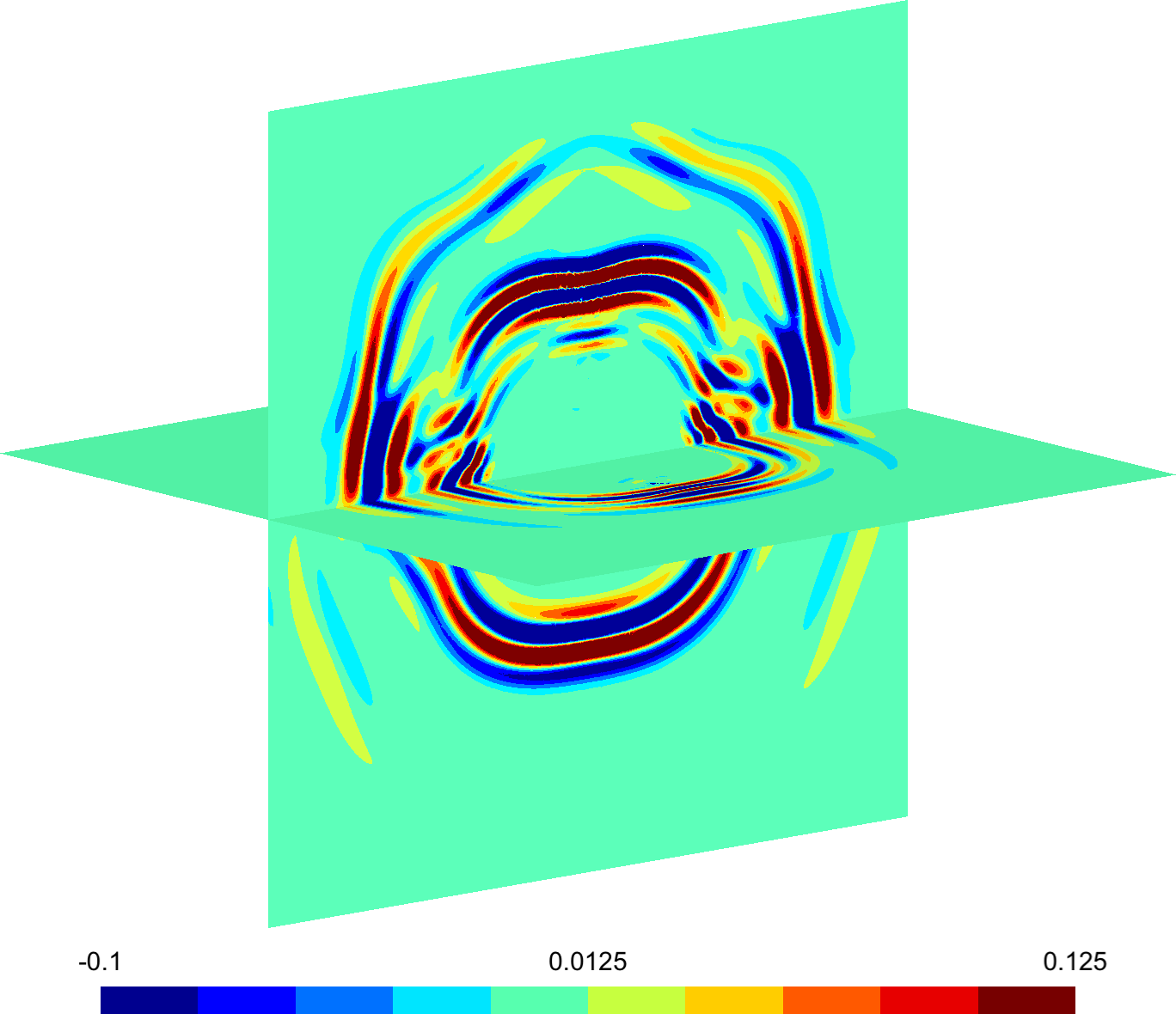}}
\caption{Mesh and $xz$, $xy$ slices of $\bm{v}_1$ at $T = .5$.  The order of approximation is taken to be $N=5$.  }
\label{fig:3d}
\end{figure}
Figure~\ref{fig:3d} shows the $x$-velocity of the computed solution at $T = .5$, and 
Figure~\ref{subfig:mesh} shows the unstructured mesh of 222824 tetrahedral elements of degree $N=5$ used to compute both solutions.  \reviewerOne{In order to capture the discontinuity in material parameters, the elements of this mesh are made to conform to the $z=0$ plane.  The mesh resolution and degree are chosen to resolve the spatial variation of the smoothed point source present in the forcing function.  For piecewise constant coefficients (using the average of each coefficient over an element), spurious reflections are observed in the solution.  When smoothly varying coefficients are resolved within an element using WADG, these spurious reflections disappear. }

\reviewerOne{
These computations are performed on an Nvidia GTX 980 GPU, following the implementation of GPU-accelerated DG methods outlined in \cite{klockner2009nodal}.  This approach breaks the computational work for each time-step into \emph{volume} and \emph{surface} kernels (for the evaluation of the DG formulation) and an \emph{update} kernel (for the application of a time integration method).  In this implementation, we apply the weight-adjusted mass matrix inverse within the update kernel as well.  Strategies for volume and surface kernels follow \cite{klockner2009nodal}, while computational approaches for WADG are outlined in \cite{chan2016weight2}.  

Non-invasive GPU-accelerated implementations of WADG are described in \cite{chan2016weight2}, where kernels for the acoustic wave equation in isotropic media are re-used.  We re-write high order DG methods based on explicit inversion of weighted mass matrices (as done in \cite{mercerat2015nodal, bencomo2015discontinuous}) into a similar non-invasive form with equivalent storage. The semi-discrete form of standard DG yields a system of ODEs over each element
\[
\bm{M}_w\td{\bm{U}}{t} = \bm{A}_h\bm{U},
\]
where $\bm{A}_h\bm{U}$ denotes the evaluation of the DG right hand side for some local vector $\bm{U}$. Multiplying by an un-weighted mass matrix on both sides gives 
\[
\bm{M}^{-1}\bm{M}_w\td{\bm{U}}{t} = \bm{M}^{-1}\bm{A}_h\bm{U},
\]
The right hand side is the same as the right hand side for the case when the weighting function is $w(x) = 1$, and can re-use DG kernels for isotropic wave propagation.   The influence of the spatially varying coefficient is incorporated by inverting the weighted projection matrix $\bm{P}_w = \LRp{\bm{M}^{-1}\bm{M}_w}^{-1} = \bm{M}_w^{-1}\bm{M}$ and applying it to the right hand side. Since the weight $w$ is spatially varying and distinct from element-to-element,  we pre-compute and store $\bm{P}_w$ explicitly over each element prior to time-stepping.  The weight-adjusted DG method is equivalent to replacing the matrices $\bm{M}_w^{-1}\bm{M}$ with the weight-adjusted projection matrix $\bm{M}^{-1}\bm{M}_{1/w}$, which can be applied in a matrix-free fashion as described in Section~\ref{sec:wadgimplement}.  

We now examine computational costs associated with the use of the weight-adjusted DG method.  Computational statistics are computed using the Nvidia profiler \texttt{nvprof}.  We consider first the costs associated with the use of weight-adjusted approximations to scalar weighted mass matrices.  While WADG clearly reduces storage costs associated with high order DG methods, it is less clear how WADG affects computational runtime on accelerator and many-core architectures.  

\begin{table}
    \centering
    \subfloat[Weighted projection matrix $\bm{P}_w$]{
    \begin{tabular}{|c||c|c|c|c|c|c|}
 \hline 
 Batch size $N_{\rm batch}$ & 1 & 2 & 3 & 4 & 5 & 6 \\
 \hhline{|=||=|=|=|=|=|=|}
N = 1 &2.264 &1.224 &0.8927 &0.7514 &0.6976 & $\bm{0.6571}$ \\
 \hline 
N = 2 &3.185 &2.83 &$\bm{2.791}$ &2.823 &2.837 &2.86 \\
 \hline 
N = 3 &$\bm{9.907}$ &10.14 &10.25 &10.16 &10.19 &10.21 \\
 \hline 
N = 4 &$\bm{29.41}$ &29.74 &30.01 &30.03 &30.23 &30.47 \\
 \hline 
N = 5 &74.48 &74.39 &74.16 &74.01 &$\bm{73.88}$ &74.32 \\
 \hline 
 N = 6 & 173.4 & 173.9 & 175.2 & 171 & $\bm{170.5}$ & 172.4 \\ 
 \hline 
N = 7 & $\bm{329.4}$ & 330.4 & 329.6 & 331 & & \\ 
 \hline  
    \end{tabular}
    } 
       
    \subfloat[Weight-adjusted projection $\bm{M}^{-1}\bm{M}_{1/w}$]{
    \begin{tabular}{|c||c|c|c|c|c|c|}
 \hline 
 Batch size $N_{\rm batch}$ & 1 & 2 & 3 & 4 & 5 & 6 \\
 \hhline{|=||=|=|=|=|=|=|} 
N = 1 & 2.382 &1.26 &0.8833 &0.6986 &$\bm{0.5832}$ &0.6534 \\
\hline
N = 2 &3.864 &2.101 &2.495 &1.968 &2.28 &$\bm{1.964}$ \\
\hline
N = 3 &7.092 &$\bm{6.788}$ &6.9 &6.89 &6.888 &6.836 \\
\hline
N = 4 &24.79 &$\bm{22.15}$ &26.63 &24.02 &24.16 & 24 \\
\hline
N = 5 &70.62 &76.6 &61.55 &$\bm{56.35}$ &56.38 &58.98 \\
\hline
N = 6 & 179.8 & 145.7 & $\bm{129.9}$ & 145.7 & 131.9 & 144.1 \\ 
\hline
N = 7 & 411.6 & 412.7 & $\bm{393}$ & 510.3 & & \\ 
\hline
    \end{tabular}
    }        
    \caption{Runtimes (nanoseconds) per element for weighted and weight-adjusted projections as a function of batch size $N_{\rm batch}$.  The lowest runtimes are highlighted in bold.  }
    \label{tab:PKRuntimes}
\end{table}

We compare the application of pre-computed and stored weighted projection matrices $\bm{P}_w$ with a matrix-free application of the weight-adjusted matrix $\bm{M}^{-1}\bm{M}_{1/w}$ using a GPU-accelerated implementation.  As described in \cite{klockner2009nodal, chan2016weight2}, we batch process $N_{\rm batch} \geq 1$ elements within a single kernel workgroup.  Table~\ref{tab:PKRuntimes} displays the average runtime per element for the weighted and weight-adjusted projection kernels as a function of batch size $N_{\rm batch}$ when using a mesh of $50,000$ elements and a quadrature which is exact for polynomials of degree $2N+1$.  

When processing only a single element per batch, WADG is less efficient than weighted projection at all orders.  However, the cost of WADG goes down rapidly with the number of elements per batch.  After optimizing over the batch size, weight-adjusted projection is faster than weighted projection up to $N=6$.  At $N=7$, WADG is slower than weighted projection, as the batch size is limited by the maximum number of active threads.\footnote{In our implementation, the number of active threads per workgroup is the number of quadrature points multiplied by the number of elements per batch.  The batch size is limited by the maximum number of threads in a workgroup ($1024$ for Nvidia GPUs).  For $N = 7$, the quadrature rule of degree $2N+1$ contains $214$ points; the largest batch size we can run is then $4$ elements, as processing $5$ elements per batch requires $1070$ threads.}  We note that, if the strength of quadrature is reduced from $2N+1$ to $2N$, the resulting WADG runtimes are faster than weighted projection at all tested orders, achieving between a $1.5-2.3\times$ speedup for $N=1,\ldots, 7$ while maintaining virtually identical numerical results \cite{chan2016weight1}.  


We can take a closer look at these results using the Nvidia profiler \texttt{nvprof}, looking in particular at the metrics \texttt{gld\_load\_throughput} and \texttt{dram\_read\_throughput}.  Both metrics track data throughput; however, the former includes data fetched from cache, while the latter does not. We fix $N=2$, increase the batch size, and record the output given by \texttt{nvprof} for the weighted projection and weight-adjusted kernels.  Table~\ref{tab:compareWADG} shows that the value of \texttt{dram\_read\_throughput} for the weighted projection kernel is higher than that of the weight-adjusted kernel, implying that more data is streamed through the kernel.  However, the value of \texttt{gld\_load\_throughput} for the weight-adjusted kernel is higher than that of the weighted projection kernel. This indicates that, while the loading of pre-computed and stored weighted projection matrices exploits the high bandwidth available to GPUs, it does not take advantage of cache locality due to the fact that the projection matrices must be loaded separately over each element.  In contrast, the matrix-free implementation of WADG allows the matrices $\bm{V}_q, \bm{P}_q$ to be re-used over multiple elements once loaded into cache.  

\begin{table}
    \centering
\subfloat[\texttt{dram\_read\_throughput}  (GB/s)]{ \begin{tabular}{|c||c|c|c|}
    \hline
  & $N_{\rm batch} = 1$ & $N_{\rm batch} = 3$ & $N_{\rm batch} = 5$\\
 \hhline{|=||=|=|=|} 
     Weighted projection $\bm{P}_w$ & 135.81 &156.23 & 153.13\\
     \hline 
    Weight adjusted projection  $\bm{M}^{-1}\bm{M}_{1/w}$ & 25.041  &38.811& 42.736\\
     \hline     
    \end{tabular}
    }
    
    \subfloat[\texttt{gld\_throughput} (GB/s)]{
    \begin{tabular}{|c||c|c|c|}    
     \hline 
       & $N_{\rm batch} = 1$ & $N_{\rm batch} = 3$ & $N_{\rm batch} = 5$\\
 \hhline{|=||=|=|=|} 
          Weighted projection $\bm{P}_w$  & 353.10 & 411.87 & 410.95\\
     \hline 
    Weight adjusted projection $\bm{M}^{-1}\bm{M}_{1/w}$  & 804.30  & 1e+03 & 1e+03 \\
     \hline
    \end{tabular}
    }
    \caption{Reported Global Load Throughput (\texttt{gld\_throughput}) and Device Memory Read Throughput (\texttt{dram\_read\_throughput}) in GB/s for $N=2$ and various $N_{\rm batch}$.}
    \label{tab:compareWADG}
\end{table}


For $N_{\rm batch} = 1$, the cache efficiency of the weight-adjusted kernel is offset by the increased computational cost of quadrature-based interpolation and projection.  However, increasing $N_{\rm batch}$ to three elements increases both the values of \texttt{gld\_load\_throughput} and \texttt{dram\_read\_throughput}, resulting in a roughly $2\times$ speedup in runtime for the weight-adjusted kernel.  These results show that the low-storage nature of the weight-adjusted kernel frees up bandwidth in exchange for increased computational work, while taking advantage of data locality.  

\pgfplotstableread[row sep=crcr]{
N V S U\\
1   97.1800   82.3310  109.6260\\
2  110.8550   98.2020  118.7820\\
3   97.5430  118.9640   85.9900\\
4   66.3910  124.3630   58.6750\\
5   47.1200  120.1480   52.1000\\
6   32.8960   90.8150   44.1840\\
7   23.0840   81.6090   25.0326\\
}\BW

\pgfplotstableread[row sep=crcr]{
N V S U\\
                         1           228.07993286888          219.086635170504          113.790994927898\\
                         2          585.427326232541          272.951627469766          255.347833123331\\
                         3          1083.76569187859          424.354263904029          565.324933321158\\
                         4          1300.21151151178          580.406327594429          652.265626483218\\
                         5          1459.76972928974          722.829057153327          775.095713662755\\
                         6          1519.01507138031          729.815340586669          1107.72548427317\\
                         7           1564.8631276467          820.602768015583          1197.77303361906\\
}\GFLOPS

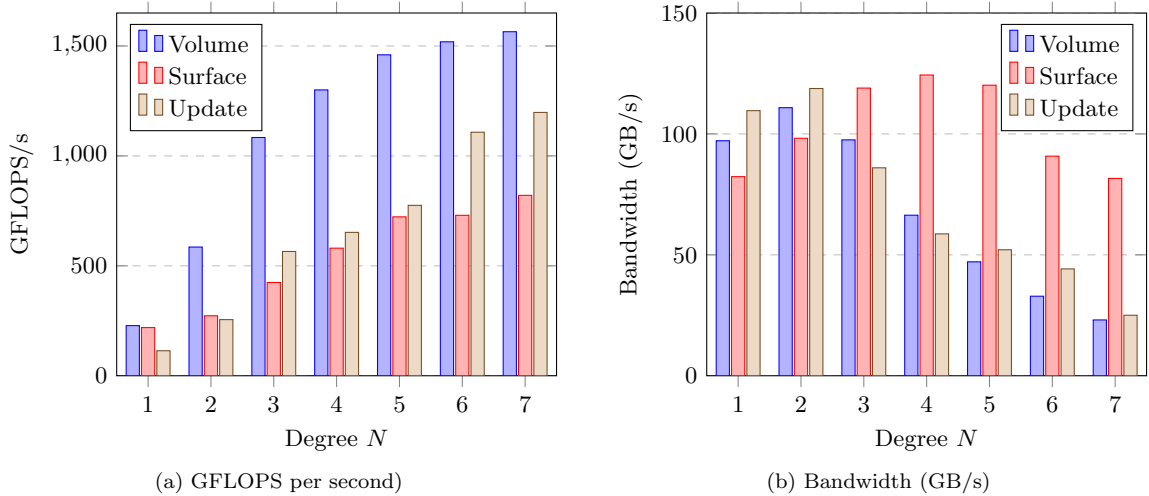
\begin{figure}
\centering
\subfloat[GFLOPS per second)]{
\begin{tikzpicture}
\begin{axis}[
	width=.45\textwidth,
	legend cell align=left,
	xlabel={Degree $N$},
	ylabel={GFLOPS/s},
xmin=.5, xmax=7.5,
	ymin=0,ymax=1650,
        ybar=2*\pgflinewidth,
        bar width=5pt,
	xtick={1,2,3,4,5,6,7},
		legend pos=north west,
	ymajorgrids=true,
	grid style=dashed,
] 
\addplot table[x=N, y=V] from \GFLOPS;
\addplot table[x=N, y=S] from \GFLOPS;
\addplot table[x=N, y=U] from \GFLOPS;
\legend{Volume, Surface, Update}
\end{axis}
\end{tikzpicture}
}
\hspace{1em}
\subfloat[Bandwidth (GB/s)]{
\begin{tikzpicture}
\begin{axis}[
	width=.45\textwidth,
	legend cell align=left,
	xlabel={Degree $N$},
	ylabel={Bandwidth (GB/s)},
xmin=.5, xmax=7.5,
	ymin=0,ymax=150,
             ybar=2*\pgflinewidth,
             bar width=5pt,
	xtick={1,2,3,4,5,6,7},
	legend pos=north east,
	ymajorgrids=true,
	grid style=dashed,
] 
\addplot table[x=N, y=V] from \BW;
\addplot table[x=N, y=S] from \BW;
\addplot table[x=N, y=U] from \BW;
\legend{Volume, Surface, Update}
\end{axis}
\end{tikzpicture}
}
\caption{Profiled GFLOPS/s and bandwidth (GB/s) for volume, surface, and update kernels.  Results are presented for an Nvidia GTX 980 GPU, on a mesh of 9918 elements.  }
\label{fig:gflopsBW}
\end{figure}

Finally, we compute the GFLOPS per second and bandwidth (GB/s) achieved by each of the kernels for elastic wave propagation in our implementation.  The results are shown in Figure~\ref{fig:gflopsBW}, and are qualitatively similar to the results reported in \cite{modave2016gpu} for the volume and surface kernels for elasticity.  The GFLOPS/s and bandwidth for the update kernel fall between the reported values for the volume and surface kernel.  The run-time of the update kernel for elastic wave propagation (in which the weight-adjusted projection matrix is applied) constitutes between $40\%$ and $50\%$ of the total run-time for $N=1,\ldots,7$.  In comparison, the update kernel for piecewise constant material properties takes roughly $35\%$ of the run-time at $N=1$ and $10\%$ of the total run-time at $N=7$, due to the fact that no additional matrix multiplications are necessary in the update kernel if material properties are assumed to be constant within an element.  

}

\section{Conclusions}
\label{sec:conclusions}

This work presents a weight-adjusted discontinuous Galerkin (WADG) method for the linear elastic wave equations with arbitrary heterogeneous media.  The method is energy stable and high order accurate for arbitrary stiffness matrices, and a slight modification results in an energy stable method for curvilinear meshes as well.  The penalty numerical fluxes for this formulation are simple to derive and implement, and their lack of dependence on the stiffness matrix allows for a unified treatment of isotropic and anisotropic media.  Numerical examples confirm the accuracy of this method for analytic solutions of the elastic wave equations, as well as its high order accuracy with respect to a reference solution for smoothly varying heterogeneous media.  Results obtained using this method also show good agreement with existing results in the literature for both problems involving both isotropic and anisotropic heterogeneous media.  \reviewerTwo{Finally, we provide computational results demonstrating the performance of the proposed methods on a single GPU.}  

We note that the implementation of this method reduces to the application of the weight-adjusted mass matrix inverse and the evaluation of constant-coefficient terms in the DG formulation.  The cost of the latter step can be reduced (especially at high orders of approximation) by using fast methods based on Bernstein-Bezier bases for the application of derivative and lift matrices for constant-coefficient terms \cite{chan2015bbdg}.  \reviewerOne{Future work will also involve a more careful study of discretization parameters (such as the penalty parameters and the points per wavelength required for accuracy), as well as the application of the proposed method to more realistic geophysical settings.  }



\section{Acknowledgments}

The author gratefully thanks Thomas Hagstrom, Tim Warburton, Axel Modave, Ruichao Ye, and Mario Bencomo for helpful and informative discussions.  {The author is supported by the National Science Foundation under awards DMS-1719818 and DMS-1712639.  }

\bibliographystyle{unsrt}
\bibliography{dgpenalty}

\end{document}